%% file: arXiv_version.tex
\documentclass[12pt]{article}

\usepackage[T1]{fontenc}
\usepackage[utf8]{inputenc}
\usepackage{newtxtext}
\usepackage{newtxmath}

\usepackage{amsthm}
\usepackage[margin=1in]{geometry}
\usepackage{microtype}

\usepackage{mathtools,amsbsy}
\usepackage{graphicx}
\usepackage[numbers,sort&compress]{natbib}
\usepackage{algorithm}
\usepackage{algpseudocode}
\usepackage{xcolor}
\usepackage{tikz}
\usepackage{booktabs}
\usepackage{multirow}
\usepackage{comment}
\usepackage{verbatim}
\usepackage[colorlinks,citecolor=blue,urlcolor=blue]{hyperref}
\usepackage{colortbl}    
\usepackage{tikz}
\usetikzlibrary{positioning, arrows.meta, calc}

\theoremstyle{plain}
\newtheorem{theorem}{Theorem}[section]
\newtheorem{lemma}[theorem]{Lemma}
\newtheorem{proposition}[theorem]{Proposition}
\newtheorem{corollary}[theorem]{Corollary}
\theoremstyle{definition}

\theoremstyle{remark}
\newtheorem{remark}[theorem]{Remark}

\newcommand{\df}{\mathrm{d}}
\newcommand{\X}{\mathcal{X}}

\newcommand{\B}{\mathcal{B}}

\newcommand{\PiL}{L^2(\Pi)}
\newcommand{\ip}[2]{\langle #1,#2\rangle_\Pi}
\newcommand{\normPi}[1]{\left\|#1\right\|_{\Pi}}

\AtBeginDocument{%
  \setlength{\abovedisplayskip}{4pt plus 2pt minus 2pt}%
  \setlength{\belowdisplayskip}{4pt plus 2pt minus 2pt}%
  \setlength{\abovedisplayshortskip}{0pt plus 2pt}%
  \setlength{\belowdisplayshortskip}{2pt plus 2pt minus 2pt}%
}
\title{Solidarity of Spectral Gaps for Component-Wise Markov Chains}
\author{Youngwoo Kwon, Galin Jones, Qian Qin}
\author{Youngwoo Kwon$^\dagger$, Galin Jones$^\dagger$, Qian Qin$^\dagger$\\
School of Statistics, University of Minnesota$^\dagger$
}
\date{October 2025}

\begin{document}

\maketitle

\begin{abstract}
Deterministic-scan and random-scan component-wise Markov chain Monte Carlo algorithms, such as Gibbs samplers and conditional Metropolis-Hastings, are popular approaches for sampling from multivariate distributions. A long-standing open question is to determine the conditions under which these algorithms have similar convergence rates.  A block-wise contraction condition for the component-wise updates is used to establish a solidarity principle for the $L^2$ spectral gaps of the associated Markov chains.  Specifically, under this condition, the spectral gaps of the random-scan and deterministic-scan versions of the Gibbs and component-wise chains are either simultaneously positive or simultaneously zero. Moreover, the spectral gaps differ by at most polynomial factors in the number of blocks.

As an application of the general results, a deterministic-scan conditional Metropolis-adjusted Langevin algorithm (MALA) for multivariate Gaussian targets is studied. The block-wise contraction condition is combined with known spectral gap bounds for the random-scan Gibbs sampler to obtain a spectral gap bound that is polynomial in dimension. The result is used to clarify how the convergence rate of the conditional MALA depends on the precision matrix of the Gaussian target and the step sizes of the block-wise MALA updates.
\end{abstract}

\section{Introduction}\label{sec:intro}

Markov chain Monte Carlo (MCMC) methods are fundamental tools for sampling from multivariate probability distributions \cite{broo:etal:2011}. A common strategy for constructing an MCMC algorithm is through so-called component-wise updates. When the component-wise updates are drawn from the full-conditional distributions, the resulting chain is a Gibbs sampler \cite{GelfandSmith1990, Geman1984}. When a Metropolis-Hastings update \cite{hastings1970monte, metropolis1953equation} is required to sample from at least one of the full conditionals, the resulting algorithm has been called conditional Metropolis-Hastings or Metropolis-Hastings-within-Gibbs. More generally, one can use any reversible Markov chain targeting the full-conditional distributions to construct component-wise MCMC algorithms. This is the setting of interest here.  

If $\Pi$ is the invariant distribution for the Markov chain, then establishing the rate of convergence to $\Pi$ is required to ensure a reliable simulation experiment. In particular, if the Markov chain is geometrically ergodic, then a central limit theorem holds under moment conditions on the functions of interest \cite{chan:geye:1994, doss:etal:quantile:2014, jones2004markov, robe:etal:assess:2020} and methods for computing asymptotically valid Monte Carlo standard errors are available \cite{jone:hara:caff:neat:2006, jone:hobe:2001, song2024weightedshapeconstrainedestimationautocovariance, vats:etal:moa:2019}. Geometric ergodicity is often established via a strictly positive spectral gap. Let $K$ denote the Markov operator on $L^2(\Pi)$ and $\|\cdot\|$ be an operator norm on $L^2(\Pi)$. For $\Pi$ as the averaging operator $f\mapsto \Pi(f)$, define the $L^2(\Pi)$ spectral gap of $K$ as $\eta_K := 1-\|K-\Pi\|$. If $\eta_K>0$, the chain is geometrically ergodic in $L^2(\Pi)$ and $\|K^n-\Pi\|$ decays at a geometric rate. 

Component-wise MCMC algorithms can be constructed by combining local updates via mixing or composition to obtain either a random-scan or a deterministic-scan version targeting $\Pi$ \cite{johnson2013component}.  Unlike the convergence behavior of random-scan schemes, which can be analyzed using standard self-adjoint operator theory, deterministic-scan Markov chains are significantly more challenging to analyze because they induce non-reversible Markov chains with complicated kernels. As such, the convergence behavior of deterministic scans in general settings has long remained an open problem.

One potentially fruitful approach is to investigate the relative performance of random- and deterministic-scan component-wise algorithms. This has been an area of recent active research for Gibbs samplers. For example, spectral gap solidarity principles have been established for Gibbs samplers, showing that a positive spectral gap for one of the random-scan or deterministic-scan Gibbs samplers implies a positive spectral gap for the other \cite{ChlebickaLatuszynskiMiasojedow2025, GaitondeMossel2024}. This has been extended to blocked and collapsed Gibbs samplers
\cite{mak2026extensionssolidarityprinciplespectral}, clarifying when spectral-gap positivity is inherited and when it can depend on the blocking or collapsing scheme. It is also well known that deterministic-scan can converge faster than their random-scan counterparts. In particular, for Gaussian targets with positive correlations, deterministic-scan Gibbs updates can converge approximately twice as fast as the random-scan Gibbs sampler in slow-mixing regimes \cite{RobertsSahu1997}.  There is also an explicit relationship between the $L^2(\Pi)$ convergence rate of deterministic-scan and random-scan Gibbs samplers in the two-component setting \cite{QinJones2023}.  However, in general, either scheme can converge faster depending on the circumstances, and the random-scan can lie between the fastest and slowest systematic scans \cite{he:etal:scanorder:2016,roberts2015surprising}. 
 
When at least one of the local updates of component-wise chains approximates the full-conditionals with Metropolis-Hastings, random-scan and deterministic-scan schemes have been compared to the corresponding Gibbs sampler in the two-component setting \cite{andrieu2016random, andrieu2018uniform, gao2026weak, qin2025spectral, QinJones2023}, under uniform ergodicity in total variation \cite{JonesRobertsRosenthal2014}, and via conductance-based bounds under conditional mixing assumptions \cite{ascolani2024scalability}. In addition, there are spectral gap bounds for reversible hybrid Gibbs chains \cite{QinJuWang2025}, relating the absolute spectral gap of a hybrid random-scan Gibbs sampler to that of the corresponding Gibbs sampler. Lower bounds on the convergence rate of hybrid random-scan chains are also available \cite{brwo:jone:lower:2025}.

Yet, the convergence rates of more general deterministic-scan component-wise update chains have remained intractable. More specifically, it is not obvious whether the spectral-gap solidarity of Gibbs samplers extends to more general component-wise Markov chains with at least two components, which is the focus of our work. This framework encompasses a wide range of MCMC algorithms, including Gibbs samplers and conditional Metropolis--Hastings updates. In particular, the spectral gap solidarity principle is extended to component-wise chains whose local kernels are perturbations of the Gibbs projections. In addition, we establish quantitative $L^2$ spectral gap bounds for deterministic-scan component-wise chains when there are at least two components.  

\input{spectral_gap_table}

The main results are now described more formally. Let $d$ be the number of component-wise updates and let $[d]$ denote the collection $\{1,\ldots, d\}$. For each $j\in [d]$, let $P_j$ denote the Gibbs update for block $j$ given the remaining coordinates. Equivalently, $P_j$ is the conditional expectation operator onto the $\sigma$-field generated by $x_{-j} := (x_i)_{i\neq j}$. Let $K_j$ denote a corresponding reversible component-wise Markov kernel that also updates only the $j$th block and targets the same conditional distribution as $P_j$; note that the $K_j$ are not limited to being Gibbs or Metropolis-Hastings updates. The random-scan component-wise kernel (RCW) and deterministic-scan component-wise kernel (DCW) are given by, respectively,
\begin{equation*}
     P_{\mathrm{RCW}} = \frac{1}{d} \sum_{j=1}^d K_j, \qquad {\rm and} \qquad P_{\mathrm{DCW}} = K_d \cdots K_1.
\end{equation*}
The random-scan Gibbs (RSG) and deterministic-scan Gibbs (DSG) kernels, $P_{\mathrm{RSG}}$ and $P_{\mathrm{DSG}}$, respectively, are constructed similarly by replacing each $K_j$ with $P_j$. 

A key assumption is that the deviation of each local component-wise update from the Gibbs update satisfies, for some $\lambda_0 \in (0,1)$ and all $j \in [d]$, 
\begin{equation}
\label{eq:bound}
    \|K_j - P_j\| \le \lambda_0 < 1.
\end{equation}
This condition gives a uniform $L^2(\Pi)$ control on how far each local reversible component-wise update $K_j$ can deviate from its Gibbs counterpart $P_j$. Since $P_j$ is the $L^2(\Pi)$ projection associated with the $j$th conditional distribution, the bound \eqref{eq:bound} ensures that each local step remains a nontrivial amount of the corresponding Gibbs contraction. It is the key device that permits the transfer of positive spectral gaps, and hence of geometric ergodicity, across scanning schemes. This bound can be obtained by several standard arguments, for instance by lower-bounding a suitable spectral gap via Dirichlet form techniques or conductance-based Cheeger's inequalities and then translating that gap bound into an $L^2$ contraction bound \cite{QinJuWang2025,kwon2024phase}. Note that the condition \eqref{eq:bound} is weaker than the uniform ergodicity condition often imposed in the component-wise MCMC literature \cite{JonesRobertsRosenthal2014, johnson2013component}. Further analysis and explanation of this condition is provided in Section~\ref{sec:notation}.

Theorem~\ref{thm:RCW_to_DCW} establishes that a positive $L^2$ spectral gap for the random-scan component-wise sampler implies a positive $L^2$ spectral gap for its deterministic-scan counterpart, extending results for the Gibbs setting \cite{ChlebickaLatuszynskiMiasojedow2025, GaitondeMossel2024}.  
The proof is given in Section~\ref{sec:RCW_to_DCW}.

\begin{theorem}\label{thm:RCW_to_DCW}
Assume $\|K_j-P_j\|\le \lambda_0<1$ for all $j \in [d]$. If for some $\eta_{\mathrm{RCW}}\in(0,1]$,
\begin{equation*}
    \|P_{\mathrm{RCW}}-\Pi\|\le 1-\eta_{\mathrm{RCW}},
\end{equation*}  
then
\begin{equation*}
\|P_{\mathrm{DCW}}-\Pi\|
\le
1-\frac{(1-\lambda_0)^2}{8(d+1)(1+\lambda_0)}\,\eta_{\mathrm{RCW}}.
\end{equation*}
\end{theorem}

Theorem~\ref{thm:RCW_to_DCW} combined with existing results relating the $L^2$ spectral gaps of the random-scan Gibbs sampler and the random-scan component-wise sampler \cite{QinJuWang2025} yields the following result.  

\begin{corollary}\label{cor:RSG_to_DCW_2}
Assume $\|K_j-P_j\|\le \lambda_0<1$ for all $j\in[d]$. If for some $\eta_{\mathrm{RSG}}\in(0,1]$,
\begin{equation*}
    \|P_{\mathrm{RSG}}-\Pi\|\le 1-\eta_{\mathrm{RSG}},
\end{equation*}
then
\begin{equation*}
\|P_{\mathrm{DCW}}-\Pi\|
\le
1-\frac{(1-\lambda_0)^3}{8(d+1)(1+\lambda_0)}\,\eta_{\mathrm{RSG}}.
\end{equation*}
\end{corollary}

The reverse comparison is also available.
Theorem~\ref{thm:DCW_to_RCW} quantifies the $L^2$ spectral gap of the random-scan component-wise chain in terms of that of a deterministic-scan component-wise chain, showing that geometric convergence in $L^2$ is preserved under changes in the scanning rule, with the proof following an argument developed in prior work \cite{GaitondeMossel2024}. Specifically, when geometric convergence holds uniformly over all scan orders, the resulting lower bound on the spectral gap is of order $d^{-2}$ up to a factor of $\log^2 d$. When geometric convergence is known only for a single scan order, the corresponding lower bound is of order $d^{-4}$ up to the same factor. Note that for this reverse direction, the resulting bounds exhibit a quadratic dependence on the deterministic-scan spectral gap as in \cite{GaitondeMossel2024}. The detailed argument is given in Section~\ref{sec:DCW_to_RCW}.

\begin{theorem}\label{thm:DCW_to_RCW}

Assume $\|K_j-P_j\|\le \lambda_0<1$ for all $j\in[d]$. For each permutation $\sigma:[d]\to[d]$, let $P_{\mathrm{DCW}}^\sigma:=K_{\sigma(d)}\cdots K_{\sigma(1)}$. Then there exists a universal constant $c_{\mathrm{RCW}}>0$ such that the following hold.

\begin{itemize}
    \item Let $d\ge 2$ and $\eta_{\mathrm{DCW}}\in(0,1]$. If $\|P_{\mathrm{DCW}}^\sigma-\Pi\|\le 1-\eta_{\mathrm{DCW}}$ for every permutation $\sigma:[d]\to[d]$, then
    \begin{align*}
    \|P_{\mathrm{RCW}}-\Pi\| \le 1- c_{\mathrm{RCW}} \, \frac{1-\lambda_0}{1+\lambda_0} \, \frac{\eta_{\mathrm{DCW}}^2}{d^2\log^2 d}.
    \end{align*}

    \item Let $d\ge 8$.  If there exists a permutation $\sigma_0:[d]\to[d]$ such that $\|P_{\mathrm{DCW}}^{\sigma_0}-\Pi\|\le 1-\eta_{\mathrm{DCW}}$ for some $\eta_{\mathrm{DCW}}\in(0,1]$, then
    \begin{align*}
    \|P_{\mathrm{RCW}}-\Pi\| \le 1- c_{\mathrm{RCW}} \, \frac{1-\lambda_0}{1+\lambda_0} \, \frac{\eta_{\mathrm{DCW}}^2}{d^4\log^2 d}.
    \end{align*}
\end{itemize}
\end{theorem}

In Section~\ref{sec:clt}, we establish a central limit theorem for both component-wise chains and bound the difference between their asymptotic variances in terms of the corresponding $L^2$ gaps. While a single step of $P_{\mathrm{RCW}}$ updates one coordinate, a single step of $P_{\mathrm{DCW}}$ performs $d$ coordinate updates. For this reason, we compare the $d$ step random-scan kernel $(P_{\mathrm{RCW}})^d$ with the one-step deterministic-scan kernel $P_{\mathrm{DCW}}$. The proofs of the following are given in Theorems~\ref{thm:clt_1} and~\ref{thm:varDCW_varRCW}.

\begin{theorem}\label{thm:clt}
Suppose $P_{\mathrm{RCW}}$ and $P_{\mathrm{DCW}}$ have positive $L^2$ gaps, $\eta_{\mathrm{RCW}}$ and $\eta_{\mathrm{DCW}}$, respectively. Let $(X_n)_{n\ge 0}$ be the chain with transition kernel $(P_{\mathrm{RCW}})^d$ and let $(Y_n)_{n\ge 0}$ be the chain with transition kernel $P_{\mathrm{DCW}}$. Then, for every $f\in L^2(\Pi)$ with $\Pi f=0$, there exist finite asymptotic variances $\sigma^2_{\mathrm{RCW}^d}(f)$ and $\sigma^2_{\mathrm{DCW}}(f)$ such that, as $n\to\infty$,
\begin{equation*}
    \frac{1}{\sqrt{n}}\sum_{k=0}^{n-1} f(X_k) \xrightarrow{d}
\mathcal N\left(0,\sigma^2_{\mathrm{RCW}^d}(f)\right)\qquad\mathrm{and}\qquad\frac{1}{\sqrt{n}}\sum_{k=0}^{n-1} f(Y_k)
\xrightarrow{d}
\mathcal N\left(0,\sigma^2_{\mathrm{DCW}}(f)\right) .
\end{equation*}
Moreover,
\begin{align*}
        \left|\sigma_{\scriptstyle{\mathrm{DCW}}}^2(f) - \sigma_{\scriptstyle{\mathrm{RCW}^d}}^2(f)\right|
        \le \frac{4}{\eta_{\mathrm{RCW}} \eta_{\mathrm{DCW}}} \,\|f\|_\Pi^2 .
    \end{align*}
\end{theorem}

In Section~\ref{sec:MALA example}, the general solidarity results are used to study deterministic-scan component-wise Metropolis-adjusted Langevin algorithm (MALA) targeting multivariate Gaussian distributions. An explicit bound on the global block-wise contraction factor is derived and combined with the spectral gap of the random-scan Gibbs sampler to obtain a quantitative $L^2(\Pi)$ convergence rate for the deterministic-scan component-wise MALA chain. This result is then specialized to compound-symmetry and autoregressive covariance structures, which clarifies how the rate depends on the target precision matrix and the block-wise step sizes.

\section{Notation and problem setup}\label{sec:notation}

Let $(\X_j,\B_j)$, $j\in[d]$, be Polish measurable spaces that represent the coordinate blocks, and set $\X := \prod_{j=1}^d \X_j$ with product $\sigma$-algebra $\B := \bigotimes_{j=1}^d \B_j$. Write points of $\X$ as $x = (x_1,\dots,x_d)$ with $x_j \in \X_j$ and let $x_{-j} = (x_i)_{i\neq j} \in \X_{-j}$ for the collection of all blocks except the $j$th.

Let $\Pi$ be a probability measure on $(\X,\B)$, which serves as the invariant distribution of Markov kernels. Let $\PiL$ denote the Hilbert space of square-integrable functions with respect to $\Pi$ with inner product $\ip{f}{g}:=\int f g\,\df\Pi$ and norm $\normPi{f}:=\ip{f}{f}^{1/2}$. We write $\Pi(f)=\int f\,\df\Pi$, and use the same symbol $\Pi$ for the orthogonal projection from $\PiL$ onto the subspace of constant functions $(\Pi f)(x):=\Pi(f)$.

Let $K$ be a Markov kernel on $(\X,\B)$ with invariant distribution $\Pi$. The associated Markov operator on $\PiL$ is defined by
\begin{align*}
(Kf)(x) := \int f(y)\,K(x,\df y) ,\quad f\in\PiL.
\end{align*}
Define the $L^2(\Pi)$ operator norm of $K$ by
\begin{align*}
\|K\| := \sup\left\{\normPi{Kf}:\ f\in L^2(\Pi),\ \normPi{f}=1\right\},
\end{align*}
and the $L^2(\Pi)$ spectral gap of $K$ by $\eta_K = 1-\|K-\Pi\|$.

For each $j\in[d]$, denote $\Pi_{-j}$ the marginal distribution of $x_{-j}$ and let $\Pi_j(\df x_j\mid x_{-j})$ be a conditional distribution of $x_j$ given $x_{-j}$. Then,
\begin{align*}
   \Pi(\df x_j,\df x_{-j}) 
   = \Pi_j(\df x_j\mid x_{-j})\,\Pi_{-j}(\df x_{-j}).
\end{align*}
For each $x_{-j}\in\X_{-j}$ write $L^2\left(\Pi(\cdot\mid x_{-j})\right)$ for the space of square-integrable functions on $\X_j$ under $\Pi_j(\df x_j\mid x_{-j})$. For $f\in L^2\left(\Pi(\cdot\mid x_{-j})\right)$, define
\begin{align*}
   \|f\|_{\Pi_{j,x_{-j}}}^2 
   := \int_{\X_j} f(y_j)^2\,\Pi_j(\df y_j\mid x_{-j}),
\end{align*}
and for a bounded linear operator $T$ on this space, denote by $\|T\|_{j,x_{-j}}$ its operator norm with respect to the vector norm $\|\cdot\|_{\Pi_{j,x_{-j}}}$.

Say that $K$ is reversible with respect to $\Pi$ if the detailed balance condition,
\begin{align*}
\Pi(\df x)\,K(x,\df y) = \Pi(\df y)\,K(y,\df x),
\end{align*}
holds. Note that integrating both sides over $x$ yields $\Pi K=\Pi$, so reversibility implies $\Pi$-invariance. Equivalently, $K$ is reversible if $\ip{f}{Kg}=\ip{Kf}{g}$ for all $f,g\in\PiL$, in which case $K$ is self-adjoint on $L^2(\Pi)$.

For each $j\in[d]$ and $x_{-j}\in\X_{-j}$, let $K_{j,x_{-j}}$ be a Markov kernel on $(\X_j,\B_j)$ reversible with respect to $\Pi_j(\cdot \mid x_{-j})$. We assume that for every $A_j\in\B_j$ the map $(x_j,x_{-j})\mapsto K_{j,x_{-j}}(x_j,A_j)$ is measurable. The associated component-wise update kernel $K_j$ on $\X$ is then defined by
\begin{align*}
(K_j f)(x_j,x_{-j})
= \int_{\X_j} f(y_j,x_{-j})\,K_{j,x_{-j}}(x_j,\df y_j)
\end{align*}
for all $f\in\PiL$. $K_j$ can also be viewed as the Markov kernel on $(\X,\B)$ of the form
$K_j\left((x_j,x_{-j}), \df(y_j,y_{-j})\right) \ = K_{j,x_{-j}}(x_j,\df y_j)\,\delta_{x_{-j}}(\df y_{-j})$. It follows that $K_j$ is reversible with respect to $\Pi$.

If $K_{j,x_{-j}}$ is taken to be the Gibbs conditional kernel, that is, if 
\begin{align*}
K_{j,x_{-j}}(x_j,A_j)=\Pi_j(A_j\mid x_{-j})
\end{align*}
for all $A_j\in\B_j$, the local kernel $K_{j,x_{-j}}$ reduces to the Gibbs conditional kernel, denoted by $P_{j,x_{-j}}$. The corresponding Gibbs update $P_j$ on $\X$ is defined by
\begin{align*}
(P_j f)(x_j,x_{-j})
:= \int_{\X_j} f(y_j,x_{-j})\,P_{j,x_{-j}}(x_j,\df y_j)
\end{align*}
for all $f\in\PiL$. Equivalently, $P_j\left((x_j,x_{-j}), \df(y_j,y_{-j})\right)
= \Pi_j(\df y_j\mid x_{-j})\,\delta_{x_{-j}}(\df y_{-j}).$

Then the Gibbs updates $P_j$, and the corresponding component-wise kernels $K_j$ satisfy the following simple structural properties. The proof is given in Appendix~\ref{Appendix_A}.

\begin{lemma}\label{lem:Pj_Kj_properties}
The Gibbs kernel $P_j$ is the orthogonal projection from $L^2(\Pi)$ onto the subspace of $x_{-j}$-measurable functions. Moreover, for a component-wise kernel $K_j$, $K_j P_j = P_j K_j = P_j$.
\end{lemma}

Random-scan and deterministic-scan component-wise kernels are built from the updates $(K_j)_{j\in[d]}$ using mixing and composition, respectively. The random-scan component-wise kernel updates a single coordinate chosen uniformly at random and is given by
\begin{align*}
(P_{\mathrm{RCW}} f)(x) = \frac{1}{d} \sum_{j=1}^d (K_j f)(x),
\end{align*}
while the deterministic-scan component-wise kernel performs a full sweep through the coordinates in the fixed order $1,\dots, d$,
\begin{align*}
(P_{\mathrm{DCW}} f)(x) = (K_d \cdots K_1 f)(x),
\end{align*}
for any $f\in L^2(\Pi)$ and $x\in\X$.

Similarly, the corresponding random-scan and deterministic-scan Gibbs kernels built from $(P_j)_{j\in[d]}$ are defined by
\begin{align*}
(P_{\mathrm{RSG}} f)(x) & = \frac{1}{d} \sum_{j=1}^d (P_j f)(x), \\
(P_{\mathrm{DSG}} f)(x) & = (P_d \cdots P_1 f)(x),
\end{align*}
for $f\in L^2(\Pi)$ and $x\in\X$.

\begin{remark}\label{rem:weights_independent} 
For $w=(w_1,\dots,w_d)$ with $w_j>0$ and $\sum_{j=1}^d w_j=1$, define the weighted random-scan Gibbs kernel by $P_{\mathrm{RSG},w} = \sum_{j=1}^d w_j P_j$. If the equal-weight random-scan Gibbs kernel $P_{\mathrm{RSG}}$ has a positive $L^2(\Pi)$ spectral gap, then so does every weighted random-scan Gibbs kernel $P_{\mathrm{RSG},w}$ with strictly positive scan probabilities \cite[Proposition~2]{JonesRobertsRosenthal2014}. A similar argument extends to general component-wise kernels $K_j$. See Appendix~\ref{Appendix_A} for details.
\end{remark}

Let $\lambda_j\in(0,1)$ be a local block-wise contraction constant for $K_j$, that is, $\|K_j-P_j\| \le \lambda_j$. 
Define the global block-wise contraction constant by $\lambda_0 = \max_{j\in[d]} \lambda_j$, so that
\begin{align*}
    \|K_j - P_j\| \le \lambda_0 < 1 \quad \text{for all } j \in [d].
\end{align*}
This global block-wise contraction condition follows from several sufficient conditions. For example, Lemma~\ref{lem:block_contraction_equiv} shows that a pointwise contraction implies the global block-wise contraction. Detailed proofs are given in Appendix~\ref{Appendix_A}.

\begin{lemma}\label{lem:block_contraction_equiv}
Let $K_j$ be a component-wise kernel. Suppose there exists $\lambda_j\in(0,1)$ such that $\|K_{j,x_{-j}}-P_{j,x_{-j}}\|_{j,x_{-j}}\le \lambda_j$ for $\Pi_{-j}$-almost every $x_{-j}\in\X_{-j}$. Then $\|K_j-P_j\|\le \lambda_j$. Moreover, $\|K_j^{n}-P_j\|\to 0$ as $n\to\infty$.
\end{lemma}

\begin{remark}\label{rem:necessity_block_contraction}
In general, it is difficult to relax the global block-wise contraction condition. A natural candidate is to assume only pointwise geometric ergodicity of the conditional component-wise updates $K_{j,x_{-j}}$ for each fixed $x_{-j}$, which is weaker than the pointwise contraction in Lemma~\ref{lem:block_contraction_equiv}. However, the hybrid sampler example for two independent standard normal coordinates shows that this weaker assumption is not sufficient \cite[Proposition~3.1]{RobertsRosenthal1997}.

In this example, each $K_j$ is a random-walk Metropolis kernel targeting a standard normal distribution, so for every fixed $x_{-j}$, the conditional kernel $K_{j,x_{-j}}$ is geometrically ergodic. While the corresponding random-scan Gibbs sampler has a positive $L^2(\Pi)$ spectral gap, the random-scan component-wise kernel $P_{\mathrm{RCW}}=(K_1+K_2)/2$ is $\Pi$-reversible but not geometrically ergodic, and hence cannot have a positive $L^2(\Pi)$ spectral gap \cite{RobertsRosenthal1997}.
One can also verify that the deterministic-scan component-wise updates $K_1K_2$ and $K_2K_1$ have zero spectral gap in $L^2(\Pi)$.

Thus, pointwise geometric ergodicity of the conditional kernels does not provide a suitable replacement for the global block-wise contraction condition in Lemma~\ref{lem:block_contraction_equiv}. We further show that $\lambda_0 = 1$ in this example. This is consistent with our main results, which imply that if $\lambda_0 < 1$, then both the random-scan and deterministic-scan component-wise chains would necessarily have a positive $L^2(\Pi)$ spectral gap. See Appendix~\ref{Appendix_A} for details.
\end{remark}

\begin{remark}\label{rem:L2contraction_to_uniform_ergodicity}
One simple but more restrictive situation where the global block-wise contraction condition holds is when the conditional update kernel satisfies a global minorization condition. Fix $j\in[d]$ and suppose there exists $\varsigma_j\in(0,1]$ such that, for $\Pi_{-j}$-almost every $x_{-j}\in\X_{-j}$, there exists a probability measure $\upsilon_{j,x_{-j}}$ on $\X_j$ satisfying
\begin{align*}
K_{j,x_{-j}}(x_j,A_j) \ge \varsigma_j\,\upsilon_{j,x_{-j}}(A_j)
\end{align*}
for all $x_j\in\X_j$ and measurable $A_j\subset\X_j$. This Doeblin condition implies that $K_{j,x_{-j}}$ is uniformly ergodic with rate at most $1-\varsigma_j$. Combined with the reversibility of $K_{j,x_{-j}}$ with respect to $\Pi_j(\cdot\mid x_{-j})$, this yields
\begin{align*}
\left\|K_{j,x_{-j}}-P_{j,x_{-j}}\right\|_{j,x_{-j}} 
\le 1-\varsigma_j
\end{align*}
\cite{RobertsRosenthal1997}. Therefore, one may take $\lambda_j=1-\varsigma_j\in(0,1)$ in Lemma~\ref{lem:block_contraction_equiv}.

Alternatively, if one can verify a geometric drift and a small-set minorization condition for $K_{j,x_{-j}}$ and, in addition, $K_{j,x_{-j}}$ is positive semidefinite on $L^2\left(\Pi_j(\cdot\mid x_{-j})\right)$ for $\Pi_{-j}$-almost every $x_{-j}$, then these conditions also imply a positive $L^2$ spectral gap \cite[Corollary 1]{taghvaei2021lyapunov}.
\end{remark}

The block-wise contraction condition can also be verified using conductance-based arguments together with an isoperimetric inequality. See Section~\ref{sec:MALA example} and Appendix~\ref{Appendix_D} for details.

\section{Between the random-scan component-wise chain and the deterministic-scan component-wise chain} \label{sec:between-CMH}

In Section~\ref{sec:between-CMH}, we analyze the quantitative relationship between the spectral gaps of random-scan and deterministic-scan component-wise chains. This allows us to transfer properties of the random-scan setting to the deterministic-scan setting.

\subsection{From random-scan component-wise chain to deterministic-scan component-wise chain}\label{sec:RCW_to_DCW}

In Section~\ref{sec:RCW_to_DCW}, we show that a positive spectral gap of the random-scan chain implies a positive spectral gap of the deterministic-scan chain, using a component-wise residual analysis. Suppose that $\|P_{\mathrm{RCW}}-\Pi\|\le 1-\eta_{\mathrm{RCW}}$ for some $\eta_{\mathrm{RCW}}\in(0,1]$, and assume the local contraction condition $\|K_j-P_j\|\le \lambda_j<1$ for each $j\in[d]$. Fix $j$ and decompose $L^2(\Pi)$ as $L^2(\Pi)=\mathcal{M}_j\oplus \mathcal{N}_j$, where $\mathcal{M}_j=\operatorname{ran}(P_j)$ is the subspace of $x_{-j}$-measurable functions and $\mathcal{N}_j=\ker(P_j)$. 

Since $K_j$ leaves $x_{-j}$ unchanged, it acts as the identity on $\mathcal{M}_j$, hence 1 is the only eigenvalue of $K_j$ on $\mathcal{M}_j$. Moreover, the bound $\|K_j-P_j\|\le \lambda_j$ implies that the operator $K_j$ restricted to $\mathcal{N}_j$ has operator norm at most $\lambda_j$. Since $K_j$ is self-adjoint, it follows that all eigenvalues of $K_j$ other than the trivial eigenvalue $1$ lie in $[-\lambda_j,\lambda_j]$. Hence, $\mathrm{Spec}(K_j)\subset[-\lambda_j,\lambda_j] \cup \{1\}$, where $\mathrm{Spec}$ denotes the spectrum of the given operator. Further, since $K_j$ is self-adjoint, inequalities between bounded real Borel functions that hold on $t\in[-\lambda_j,1]$ can be applied to $K_j$ as operator inequalities~\cite{RudinFA}. We apply this principle to control $(I-K_j)^2$, which leads to Lemmas~\ref{lem:residual_upperbound} and \ref{lem:residual_lowerbound}.

\begin{lemma}\label{lem:residual_upperbound}
     Assume $\|K_j-P_j\|\le \lambda_0<1$ for all $j\in[d]$. Define $u_0(f) = f - \Pi f$ and $u_j(f) = K_j u_{j-1}(f)$ for $f\in L^2(\Pi)$, $j\in [d]$. Then,
    \begin{align*}
        \|K_j f - f\|_\Pi^2 \le 4 j \frac{1+\lambda_0}{1-\lambda_0} \left( \|f - \Pi f \|_\Pi^2 - \|u_j(f)\|_\Pi^2\right).
    \end{align*}
\end{lemma}

\begin{lemma}\label{lem:residual_lowerbound}
    Assume $\|K_j-P_j\|\le \lambda_j<1$ for $j\in[d]$. Let $f\in L^2(\Pi)$. Then,
    \begin{align*}
        \|K_j f - f\|_\Pi^2
        \ge - (1-\lambda_j)\,\langle f, K_j f - f \rangle_\Pi.  
    \end{align*}
\end{lemma}

The proofs of these bounds are given in Appendix~\ref{Appendix_B}. Combining Lemmas~\ref{lem:residual_upperbound} and~\ref{lem:residual_lowerbound}, we obtain Theorem~\ref{thm:RCW_to_DCW}.

\begin{proof}[Proof of Theorem~\ref{thm:RCW_to_DCW}]
    Let $f\in L^2(\Pi)$. Then,
    \begin{align*}
        \|P_{\mathrm{RCW}} f - \Pi f\|_\Pi^2 
        &= \biggl\|f - \Pi f + d^{-1}\sum_{j \in [d]} (K_j f - f) \biggr\|_\Pi^2 \\
        & = \left\| f - \Pi f \right\|_\Pi^2 + \sum_{j \in [d]} \frac{2}{d} \left\langle f, K_j f - f \right\rangle_\Pi + \biggl\|\sum_{j \in [d]} \frac{1}{d} (K_j f - f) \biggr\|_\Pi^2\\
        & \ge \left\|f - \Pi f \right\|_\Pi^2 + \sum_{j \in [d]} \frac{2}{d} \left\langle f , K_j f - f \right\rangle_\Pi.
    \end{align*}
    Using Lemmas~\ref{lem:residual_upperbound} and \ref{lem:residual_lowerbound}, we can lower bound the second term as follows.
    \begin{align*}
        \|P_{\mathrm{RCW}} f - \Pi f\|_\Pi^2  & \ge \left\|f - \Pi f \right\|_\Pi^2 + \frac{2}{d}\sum_{j \in [d]}  \left\langle f , K_j f - f \right\rangle_\Pi\\
        & \geq \|f - \Pi f \|_\Pi^2 - \sum_{j \in [d]} \frac{2}{d(1-\lambda_j)} \left \| K_j f - f \right \|_\Pi^2 \\
        & \geq \|f - \Pi f \|_\Pi^2  - \sum_{j \in [d]} \frac{8j}{d(1-\lambda_j)} \left(\frac{1+\lambda_0}{1 - \lambda_0}\right) \left(\|f - \Pi f \|_\Pi^2 -\| u_j(f)\|_\Pi^2 \right).
    \end{align*}    
    As $\lambda_j \le \lambda_0 < 1$, we have
    \begin{align*}
         \|P_{\mathrm{RCW}} f - \Pi f\|_\Pi^2  &  \geq \left\|f - \Pi f \right\|_\Pi^2 - \frac{8}{d(1-\lambda_0)} \left(\frac{1+\lambda_0}{1 - \lambda_0}\right) \sum_{j \in [d]} j \left(\|f - \Pi f \|_\Pi^2 -\| u_j(f)\|_\Pi^2 \right)\\
         & \geq \|f - \Pi f\|_\Pi^2 - \frac{4(d+1) (1+\lambda_0)}{(1-\lambda_0)^2} \left(\|f - \Pi f \|_\Pi^2 -\|P_{\mathrm{DCW}}f - \Pi f\|_\Pi^2 \right),
    \end{align*}
     where the second inequality follows from $\|f - \Pi f\|_\Pi = \|u_0(f)\|_\Pi \geq \|u_1(f)\|_\Pi  \geq \dots \geq \|u_d(f)\|_\Pi = \|P_{\mathrm{DCW}}f - \Pi f \|_\Pi$. 
     
     As $\|P_{\mathrm{RCW}} - \Pi\| \le 1-\eta_{\mathrm{RCW}}$, we have $\| P_{\mathrm{RCW}} f - \Pi f\|_\Pi^2 \le (1-\eta_{\mathrm{RCW}})^2 \|f\|_\Pi^2 \le (1-\eta_{\mathrm{RCW}}) \|f\|_\Pi^2$. Therefore,
     \begin{align*}
         1-\eta_{\mathrm{RCW}} \ge \frac{\|f - \Pi f \|_\Pi^2}{\|f\|_\Pi^2} - \frac{4(d+1) (1+\lambda_0)}{(1-\lambda_0)^2} \left(\frac{\|f - \Pi f \|_\Pi^2}{\|f\|_\Pi^2} -\frac{\|P_{\mathrm{DCW}}f - \Pi f\|_\Pi^2}{\|f\|_\Pi^2} \right).
     \end{align*}
     Then, the square of the operator norm $\|P_{\mathrm{DCW}} - \Pi\|$ is bounded above by
     \begin{align*}
         & \|P_{\mathrm{DCW}} -\Pi \|^2 \\
         &\le \sup_{f \ne 0, \, f \in \PiL} \left(1 - \frac{(1-\lambda_0)^2}{4(d+1)(1+\lambda_0)} \right) \frac{\|f - \Pi f \|_\Pi^2}{\|f\|_\Pi^2}  + \frac{(1-\lambda_0)^2}{4(d+1)(1+\lambda_0)} (1 - \eta_{\mathrm{RCW}})\\
         &\le 1 - \frac{(1-\lambda_0)^2}{4(d+1)(1+\lambda_0)}\eta_{\mathrm{RCW}},
     \end{align*}
     and the inequality $\sqrt{1-x} \le 1- 0.5x$ for $x \in (0,1]$ yields a desired result
     \begin{align*}
         \|P_{\mathrm{DCW}} -\Pi \| \le 1 - \frac{(1-\lambda_0)^2}{8(d+1)(1+\lambda_0)}\eta_{\mathrm{RCW}}.
     \end{align*}
\end{proof}

Now we relate the spectral gap of the random-scan component-wise chain to that of the random-scan Gibbs chain. Under the condition $\|K_{j,x_{-j}}-P_{j,x_{-j}}\|_{j,x_{-j}}\le \lambda_0<1$ for $\Pi_{-j}$-almost every $x_{-j}\in\X_{-j}$, the spectral gap of the random-scan component-wise chain can be bounded in terms of the spectral gap of the random-scan Gibbs chain \cite[Corollary~4]{QinJuWang2025}. The following lemma establishes a similar result under the global block-wise contraction assumption.
\begin{lemma}\label{lem:RCW_to_RSG}
    For each $j \in [d]$, assume that $\|K_j - P_j\| \leq \lambda_0 < 1$.
    Then
    \begin{align*}
        (1-\lambda_0) (1 - \|P_{\mathrm{RSG}} - \Pi\|) 
        \leq 1 -\|P_{\mathrm{RCW}} - \Pi\| 
        \leq (1+\lambda_0) (1 - \|P_{\mathrm{RSG}} - \Pi\|) .
    \end{align*}
\end{lemma}
The proof of the lemma is provided in Appendix~\ref{Appendix_B}. Together with Theorem~\ref{thm:RCW_to_DCW}, this yields Corollary~\ref{cor:RSG_to_DCW_2}.

\begin{remark} 
It is known that positive spectral gaps satisfy a solidarity property between the random-scan Gibbs sampler and the deterministic-scan Gibbs sampler \cite{ChlebickaLatuszynskiMiasojedow2025}. Hence, under the block-wise contraction condition, a positive spectral gap for the deterministic-scan Gibbs sampler also implies a positive spectral gap for both the random-scan and deterministic-scan component-wise samplers. 
\end{remark}

We further give a direct route from the random-scan Gibbs chain to the deterministic-scan component-wise chain that does not pass through the random-scan component-wise chain. This argument yields a fully explicit bound in terms of $\eta_{\mathrm{RSG}}$, $d$, and $\lambda_0$, and it may provide a sharper bound than Corollary~\ref{cor:RSG_to_DCW_2} in limited regimes. We provide the detailed argument in Appendix~\ref{Appendix_C}.

\subsection{From deterministic-scan component-wise chain to random-scan component-wise chain} \label{sec:DCW_to_RCW}

In Section~\ref{sec:DCW_to_RCW}, we prove Theorem~\ref{thm:DCW_to_RCW} by deriving bounds on products of random-scan Markov operators, using arguments adapted from existing work \cite{GaitondeMossel2024}. Our strategy is to bound the operator norm of the $L$th power of the random-scan component-wise kernel $(P_{\mathrm{RCW}})^L$ for $L \ge d$. We decompose this expansion into two parts. The first part corresponds to update sequences that include all indices in $[d]$ and thus behave essentially like a deterministic-scan. The second part collects the remaining terms. To facilitate this, we establish a lemma showing that the spectral gap bound for a deterministic-scan extends to arbitrary longer products containing a full sweep of the deterministic-scan. Recall that we defined the deterministic-scan component-wise sampler for permutation $\sigma: [d] \to [d]$ as $P_{\mathrm{DCW}}^\sigma = K_{\sigma(d)} \dots K_{\sigma(1)}$.

\begin{lemma}\label{lem:DCW_larger_product}
    Assume that $\|K_j-P_j\|\le \lambda_0<1$ for all $j\in[d]$, and $\|P_{\mathrm{DCW}}^\sigma-\Pi\| \le 1-\eta_{\mathrm{DCW}}$
    for some $\eta_{\mathrm{DCW}} \in (0,1]$. Let $L\ge d$, and let $\theta:[L]\to[d]$ be such that $(\theta(L),\dots,\theta(1))$ contains $(\sigma(d),\dots,\sigma(1))$ as a subsequence. Then
    \begin{align*}
    \|K_{\theta(L)} \cdots K_{\theta(1)} - \Pi \|
    \le 1 - \frac{1-\lambda_0}{4(L-d+1)(1+\lambda_0)}\,\eta_{\mathrm{DCW}}^2.
    \end{align*}
\end{lemma}
The proof of Lemma~\ref{lem:DCW_larger_product} is given in Appendix~\ref{Appendix_B}. We are now ready to prove Theorem~\ref{thm:DCW_to_RCW}.

\begin{proof}[Proof of Theorem~\ref{thm:DCW_to_RCW}]

Recall that for all $j\in [d]$ the operator $K_j$ is self-adjoint, so $P_{\mathrm{RCW}} - \Pi$ is also self-adjoint on $L^2(\Pi)$. Then, for every integer $L\ge d$,
\begin{align*}
    \|P_{\mathrm{RCW}} - \Pi\|^L
    = \|(P_{\mathrm{RCW}})^L - \Pi \|
    &= \left\|\left(\frac{1}{d}\sum_{i=1}^d K_i\right)^L - \Pi\right\|\\
    &= \left\|\frac{1}{d^L}\sum_{\theta(1),\dots,\theta(L)\in[d]} \left(K_{\theta(L)}\cdots K_{\theta(1)} - \Pi\right) \right\|.
\end{align*}
We bound the norm of this average for a suitable choice of $L$ in each case.

\smallskip\noindent\textbf{Case 1 (all scans have a gap).}
Assume that for every permutation $\sigma$ of $[d]$, it holds that
\begin{align*}
  \|K_{\sigma(d)}\cdots K_{\sigma(1)}-\Pi\|\le 1-\eta_{\mathrm{DCW}}.
\end{align*}
Let $L=\lceil 2d\log d\rceil$ and define
\begin{align*}
  \mathbb A:=\left\{(\theta(L),\ldots,\theta(1))\in[d]^L \mid \{\theta(1),\ldots,\theta(L)\}=[d]\right\}.
\end{align*}
So $\mathbb A$ is the event that $(\theta(L), \cdots \theta(1))$ contains $[d]$ in some order for at least one permutation.
Now let $\Pr(\mathbb A)$ denote the probability that $(\theta(L),\dots,\theta(1))\in\mathbb A$ when $\theta(1),\dots,\theta(L)$ are chosen independently and uniformly from $[d]$.
So, $\Pr (\mathbb A) = |\mathbb A| / d^L$. Then by a standard coupon–collector bound~\cite{levin2017markov},
\begin{align*}
  \Pr(\mathbb A^c)=1-\frac{|\mathbb A|}{d^L}\le d\left(1-\frac{1}{d}\right)^L\le d\exp\left(-\frac{L}{d}\right)\le d\exp(-2\log d)=\frac{1}{d}.
\end{align*}
Fix $(\theta(L),\ldots,\theta(1))\in \mathbb A$. Let $\sigma$ be the permutation given by the order in which new indices in $[d]$ appear when revealing $(\theta(1),\ldots,\theta(L))$. Define
\begin{align*}
  k_j := \min\{l\in[L]:\ \theta(l)=\sigma(j)\}.
\end{align*}
Since $\sigma$ is the permutation determined by the order in which new indices from $[d]$ appear, we have $  1 = k_1 < \cdots < k_d \le L$, and therefore $(\theta(L),\ldots,\theta(1))$ contains $(\sigma(d),\dots,\sigma(1))$ as a subsequence. Hence Lemma~\ref{lem:DCW_larger_product} applies to $K_{\theta(L)}\cdots K_{\theta(1)}$, yielding
\begin{align*}
  \|K_{\theta(L)}\cdots K_{\theta(1)}-\Pi\|
  &\le 1-\frac{1-\lambda_0}{4(L-d+1)(1+\lambda_0)}\,\eta_{\mathrm{DCW}}^2 \\
  &\le 1-\frac{1-\lambda_0}{1+\lambda_0}\; \frac{\eta_{\mathrm{DCW}}^2}{8d\log d},
\end{align*}
where the last inequality uses $L-d+1\le 2d\log d$.

For $(\theta(L),\ldots,\theta(1))\notin \mathbb A$ we use trivial bound $\|K_{\theta(L)}\cdots K_{\theta(1)}-\Pi\|\le 1$. Thus,
\begin{align*}
  \|P_{\mathrm{RCW}}-\Pi\|^L
  &\le \left(1-\frac{1-\lambda_0}{1+\lambda_0}\cdot\frac{\eta_{\mathrm{DCW}}^2}{8d\log d}\right)\Pr(\mathbb A)+\Pr(\mathbb  A^c)\\
  &=1-\frac{1-\lambda_0}{1+\lambda_0}\cdot\frac{\eta_{\mathrm{DCW}}^2}{8d\log d}\Pr(\mathbb A)\\
  & \le 1-\frac{d-1}{d}\,\frac{1-\lambda_0}{1+\lambda_0}\cdot\frac{\eta_{\mathrm{DCW}}^2}{8d\log d}.
\end{align*}
Using $(1-x)^{1/L}\le 1-x/L$ for $x\in(0,1]$,
\begin{align*}
  \|P_{\mathrm{RCW}}-\Pi\|
  \le 1-\frac{d-1}{d}\,\frac{1-\lambda_0}{1+\lambda_0}\cdot\frac{\eta_{\mathrm{DCW}}^2}{(8d\log d)\lceil 2d\log d\rceil}.
\end{align*}
Since $d \ge 2$, we have $1/2 \le (d-1) / d $ and $\lceil 2d \log d \rceil \le 2d \log d + 1 \le 4d \log d$. Then for any $c_{\mathrm{RCW}} \in (0, 1/64)$,
\begin{align*}
   c_{\mathrm{RCW}} \frac{1}{d^2\log^2 d} \le \frac{d-1}{32d^3\log^2 d} \le \frac{d-1}{8d^2 \log d\lceil 2d\log d\rceil}, 
\end{align*} 
which gives a desired inequality.

\smallskip\noindent\textbf{Case 2 (one scan has a gap).}
Assume there exists a permutation $\sigma_0$ such that
\begin{align*}
  \|K_{\sigma_0(d)}\cdots K_{\sigma_0(1)} -\Pi\|
  \le 1-\eta_{\mathrm{DCW}}.
\end{align*}
Let $L=\lceil d^2\log d\rceil$ and define
\begin{align*}
  \mathbb B := \{(\theta(L),\ldots,\theta(1))\in[d]^L:\ (\theta(L),\ldots,\theta(1)) \text{ contains } \sigma_0 \text{ as an ordered subsequence}\}.
\end{align*}
Assume that $\theta(1),\dots,\theta(L)$ are independent and uniformly distributed on $[d]$. Define binary random variables $(C_i)_{i=1}^L$ as follows. Let $C_1=1$ if $\theta(1)=\sigma_0(1)$, and let $C_1=0$ otherwise. For $j=1,\dots,L-1$, let $C_{j+1}=1$ if
\begin{align*}
\theta(j+1)=\sigma_0\left[\left(\sum_{i=1}^j C_i \bmod d \right)+1\right],
\end{align*}
and let $C_{j+1}=0$ otherwise. Before $\sum_{i=1}^j C_i$ reaches $d$, this recursion simply records whether $\theta(j+1)$ matches the next required entry of $\sigma_0$. Hence the event $\mathbb B$ is exactly
\begin{align*}
\mathbb B=\left\{\sum_{i=1}^L C_i\ge d\right\}.
\end{align*}
Moreover, since $\theta(j+1)$ is independent of $(\theta(1),\dots,\theta(j))$ and is uniformly distributed on $[d]$, we have $C_{j+1}\mid (C_1,\dots,C_j)\sim \mathrm{Bernoulli}(1/d)$ for every $j=0,\dots,L-1$. Therefore, $(C_i)_{i=1}^L$ are i.i.d. $\mathrm{Bernoulli}(1/d)$, and
\begin{align*}
\Pr(\mathbb B^c)
= \Pr\left(\sum_{i=1}^L C_i<d\right)
= \Pr\left(\mathrm{Bin}(L,1/d)<d\right).
\end{align*}

Let $X\sim \mathrm{Bin}(L,1/d)$ so that $\mathbb EX=L/d$ where $d^2\log d\le L\le d^2\log d+1$.
Then by Hoeffding's inequality,
\begin{align*}
\Pr\left(X <d\right)
&\le \exp\left(-2\frac{(L/d-d)^2}{L}\right)\\
&\le \exp\left(-2\frac{(d\log d-d)^2}{d^2\log d+1}\right)\\
&\le \exp(-2\log d+4) = \frac{e^4}{d^2} < 1,
\end{align*}
since $d \ge 8$.
Also, fix $(\theta(L),\ldots,\theta(1))\in \mathbb{B}$ and let $1\le k_1<\cdots<k_d\le L$ be indices such that $\theta(k_j)=\sigma_0(j)$ for all $j\in[d]$.
Using $\Pi K_j=\Pi$ and $\|K_j\|\le 1$,
\begin{align*}
  \|K_{\theta(L)}\cdots K_{\theta(1)}-\Pi\|
  &=\|K_{\theta(L)}\cdots K_{\theta(k_d+1)}(K_{\theta(k_d)}\cdots K_{\theta(1)}-\Pi)\|\\
  &\le \|K_{\theta(k_d)}\cdots K_{\theta(1)}-\Pi\|\\
  & =\|(K_{\theta(k_d)}\cdots K_{\theta(k_1)}-\Pi)K_{\theta(k_1-1)}\cdots K_{\theta(1)}\|\\
  &\le \|K_{\theta(k_d)}\cdots K_{\theta(k_1)}-\Pi\|.
\end{align*}
Then Lemma~\ref{lem:DCW_larger_product} applied to the length $(k_d-k_1+1)$ product $K_{\theta(k_d)}\cdots K_{\theta(k_1)}$ yields
\begin{align*}
  \|K_{\theta(k_d)}\cdots K_{\theta(k_1)} - \Pi\|
  &\le 1-\frac{\eta_{\mathrm{DCW}}^2(1-\lambda_0)}{4\,((k_d-k_1+1)-d+1)\,(1+\lambda_0)}\\
  & \le 1-\frac{\eta_{\mathrm{DCW}}^2(1-\lambda_0)}{4\,L\,(1+\lambda_0)}\\
  &\le 1-\frac{\eta_{\mathrm{DCW}}^2(1-\lambda_0)}{4(1+\lambda_0)(d^2\log d+1)}\\
  & \le 1-\frac{\eta_{\mathrm{DCW}}^2(1-\lambda_0)}{8(1+\lambda_0)d^2\log d},
\end{align*}
where the last inequality uses $d^2\log d\ge 1$ for $d\ge 2$. Then,
\begin{align*}
  & \|P_{\mathrm{RCW}} - \Pi\|^L\\
  &= \left\|\frac{1}{d^L}\sum_{(\theta(1),\ldots,\theta(L))\in[d]^L}
           \left(K_{\theta(L)}\cdots K_{\theta(1)} -\Pi \right) \right\|\\
  &\le \frac{1}{d^L}\sum_{(\theta(L),\ldots,\theta(1)) \, \in \mathbb B}
        \left\|K_{\theta(L)}\cdots K_{\theta(1)} -\Pi \right\|
     + \frac{1}{d^L}\sum_{(\theta(L),\ldots,\theta(1)) \, \in \mathbb B^c}
       \left\|K_{\theta(L)}\cdots K_{\theta(1)} -\Pi \right\|\\
  &\le \left\{1-\frac{(1-\lambda_0)\eta_{\mathrm{DCW}}^2}{8(1+\lambda_0)d^2\log d}\right\}\Pr(\mathbb B)+\Pr(\mathbb B^c)\\
  &\le 1-\left(1-\frac{e^4}{d^2}\right)\frac{1-\lambda_0}{8(1+\lambda_0)}\cdot\frac{\eta_{\mathrm{DCW}}^2}{d^2\log d},
\end{align*}
where we used $\Pr(\mathbb B) \ge 1 - e^4/d^2 > 0$ and a trivial bound $\left\|K_{\theta(L)}\cdots K_{\theta(1)} -\Pi \right\| \le 1$ for $(\theta(L),\ldots,\theta(1)) \, \in \mathbb B^c$.Finally, using $L=\lceil d^2\log d\rceil \le 2d^2 \log d$ and $(1-x)^{1/L}\le 1-x/L$ for $x\in(0,1]$, 
\begin{align*}
  \|P_{\mathrm{RCW}} - \Pi\|\le 1-c_{\mathrm{RCW}}\,\frac{1-\lambda_0}{1+\lambda_0}\cdot\frac{\eta_{\mathrm{DCW}}^2}{d^4\log^2 d},
\end{align*}
for any constant $c_{\mathrm{RCW}} \in (0, ( 1- e^4 /d^2)/16)$ where $d \ge 8$.
\end{proof}

\section{Central Limit Theorem and Variance Comparison}\label{sec:clt}

In Section~\ref{sec:clt}, we establish a central limit theorem (CLT) for the component-wise chains and compare their asymptotic variances. Throughout this section, we use the $d$-step random-scan component-wise kernel $(P_{\mathrm{RCW}})^d$ rather than $P_{\mathrm{RCW}}$, since one iteration of $P_{\mathrm{DCW}}$ updates all $d$ blocks, whereas one iteration of $P_{\mathrm{RCW}}$ updates only one block.

We rely on a CLT for geometrically ergodic Markov chains \cite[Theorem 21.2.6]{DoucMoulinesPriouret2018}, which states that if a Markov kernel $K$ with unique invariant distribution $\Pi$ satisfies $\sum_{m=0}^\infty \|K^m f\|_{L^2(\Pi)} < \infty$ for any $f \in L^2(\Pi)$ with $\Pi f = 0$, then the CLT holds with asymptotic variance
\begin{align}\label{eq:asymp_var_def}
    \sigma_{K}^2(f) = \Pi(f^2) + 2 \sum_{m=1}^\infty \langle f, K^m f \rangle_{\Pi},
\end{align}
where $\sigma_{K}^2(f) > 0$ whenever $\Pi(f^2) > 0$.

Theorem~\ref{thm:clt_1} shows that a positive $L^2(\Pi)$ spectral gap for the random-scan and deterministic-scan component-wise chains guarantees the CLT for functions with finite second moment.

\begin{theorem}\label{thm:clt_1}
    Let $\Pi$ be the unique invariant distribution. Assume that $P_{\mathrm{RCW}}$ and $P_{\mathrm{DCW}}$ have positive $L^2(\Pi)$ spectral gaps. Let $(X_n)_{n\ge0}$ be a Markov chain with transition kernel $(P_{\mathrm{RCW}})^d$, and let $(Y_n)_{n\ge0}$ be a Markov chain with transition kernel $P_{\mathrm{DCW}}$.
    Then, for any $f \in L^2(\Pi)$ with $\Pi f = 0$, the asymptotic variances of $(X_n)_{n\ge0}$ and $(Y_n)_{n\ge0}$ are well defined by \eqref{eq:asymp_var_def}, and are denoted by $\sigma_{\scriptstyle{\mathrm{RCW}^d}}^2(f)$ and $\sigma_{\scriptstyle{\mathrm{DCW}}}^2(f)$, respectively.
    It follows that the central limit theorem holds under $\mathbb{P}_\Pi$, i.e. when the Markov chain is started from $\Pi$, as $n \to \infty$,
    \begin{align*}
        n^{-1/2} \sum_{k=0}^{n-1} f(X_k) &\xRightarrow{\, \mathbb{P}_\Pi \,} \mathrm N\left(0, \sigma_{\scriptstyle{\mathrm{RCW}^d}}^2(f)\right),\\
        n^{-1/2} \sum_{k=0}^{n-1} f(Y_k) &\xRightarrow{\, \mathbb{P}_\Pi \,} \mathrm N\left(0, \sigma_{\scriptstyle{\mathrm{DCW}}}^2(f)\right).
    \end{align*}
\end{theorem}
The proof follows directly from the central limit theorem for geometrically ergodic Markov chains cited above. Details are provided in Appendix~\ref{Appendix_B}.

\begin{remark}
While Theorem~\ref{thm:clt_1} establishes the central limit theorem under $\mathbb{P}_\Pi$, analogous conclusions for general initial distributions can be obtained under additional regularity assumptions. If $P_{\mathrm{RCW}}$ and $P_{\mathrm{DCW}}$ are $\psi$-irreducible and positive Harris recurrent, then the above central limit theorem holds under $\mathbb P_{\xi}$ for any initial distribution $\xi$ on $\mathcal{X}$ with the same asymptotic variances \cite[Sec 21.1.2]{DoucMoulinesPriouret2018}.
\end{remark}

\begin{remark}
It is well known that geometric ergodicity in total variation distance implies a CLT for functions with a finite $(2+\delta)$ moment for some $\delta>0$ \cite{chan:geye:1994}, and the moment condition cannot be weakened \cite{hagg:2005}.  Geometric ergodicity in $L^2(\Pi)$ implies geometric ergodicity in total variation, which permits the weaker moment condition for $P_{\mathrm{DCW}}$ in Theorem~\ref{thm:clt_1}. Of course, under reversibility, as is the case for $(P_{\mathrm{RCW}})^d$, the two notions of geometric ergodicity are equivalent \cite{gallegos2024equivalences}, and it is well known that only a finite second moment is required \cite{jones2004markov, RobertsRosenthal1997}.
\end{remark}

The variances of the asymptotic distributions in Theorem~\ref{thm:clt_1} are comparable, bounded by a constant that depends on the product of the spectral gaps.

\begin{theorem}\label{thm:varDCW_varRCW}
    Let $\eta_{\mathrm{RCW}}, \eta_{\mathrm{DCW}} \in (0,1]$ denote the spectral gaps of $P_{\mathrm{RCW}}$ and $P_{\mathrm{DCW}}$, respectively. Then for all $f \in L^2(\Pi)$,
    \begin{align*}
        \left|\sigma_{\scriptstyle{\mathrm{DCW}}}^2(f) - \sigma_{\scriptstyle{\mathrm{RCW}^d}}^2(f)\right|
        \le \frac{4}{\eta_{\mathrm{RCW}} \eta_{\mathrm{DCW}}} \,\|f\|_\Pi^2 .
    \end{align*}
\end{theorem}

\begin{proof}[Proof of Theorem~\ref{thm:varDCW_varRCW}]
    Notice that
    \begin{align*}
        \left|\sigma_{\scriptstyle{\mathrm{DCW}}}^2(f)
        -\sigma_{\scriptstyle{\mathrm{RCW}^d}}^2(f)\right|
        &= 2\left|\sum_{n\ge1}\langle f,
        (P_{\mathrm{DCW}})^{n}f\rangle_\Pi
        - \sum_{n\ge1}\langle f,
        (P_{\mathrm{RCW}})^{nd} f\rangle_\Pi\right| \\
        &\le 2\sum_{n\ge1}\left|\langle f,
        ((P_{\mathrm{DCW}})^{n}
        -(P_{\mathrm{RCW}})^{nd})f\rangle_\Pi\right|.
    \end{align*}
    Since each $K_j$ is $\Pi$-invariant, for every $k\ge1$,
    $(P_{\mathrm{DCW}}-\Pi)^k
    = (P_{\mathrm{DCW}})^{k}-\Pi$ and
    $(P_{\mathrm{RCW}}-\Pi)^k
    = (P_{\mathrm{RCW}})^{k}-\Pi$. Then,
    $(P_{\mathrm{DCW}})^{n}
    -(P_{\mathrm{RCW}})^{nd}
    = (P_{\mathrm{DCW}}-\Pi)^{n}
    -(P_{\mathrm{RCW}}-\Pi)^{nd}$.

    Fix $n\ge1$ and for $m=1,\dots,nd$ set
    $j(m) := d - ((m-1)\bmod d)$,
    $A_m := K_{j(m)} - \Pi$ and
    $B := P_{\mathrm{RCW}}-\Pi$.
    Then, for $\ell=1,\dots,n$, each block of length $d$ satisfies 
    \begin{align} \label{eq:asymptotic_var_1}
        A_{(\ell-1)d+1}\cdots A_{\ell d}
        &= (K_d-\Pi)\cdots(K_1-\Pi) \notag\\
        &= K_d\cdots K_1 - \Pi = P_{\mathrm{DCW}}-\Pi,
    \end{align}
    so $\prod_{m=1}^{nd} A_m = A_{1} \,\cdots\, A_{nd}
    = (P_{\mathrm{DCW}}-\Pi)^n$ and
    $B^{nd}
    = (P_{\mathrm{RCW}}-\Pi)^{nd}$, and therefore
    $(P_{\mathrm{DCW}})^{n}
    -(P_{\mathrm{RCW}})^{nd}
    = \prod_{m=1}^{nd} A_m - B^{nd}$.
    Then,
    \begin{align*}
       \prod_{m=1}^{nd} A_m - B^{nd}
        &= \sum_{m=2}^{nd} \left( A_1 \dots A_m B^{nd - m}  - A_1 \dots A_{m-1} B^{nd-m+1} \right) + A_1 B^{nd-1} - B^{nd}\\
        &= \sum_{m=1}^{nd}
        \left(\prod_{t<m}A_t\right)(A_m-B)B^{nd-m},
    \end{align*}
    hence
    \begin{align*}
        \left|\sigma_{\scriptstyle{\mathrm{DCW}}}^2(f)
        -\sigma_{\scriptstyle{\mathrm{RCW}^d}}^2(f)\right|
        &\le 2\sum_{n\ge1}\sum_{m=1}^{nd}
        \left|\left\langle f,
        \left(\prod_{t<m}A_t\right)(A_m-B)B^{nd-m}f\right\rangle_\Pi\right| \\
        &\le 2\sum_{n\ge1}\sum_{m=1}^{nd}
        \|f\|_\Pi\,
        \left\|\left(\prod_{t<m}A_t\right)(A_m-B)B^{nd-m}f\right\|_\Pi.
    \end{align*}
    By assumption,
    $\|P_{\mathrm{RCW}}-\Pi\|
    \le 1-\eta_{\mathrm{RCW}}$ and
    $\|P_{\mathrm{DCW}}-\Pi\|
    \le 1-\eta_{\mathrm{DCW}}$. Then, for all $k,q\ge1$, $\|B^{k}\|= \|(P_{\mathrm{RCW}}-\Pi)^{k}\|
        \le (1-\eta_{\mathrm{RCW}})^{k}$ and $\|(P_{\mathrm{DCW}}-\Pi)^{q}\| \le (1-\eta_{\mathrm{DCW}})^{q}$.
    Moreover
    \begin{align*}
        A_m - B
        &= (K_{j(m)}-\Pi) - (P_{\mathrm{RCW}}-\Pi)
        = K_{j(m)} - P_{\mathrm{RCW}},
    \end{align*}
    and we have $\|A_m - B\|
    \le \|K_{j(m)}\| + \|P_{\mathrm{RCW}}\|
    \le 2$. Then, for $h_m := (A_m-B)B^{nd-m}f$,
    \begin{align*}
        \|h_m\|_\Pi
        \le \|A_m-B\|\,\|B^{nd-m}f\|_\Pi \le 2(1-\eta_{\mathrm{RCW}})^{nd-m}\|f\|_\Pi.
    \end{align*}
    Also, writing $m=qd+r$ where $q\ge0$ a non-negative integer and $1\le r\le d$, the product term satisfies
    \begin{align*}
        \prod_{t<m}A_t
        &= \left(\prod_{\ell=1}^{q}A_{(\ell-1)d+1}\cdots A_{\ell d}\right)
           \left(\prod_{t=qd+1}^{qd+r-1}A_t\right).
    \end{align*}
    For any $g\in L^2(\Pi)$,
    $(K_j - \Pi)g
    = K_j(g - \Pi g)$. Together with $\|g - \Pi g\| \le \| g\|$ this implies $\|K_j - \Pi\|\le1$ and therefore $\|A_m\|\le1$ for all $m$. Using Equation~\eqref{eq:asymptotic_var_1}, 
    \begin{align*}
        \left\|\prod_{t<m}A_t\right\|
        &\le \|(P_{\mathrm{DCW}}-\Pi)^q\|\,
             \prod_{t=qd+1}^{qd+r-1}\|A_t\| \le (1-\eta_{\mathrm{DCW}})^{q}.
    \end{align*}
    Therefore
    \begin{align*}
        \left\|\left(\prod_{t<m}A_t\right)h_m\right\|_\Pi
        &\le (1-\eta_{\mathrm{DCW}})^{q}\|h_m\|_\Pi \\
        &\le 2(1-\eta_{\mathrm{DCW}})^{\lfloor(m-1)/d\rfloor}
           (1-\eta_{\mathrm{RCW}})^{nd-m}\|f\|_\Pi.
    \end{align*}
    Combining the bounds, obtain
    \begin{align*}
        \left|\sigma_{\scriptstyle{\mathrm{DCW}}}^2(f)
        -\sigma_{\scriptstyle{\mathrm{RCW}^d}}^2(f)\right|
        &\le 4\|f\|_\Pi^2
        \sum_{n\ge1}\sum_{m=1}^{nd}
        (1-\eta_{\mathrm{DCW}})^{\lfloor(m-1)/d\rfloor}
        (1-\eta_{\mathrm{RCW}})^{nd-m}.
    \end{align*}
    Let $a = 1-\eta_{\mathrm{DCW}}$ and
    $b = 1-\eta_{\mathrm{RCW}}$. Then,
    \begin{align*}
        \sum_{n\ge1}\sum_{m=1}^{nd}
        a^{\lfloor(m-1)/d\rfloor}b^{nd-m}
        & = \sum_{n\ge 1} \Bigl\{ (b^{nd-1} + \dots + b^{nd - d}) + a(b^{nd-(d+1)} + \dots + b^{nd - 2d})  \\ & \qquad + \dots + a^{n-1} (b^{nd - ((n-1)d + 1)} + \dots + b^{nd - nd}) \Bigr\}\\
        & = (1 + b + \cdots) + a (1 + b + \cdots) + a^2 (1 + b + \cdots)\cdots\\
        & = \frac{1}{(1-a)(1-b)} = \frac{1}{\eta_{\mathrm{DCW}}\,
        \eta_{\mathrm{RCW}}},
    \end{align*}
    where, for the second equality, all the $b$-terms are grouped according to their order in $a$. Hence
    \begin{align*}
        \left|\sigma_{\scriptstyle{\mathrm{DCW}}}^2(f)
        -\sigma_{\scriptstyle{\mathrm{RCW}^d}}^2(f)\right|
        &\le \frac{4}{\eta_{\mathrm{DCW}}\,
        \eta_{\mathrm{RCW}}}\,\|f\|_\Pi^2,
    \end{align*}
    which completes the proof.
\end{proof}

\begin{remark}
    Since $P_{\mathrm{RCW}}$ is self-adjoint on $L^2(\Pi)$ with spectral gap $\eta_{\mathrm{RCW}}$, the spectral gap of $(P_{\mathrm{RCW}})^d$ equals $1-(1-\eta_{\mathrm{RCW}})^d$. While the variance of the asymptotic distribution $\sigma^2_{\scriptstyle{\mathrm{RCW}^d}}(f)$ can be bounded in terms of $\eta_{\mathrm{RCW}}$, the focus here is to demonstrate that the variance of the asymptotic distribution of the deterministic-scan kernel is comparable to that of the $d$-step random-scan kernel.
\end{remark}

\section{Application: Gaussian MALA}\label{sec:MALA example}

The goal is to derive a spectral gap bound for the deterministic-scan component-wise Metropolis-adjusted Langevin algorithm (MALA) targeting a multivariate Gaussian distribution using the results of Section~\ref{sec:RCW_to_DCW}, and to clarify how the resulting geometric rates depend on the target precision matrix and guide step-size choices.

\subsection{Gaussian block MALA}\label{sec:gaussian_block_MALA}

Consider block-wise MALA targeting a Gaussian distribution $\Pi=\mathcal N(\mu,Q^{-1})$ on $\mathbb R^N$ with density $\pi$, where $\mu = (\mu_1,\dots,\mu_d) \in\mathbb R^N$ and $Q$ is a positive definite precision matrix. Write the precision matrix as a block matrix
\begin{align*}
Q = \begin{pmatrix}
Q_{1,1} & Q_{1,2} & \cdots & Q_{1,d} \\
Q_{2,1} & Q_{2,2} & \cdots & Q_{2,d} \\
\vdots  & \vdots  & \ddots & \vdots  \\
Q_{d,1} & Q_{d,2} & \cdots & Q_{d,d}
\end{pmatrix}, \qquad Q_{j,k}\in\mathbb R^{N_j\times N_k},
\end{align*} 
where $N_j$ is the size of block $j$ so that $\sum_{j=1}^d N_j = N$. Recall that $x_{-j}=(x_k)_{k\ne j}$, and define $\mu_{-j}=(\mu_k)_{k\ne j}$. Define the off-diagonal block row $Q_{j,-j} = (Q_{j,k})_{k\ne j}\in\mathbb R^{N_j\times (N-N_j)}.$ For the update of the $j$th block, the conditional distribution of the $j$th block under $\Pi$ is 
\begin{align*}
\Pi_j(\cdot \mid x_{-j}) = \mathcal N\left(m_j(x_{-j}),\,Q_{j,j}^{-1}\right)
\qquad
m_j(x_{-j}) = \mu_j - Q_{j,j}^{-1}Q_{j,-j}(x_{-j}-\mu_{-j}),
\end{align*}
and denote its density by $\pi_{j,x_{-j}}$.  

Consider a component-wise MALA update targeting the conditional distribution. For the update of the $j$th block $x_j$ with step size $h_j>0$, MALA proposes a state $y=(y_j,x_{-j})$, with
\begin{align*}
y_j \mid x \sim \mathcal N\left(x_j+h_j\nabla_j\log\pi(x),\,2h_j I_{N_j}\right),
\end{align*}
and accepts $y$ with probability
\begin{align*}
\alpha_j(x,y)=\min \left\{ 1,  \frac{\pi(y)\,r_j(y,x)}{\pi(x)\,r_j(x,y)} \right\},
\end{align*}
where $r_j(x,y)$ is the corresponding proposal density. Then the block MALA update admits the component-wise form
\begin{align*}
K_j(x,\mathrm{d}y)
= K_{j,x_{-j}}(x_j,\mathrm{d}y_j) \, \delta_{x_{-j}}(\mathrm{d}y_{-j}),
\end{align*}
where $K_{j,x_{-j}}$ is the MALA kernel on $\mathbb R^{N_j}$ with invariant distribution $\Pi_j(\cdot\mid x_{-j})$. The corresponding Gibbs update $P_j$ is the conditional expectation operator
\begin{align*}
P_j(x,\mathrm{d}y)
= \Pi_j(\mathrm{d}y_j\mid x_{-j})\,\delta_{x_{-j}}(\mathrm{d}y_{-j}).
\end{align*}

The goal is to obtain a lower bound on the spectral gap of the deterministic-scan component-wise MALA chain. Recall that Corollary~\ref{cor:RSG_to_DCW_2} requires the spectral gap of the random-scan Gibbs sampler $\eta_{\mathrm{RSG}}$ and the global $L^2$ contraction rate $\lambda_0$. Define $D:=\operatorname{diag}(Q_{1,1}^{-1},\dots,Q_{d,d}^{-1})$, the inverse block diagonal of $Q$. Then, the random-scan Gibbs sampler has a spectral gap 
\begin{align}\label{eq:Gaussian_RSG_spectral_gap}
\eta_{\mathrm{RSG}}=\frac{1-\psi_{\max}(I_N-DQ)}{d}\in(0,1],
\end{align}
where $\psi_{\max}$ denotes the largest eigenvalue of a given matrix \cite{RobertsSahu1997}. Note that $1-\psi_{\max}(I_N-DQ) = \psi_{\min}(DQ)$, which is always positive since the precision matrix $Q$ is positive definite. This gap, combined with the global block-wise contraction constant $\lambda_0$ for block-wise MALA updates, can be used to obtain a spectral gap bound for the deterministic-scan component-wise MALA chain.

\begin{proposition}\label{prop:MALA_DCW_gap}
Fix $j\in[d]$. Let $\psi_{\max}(Q_{j,j})$ and $\psi_{\min}(Q_{j,j})$ be the largest and smallest eigenvalues of $Q_{j,j}$. If $h_j \le 1/\{2\psi_{\max}(Q_{j,j})\}$ and $h_j^2 \le 1/\{10\,\operatorname{tr}(Q_{j,j}^2)\}$, then
\begin{align*}
\|K_j-P_j\|\le 1-c_0 h_j \psi_{\min}(Q_{j,j}),
\end{align*} 
where $c_0 \approx 3.1\times 10^{-6}> 0$ is a universal constant.
\end{proposition}

Thus, the step size cannot be too large, as in other results for MALA \cite{Chewi2021,roberts1998optimal}. Moreover, the admissible range of $h_j$ is determined explicitly by the block precision matrix $Q_{j,j}$ of the target Gaussian distribution. This yields a concrete criterion for choosing $h_j$ so that the deterministic-scan component-wise MALA chain has a positive spectral gap. See Appendix~\ref{Appendix_D} for details.

Applying Corollary~\ref{cor:RSG_to_DCW_2} with Proposition~\ref{prop:MALA_DCW_gap} yields an $L^2(\Pi)$ convergence rate bound for the deterministic-scan component-wise MALA chain $P_{\mathrm{DCW}}$.
\begin{proposition}\label{prop:DC_MALA_gap}
Assume that the conditions of Proposition~\ref{prop:MALA_DCW_gap} hold for all $j\in[d]$. Define $\lambda_0 := \max_{j\in[d]} \left(1-c_0 h_j \psi_{\min}(Q_{j,j})\right).$ Then
\begin{align*}
\|P_{\mathrm{DCW}}-\Pi\|
\le
1-\frac{(1-\lambda_0)^3}{8(d+1)(1+\lambda_0)} \frac{1-\psi_{\max}(I_N-DQ)}{d},
\end{align*}
where $D = \operatorname{diag}(Q_{1,1}^{-1},\dots,Q_{d,d}^{-1})$.
\end{proposition}

\begin{remark}\label{rem:dim_of_h_j}
Suppose that $\psi_{\max}(Q_{j,j})$ remains bounded as $N_j$ increases. Then, the choice 
\begin{align*}
h_j \le \frac{1}{\sqrt{10 N_j}\,\psi_{\max}(Q_{j,j})} \le \frac{1}{\sqrt{10 \operatorname{tr}(Q_{j,j}^2) }} 
\end{align*}
is sufficient for the step-size conditions in Proposition~\ref{prop:MALA_DCW_gap}. Thus, up to constants depending on $Q_{j,j}$, an admissible step size scales as $h_j \asymp N_j^{ \,-1/2}$.
\end{remark}
    
\begin{remark}\label{rem:within_and_between_block}
Note that the term $1-\psi_{\max}(I_N-DQ)=\psi_{\min}(DQ)$ is governed by the between-block dependence of the precision matrix $Q$. Suppose that the between-block dependence is weak in the sense that $\|Q_{j,k}\|_2 \ll \psi_{\min}(Q_{j,j})$ for all $j\neq k$. Then,
\begin{align*}
\|(DQ)_{j,k}\|_2
= \|Q_{j,j}^{-1}Q_{j,k}\|_2
\le \|Q_{j,j}^{-1}\|_2\,\|Q_{j,k}\|_2
= \frac{\|Q_{j,k}\|_2}{\psi_{\min}(Q_{j,j})} \ll 1.
\end{align*}
Hence, $DQ$ is close to the block identity, and $\eta_{\mathrm{RSG}}\approx 1/d$.

By contrast, $\lambda_0$ reflects the within-block dependence. Indeed, Proposition~\ref{prop:MALA_DCW_gap} gives $\lambda_0=\max_{j\in[d]}\left(1-c_0 h_j \psi_{\min}(Q_{j,j})\right),$
where $h_j$ depends on $Q_{j,j}$ through the constraints $h_j \le 1/\{2\psi_{\max}(Q_{j,j})\}$ and $h_j^2 \le 1/\{10\,\operatorname{tr}(Q_{j,j}^2)\}$.
Thus, stronger within-block dependence typically forces smaller step sizes and hence larger $\lambda_0$.
\end{remark}

In Sections~\ref{sec:compound_symmetry} and~\ref{sec:Autoregressive}, two specific covariance structures will be considered, namely compound-symmetry and first-order autoregressive. For each case, sufficient step-size choices $\{h_j\}_{j\in[d]}$  are provided in terms of the block matrices $\{Q_{j,j}\}_{j\in[d]}$ along with explicit lower bounds on the spectral gap of the corresponding deterministic-scan component-wise MALA chain.

\subsection{Gaussian targets with compound-symmetry covariance}\label{sec:compound_symmetry}

Consider the Gaussian target distribution with compound-symmetry covariance matrix. Lemma~\ref{lem:MALA_compound_symmetry_precision} gives a sufficient step size $h_j$ together with an upper bound on the global block-wise contraction constant $\lambda_0$. The following proposition computes $\psi_{\max}(I_N-DQ)$ explicitly and yields the corresponding spectral gap bound for the deterministic-scan component-wise MALA chain.

\begin{lemma}\label{lem:MALA_compound_symmetry_precision}
Suppose $\Sigma = Q^{-1} = (1-\zeta)I_N + \zeta 11^\top \in \mathbb R^{N\times N}$ for $\zeta\in [0,1)$, and let $N_1 = N_2 = \dots = N_d = s$, so that $N = sd$. Then the step size $h_j = (1-\zeta)/(\sqrt{10 s})$ satisfies the conditions in Proposition~\ref{prop:MALA_DCW_gap} and the corresponding constant $\lambda_0$ satisfies
\begin{align*}
\lambda_0 \le 1 - c_0 \frac{1+ (N - s- 1)\zeta}{ 1+(N-1)\zeta} \frac{1}{\sqrt{10 s}} .
\end{align*}
where $c_0$ is the universal constant in Proposition~\ref{prop:MALA_DCW_gap}.
\end{lemma}

\begin{proposition}\label{prop:psiA_compound_symmetry}
Suppose the assumptions of Lemma~\ref{lem:MALA_compound_symmetry_precision}. Let $d\ge 2$ and $D=\operatorname{diag}(Q_{1,1}^{-1},\dots,Q_{d,d}^{-1})$. Then, the largest eigenvalue of $I_N-DQ$ equals
\begin{align*}
\psi_{\max}(I_N-DQ)=\frac{(N-s)\zeta}{1+(N-s-1)\zeta}.
\end{align*}
It follows that the deterministic-scan component-wise MALA chain has spectral gap satisfying
\begin{align*}
  \|P_{\mathrm{DCW}}-\Pi\| \le 1-\frac{c_{0}^{\,3}(1-\zeta)}{2^9\,N\sqrt{s}(d+1)}\cdot
\frac{\left(1+(N-s-1)\zeta\right)^{2}}{\left(1+(N-1)\zeta\right)^{3}}
\end{align*}
\end{proposition}
We provide the proofs in Appendix~\ref{Appendix_D}. The bound increases monotonically as $\zeta$ decreases and is maximized at $\zeta=0$, where both the covariance and precision matrices reduce to the $N\times N$ identity matrix. The bound becomes smaller as the block size $s$ or the number of blocks $d$ increases.

\subsection{Gaussian targets with autoregressive covariance}\label{sec:Autoregressive}

Consider the first-order autoregressive covariance matrix
\begin{align*}
    \Sigma = Q^{-1} = \begin{pmatrix}
        1 & \varphi & \varphi^2 & \cdots & \varphi^{N-1} \\
        \varphi & 1 & \varphi & \cdots & \varphi^{N-2} \\
        \vdots & \vdots & \vdots & \ddots & \vdots \\
        \varphi^{N-1} & \varphi^{N-2} & \varphi^{N-3} & \cdots & 1
    \end{pmatrix},
\end{align*}
where $\varphi \in (-1,1)$. Then the precision matrix $Q=\Sigma^{-1}$ can be written as the tridiagonal matrix
\begin{align*}
Q
= \frac{1}{1-\varphi^2}
\begin{pmatrix}
1 & -\varphi & 0 & 0 & \cdots & 0 & 0 & 0 \\
-\varphi & 1+\varphi^2 & -\varphi & 0 & \cdots & 0  & 0 & 0 \\
0 & -\varphi & 1+\varphi^2 & -\varphi & \cdots & 0 & 0 & 0 \\
\vdots & \vdots & \vdots & \vdots & \ddots & \vdots & \vdots & \vdots \\
0 & 0 & 0 & 0 & \cdots & -\varphi & 1+\varphi^2 & -\varphi \\
0 & 0 & 0 & 0 & \cdots & 0 & -\varphi & 1
\end{pmatrix}.
\end{align*}
As in the compound-symmetry case, we first choose a sufficient step size by controlling the block-wise eigenvalues of $Q_{j,j}$ and then derive a spectral gap bound for the deterministic-scan component-wise MALA chain.
\begin{lemma}\label{lem:MALA_AR1_precision}
Assume the above first-order autoregressive covariance model. Suppose $N_1=\dots=N_d=s$, so that $N=sd$. Then,
\begin{align*}
h_j = \frac{1}{\sqrt{10s}} \, \frac{1-|\varphi|}{1+|\varphi|}
\end{align*}
satisfies the step-size conditions in Proposition~\ref{prop:MALA_DCW_gap} for every $j\in[d]$ and yields
\begin{align*}
\lambda_0 \le 1-c_0\frac{(1-|\varphi|)^2}{(1+|\varphi|)^2}\,\frac{1}{\sqrt{10s}}.
\end{align*}
\end{lemma}

\begin{proposition}\label{prop:psiA_AR1}
Assume the setting of Lemma~\ref{lem:MALA_AR1_precision}. Let $D=\operatorname{diag}(Q_{1,1}^{-1},\dots,Q_{d,d}^{-1})$. Then,
\begin{align*}
\psi_{\max}(I_N-DQ)\le \frac{|\varphi|+|\varphi|^{s}}{1+|\varphi|^{s+1}}.
\end{align*}
It follows that the deterministic-scan component-wise MALA chain has a spectral gap satisfying
\begin{align*}
  \|P_{\mathrm{DCW}}-\Pi\| \le 1-\frac{c_{0}^{\,3}}{2^9\,N\sqrt{s}(d+1)}\cdot
\frac{(1-|\varphi|)^{7}\left(1-|\varphi|^{s}\right)}{\left(1+|\varphi|\right)^{6}\left(1+|\varphi|^{s+1}\right)}.
\end{align*}
\end{proposition}
The proof of Lemma~\ref{lem:MALA_AR1_precision} uses the uniform bounds $\psi_{\max}(Q_{j,j}) \le ({1+|\varphi|}) / ({1-|\varphi|})$ and $\psi_{\min}(Q_{j,j}) \ge ({1-|\varphi|})/({1+|\varphi|})$,
together with $\operatorname{tr}(Q_{j,j}^2)\le s\psi_{\max}(Q_{j,j})^2$. Here, Proposition~\ref{prop:psiA_AR1} suggests that the resulting spectral gap bound decreases as $|\varphi|$ increases. See Appendix~\ref{Appendix_D} for details.

\section{Final Remarks}\label{sec:final_remarks}

The $L^2(\Pi)$ spectral gaps of deterministic-scan and random-scan component-wise samplers were established and quantified under a global block-wise contraction condition. The results show that geometric ergodicity is preserved when passing from the Gibbs sampler to more general component-wise updates under this condition. We further established explicit polynomial bounds relating the spectral gaps of the deterministic- and random-scan component-wise chains and showed that their asymptotic variances are comparable. As an application, the contraction condition was verified for Gaussian block-wise MALA updates, and explicit spectral gap bounds were obtained for deterministic-scan component-wise MALA under compound-symmetry and autoregressive covariance structures.

\bibliographystyle{plain}
\bibliography{refs}
\newpage

\appendix

\section{Proofs in Section 2}\label{Appendix_A}

\begin{proof}[Proof of Lemma~\ref{lem:Pj_Kj_properties}]

Let $X = (x_j, x_{-j}) \sim \Pi$. For any $f\in L^2(\Pi)$, we have
\begin{align*}
(P_j f)(x_j,x_{-j})
= \int_{\X_j} f(y_j,x_{-j})\,\Pi_j(\df y_j\mid x_{-j}),
\end{align*}
where $\Pi_j(d y_j \mid x_{-j})$ is a conditional distribution of $y_j$ given $x_{-j}$. So $P_j f$ is conditional expectation, and therefore $P_j$ is the $L^2(\Pi)$-orthogonal projection onto the subspace of $x_{-j}$-measurable functions. Hence, $P_j$ is idempotent, that is, $P_j^2 = P_j$.

Next, recall that
\begin{align*}
(K_j f)(x_j,x_{-j})
= \int_{\X_j} f(y_j,x_{-j})\,K_{j,x_{-j}}(x_j,\df y_j).
\end{align*}
Since $P_j f$ depends only on $x_{-j}$, $(P_j f)(y_j,x_{-j}) = (P_j f)(x_j,x_{-j})$ for all $x_j,y_j\in\X_j$. Then,
\begin{align*}
(K_j P_j f)(x_j,x_{-j})
&= \int_{\X_j} (P_j f)(y_j,x_{-j})\,K_{j,x_{-j}}(x_j,\df y_j) \\
&= (P_j f)(x_j,x_{-j}) \int_{\X_j} K_{j,x_{-j}}(x_j,\df y_j) \\
&= (P_j f)(x_j,x_{-j}),
\end{align*}
so $K_j P_j = P_j$.

Finally, for $\Pi_{-j}$-almost every $x_{-j}$, $K_{j,x_{-j}}$ is invariant with respect to $\Pi_j(\cdot\mid x_{-j})$. That is, for all $A_j\in\B_j$,
\begin{align*}
\int_{\X_j} K_{j,x_{-j}}(z_j,A_j)\,\Pi_j(\df z_j\mid x_{-j})
= \Pi_j(A_j\mid x_{-j}).
\end{align*}
By Fubini's theorem, it follows that
\begin{align*}
(P_j K_j f)(x_j,x_{-j})
&= \int_{\X_j} (K_j f)(z_j,x_{-j})\,\Pi_j(\df z_j\mid x_{-j}) \\
&= \int_{\X_j} \int_{\X_j} f(y_j,x_{-j})\,K_{j,x_{-j}}(z_j,\df y_j)\,\Pi_j(\df z_j\mid x_{-j}) \\
&= \int_{\X_j} f(y_j,x_{-j})\,\Pi_j(\df y_j\mid x_{-j}) \\
&= (P_j f)(x_j,x_{-j}),
\end{align*}
so $P_j K_j = P_j$.
\end{proof}
Recall that $\mathcal{M}_j = \operatorname{ran}(P_j)$ and $\mathcal{N}_j = \operatorname{ker}(P_j)$. The following lemmas establish basic properties of $K_j$ and $P_j$ when restricted to $\mathcal{M}_j$ and $\mathcal{N}_j$.

\begin{lemma}\label{lem:basic_Kj_1}
    For any $f \in \PiL$, let $f_0 = (I-P_j) f$. Then, $(P_j - K_j) f = -K_j f_0$
\end{lemma}

\begin{proof}
    \begin{align*}
        (P_j-K_j)f 
        &= P_j f - K_j (P_j f + (I - P_j)f) \\ 
        &= P_j f - K_j P_j f - K_j f_0 \\ 
        &= P_j f - P_j f - K_j f_0 = -K_j f_0.
    \end{align*}
    where the third equality follows from Lemma~\ref{lem:Pj_Kj_properties}
\end{proof}

\begin{lemma}\label{lem:basic_Kj_2}
    Define $K_{\mathcal{N}_j}:\mathcal{N}_j\to\mathcal{N}_j$ as $K_{\mathcal{N}_j} f_0 = (I-P_j)K_j f_0$ for $f_0\in\mathcal{N}_j$. Then $K_j f_0 = K_{\mathcal{N}_j} f_0$. Also, for any $f\in L^2(\Pi)$ with $f_0 = (I-P_j)f\in\mathcal{N}_j$, we have $(I-K_j)f = (I_{\mathcal{N}_j}-K_{\mathcal{N}_j})f_0$.
\end{lemma}

\begin{proof}
Note that for any $f_0 \in \mathcal{N}_j$ we have $P_j(K_j f_0) = P_j f_0 = 0$ by Lemma~\ref{lem:Pj_Kj_properties}. Thus $K_j(\mathcal{N}_j) \subseteq \mathcal{N}_j$, and the restricted operator $K_{\mathcal{N}_j}:\mathcal{N}_j \to \mathcal{N}_j$ given by $K_{\mathcal{N}_j} = (I-P_j)K_j$ is well defined. Then,
    \begin{align*}
        K_j f_0 &= K_j (I- P_j) f_0 = (I-P_j) K_j f_0 = K_{\mathcal{N}_j} f_0
    \end{align*}
    where second equality follows from $K_jP_j = P_jK_j$ by Lemma~\ref{lem:Pj_Kj_properties}. Further,
    \begin{align*}
        (I-K_j) f &= (I- K_j) (P_j f + f_0) = (I - K_j) f_0 = f_0 - K_{\mathcal{N}_j} f_0 = (I_{\mathcal{N}_j} - K_{\mathcal{N}_j}) f_0
    \end{align*}
    where second equality follows from $K_jP_j = P_j$ by Lemma~\ref{lem:Pj_Kj_properties}.
\end{proof}

\begin{proof}[Proof of Remark~\ref{rem:weights_independent}]
Assume that $\|P_{\mathrm{RCW}}-\Pi\|\le 1-\eta_{\mathrm{RCW}}$.
Let $(w_j)_{j=1}^d$ be weights satisfying $w_j>0$ for all $j\in[d]$ and
$\sum_{j=1}^d w_j=1$.
Define $w_{\min}:=\min_{j\in[d]} w_j>0$. Set $P_{\mathrm{RCW}, w}:=\sum_{j=1}^d w_j K_j$.
Then
\begin{align*}
P_{\mathrm{RCW}, w}-\Pi
 = \sum_{j=1}^d w_j (K_j-\Pi) = d \cdot w_{\min}\left(P_{\mathrm{RCW}}-\Pi\right)
  + \sum_{j=1}^d (w_j-w_{\min})(K_j-\Pi).
\end{align*}
Taking operator norms and using the triangle inequality gives
\begin{align*}
\|P_{\mathrm{RCW}, w}-\Pi\|
&\le d \cdot w_{\min}\,\|P_{\mathrm{RCW}}-\Pi\|
    + \sum_{j=1}^d (w_j-w_{\min})\,\|K_j-\Pi\|.
\end{align*}
Since each $K_j$ is a Markov operator with invariant measure $\Pi$, we have
$\|K_j-\Pi\|\le 1$.
Therefore
\begin{align*}
\|P_{\mathrm{RCW}, w}-\Pi\|
&\le d \cdot w_{\min}(1-\eta_{\mathrm{RCW}}) + \sum_{j=1}^d (w_j-w_{\min})
= 1 - d \cdot w_{\min}\eta_{\mathrm{RCW}}.
\end{align*}
Hence, the weighted random-scan component-wise chain $P_{\mathrm{RCW}, w}$ has a positive $L^2(\Pi)$ spectral gap, bounded below by $d\cdot w_{\min}\,\eta_{\mathrm{RCW}}$.
\end{proof}

\begin{proof}[Proof of Lemma~\ref{lem:block_contraction_equiv}]
Recall that $\mathcal{N}_j = \operatorname{ker}(P_j)$. For any $f\in L^2(\Pi)$, let $f_0 = (I-P_j)f\in\mathcal{N}_j$. By Lemma~\ref{lem:basic_Kj_1}, we have $(P_j-K_j)f=-K_j f_0$. It follows that
\begin{align}\label{eq:L2_to_Nj}
\|P_j-K_j\|
= \sup_{f\in L^2(\Pi)\setminus\{0\}} \frac{\|(P_j-K_j)f\|_\Pi}{\|f\|_\Pi}
= \sup_{f_0\in \mathcal{N}_j\setminus\{0\}} \frac{\|K_j f_0\|_\Pi}{\|f_0\|_\Pi}.
\end{align}
The second equality holds because every $f\in L^2(\Pi)$ admits the orthogonal decomposition $f=P_j f+f_0$ with $f_0:=(I-P_j)f\in\mathcal{N}_j$, so that $\|f\|_\Pi^2=\|P_j f\|_\Pi^2+\|f_0\|_\Pi^2$ and hence $\|f\|_\Pi\ge \|f_0\|_\Pi$. Moreover, since $(P_j-K_j)P_j f=0$ and $(P_j-K_j)f_0=-K_j f_0$ for $f_0\in\mathcal{N}_j$, we obtain $(P_j-K_j)f=-K_j f_0$.

Now, consider
\begin{align*}
\|K_j f_0\|_\Pi^2
= \int_{\X_{-j}} \int_{\X_j} \left|(K_j f_0)(x_j,x_{-j})\right|^2 \,\Pi_j(\df x_j\mid x_{-j})\,\Pi_{-j}(\df x_{-j}).
\end{align*}
Since $K_j$ only updates the $j$th component, we have $(K_j f_0)(\cdot,x_{-j}) = K_{j,x_{-j}}\left(f_0(\cdot,x_{-j})\right)$ in $L^2\left(\Pi(\cdot\mid x_{-j})\right)$, so the inner integral equals $\left\|K_{j,x_{-j}} f_0(\cdot,x_{-j})\right\|_{j, x_{-j}}^2$. Thus
\begin{align*}
\|K_j f_0\|_\Pi^2
= \int_{\X_{-j}} \left\|K_{j,x_{-j}} f_0(\cdot,x_{-j})\right\|_{j, x_{-j}}^2 \,\Pi_{-j}(\df x_{-j}).
\end{align*}

By assumption, there exists $\lambda_j\in(0,1)$ such that
$\|K_{j,x_{-j}} - P_{j,x_{-j}}\|_{j,x_{-j}}\le \lambda_j$ for $\Pi_{-j}$-almost every $x_{-j}\in\X_{-j}$.
Let $f_0\in \mathcal{N}_j$. Then $P_j f_0 = 0$, so by the definition of $P_j$ we have
$P_{j,x_{-j}} f_0(\cdot,x_{-j}) = 0$ for $\Pi_{-j}$-almost every $x_{-j}$, and hence
\begin{align*}
K_{j,x_{-j}} f_0(\cdot,x_{-j})
= (K_{j,x_{-j}} - P_{j,x_{-j}}) f_0(\cdot,x_{-j}).
\end{align*}
Then,
\begin{align*}
\|K_{j,x_{-j}} f_0(\cdot,x_{-j})\|_{j,x_{-j}}
= \|(K_{j,x_{-j}} - P_{j,x_{-j}}) f_0(\cdot,x_{-j})\|_{j,x_{-j}}
\le \lambda_j \,\|f_0(\cdot,x_{-j})\|_{j,x_{-j}}
\end{align*}
for $\Pi_{-j}$-almost every $x_{-j}$. Therefore,
\begin{align*}
\|K_j f_0\|_\Pi^2
&= \int_{\X_{-j}} \left\|K_{j,x_{-j}} f_0(\cdot,x_{-j})\right\|_{j,x_{-j}}^2\,\Pi_{-j}(\df x_{-j}) \\
&\le \lambda_j^2 \int_{\X_{-j}} \left\|f_0(\cdot,x_{-j})\right\|_{j,x_{-j}}^2\,\Pi_{-j}(\df x_{-j}) \\
&= \lambda_j^2 \|f_0\|_\Pi^2.
\end{align*}
Taking the supremum over $f_0\in \mathcal{N}_j\setminus\{0\}$ and using \eqref{eq:L2_to_Nj} yields
\begin{align*}
\|P_j-K_j\|
= \sup_{f_0\in \mathcal{N}_j\setminus\{0\}} \frac{\|K_j f_0\|_\Pi}{\|f_0\|_\Pi}
\le \lambda_j.
\end{align*}

Further, we show that $\|K_j^{n}-P_j\|\to 0$ as $n\to\infty$ if and only if $\|K_j -P_j\| \le \lambda_j$ for some $\lambda_j \in (0,1)$. Since each $K_j$ leaves $\Pi$ invariant and only updates the $j$th component, Lemma~\ref{lem:Pj_Kj_properties} gives $K_j P_j = P_j K_j = P_j$. This implies the algebraic identity
\begin{align*}
(K_j - P_j)^n = K_j^n - P_j \quad \text{for any } n \ge 1.
\end{align*}
Moreover, $K_j$ is self-adjoint on $L^2(\Pi)$, and $P_j$ is an orthogonal projection. Therefore $K_j - P_j$ is also self-adjoint, so by the spectral theorem
\begin{align*}
\|K_j^n - P_j\|
= \|(K_j - P_j)^n\|
= \|K_j - P_j\|^n.
\end{align*}
If $\|K_j - P_j\|\le\lambda_j<1$, then $\|K_j^n - P_j\|\le\lambda_j^n\to0$ as $n\to\infty$, which implies $\|K_j^{n}-P_j\|\to 0$ as $n\to\infty$. Conversely, if $\|K_j^{n}-P_j\|\to 0$ as $n\to\infty$, $\|K_j - P_j\|^n = \|K_j^n - P_j\|\to0$, which implies $\|K_j - P_j\|<1$, so $\|K_j - P_j\|\le\lambda_j$ holds with some $\lambda_j\in(0,1)$.
\end{proof}

\begin{proof}[Proof of Remark~\ref{rem:necessity_block_contraction}]
Consider the hybrid sampler which targets $\Pi=\mathcal N(0,1)\otimes\mathcal N(0,1)$ on $\X=\mathbb R^2$ \cite[Proposition~3.1]{RobertsRosenthal1997}. For $j\in\{1,2\}$ let $P_j$ denote the Gibbs update of the $j$th coordinate and let $K_j$ denote the random-walk Metropolis with proposal distribution as follows. Given $(X_n,Y_n)=(x,y)$, $K_1$ keeps $y$ fixed and proposes $z_x\sim\mathcal N(x,1+y^2)$, and $K_2$ keeps $x$ fixed and proposes $z_y\sim\mathcal N(y,1+x^2)$. Let $x_{-1}=y$ and $x_{-2}=x$. For each $x_{-j}$, the conditional target is $\Pi_j(\cdot\mid x_{-j})=\mathcal N(0,1)$, and the operator $P_{j,x_{-j}}$ is the orthogonal projection onto constants in $L^2(\mathcal N(0,1))$.

Our goal is to show that $\|K_j - P_j \| = 1$ by using arbitrarily small acceptance probabilities. Fix $j\in\{1,2\}$ and let $\varepsilon\in(0,1/4)$. Let $a_{j,x_{-j}}(u)$ denote the one-step acceptance probability of $K_{j,x_{-j}}$ when the current state equals $u\in\X_j$. Then, for every $\varepsilon>0$ we can choose a conditioning value $x_{-j}^0$ and a point $u_0\in\X_j$ such that $a_{j,x_{-j}^0}(u_0)\le \varepsilon/2$ \cite{RobertsRosenthal1997}.

Since the Metropolis acceptance probability is continuous in both the current state and the conditioning value for this Gaussian proposal, there exist open neighborhoods $A_0\subset\X_j$ of $u_0$ and $B\subset\X_{-j}$ of $x_{-j}^0$, with $\Pi_{-j}(B)>0$, such that
\begin{align*}
    \sup_{u\in A_0,\ x_{-j}\in B} a_{j,x_{-j}}(u)\le \varepsilon.
\end{align*}
Since $\Pi_j(\cdot\mid x_{-j})=\mathcal N(0,1)$ is non-atomic, we can choose a measurable set $A\subseteq A_0$ such that
$0 < p:=\Pi_j(A\mid x_{-j})\le \varepsilon$. Note that $p$ does not depend on $x_{-j}$ because the conditional law is always $\mathcal N(0,1)$. In particular, $p\le \varepsilon < 1/2$, so $\sqrt{p/(1-p)}\le \sqrt{(1-p)/p}$.

Define
\begin{align*}
g(\cdot):=\frac{\mathbf 1_A(\cdot)-p}{\sqrt{p(1-p)}}.
\end{align*}
Then $g(u)=\sqrt{(1-p)/p}$ for $u\in A$ and $-\sqrt{p/(1-p)}$ for $u\notin A$. Hence $P_{j,x_{-j}}g=0$, and $\|g\|_{\Pi_{j,x_{-j}}}=1$
for every $x_{-j}\in\X_{-j}$.

For $u\in A$ and $x_{-j}\in B$, write
\begin{align*}
(K_{j,x_{-j}}g)(u)
= (1-a_{j,x_{-j}}(u))g(u)
+ \int g(v)\,K_{j,x_{-j}}(u,\mathrm dv)\,\mathbf 1_{\{v\ne u\}}.
\end{align*}
Since $g(v)\ge -\sqrt{p/(1-p)}$ for all $v$ and $\mathbf 1_{\{v\ne u\}}$ integrates to $a_{j,x_{-j}}(u)$, obtain
\begin{align*}
(K_{j,x_{-j}}g)(u)
&\ge (1-a_{j,x_{-j}}(u))\sqrt{\frac{1-p}{p}}
- a_{j,x_{-j}}(u)\sqrt{\frac{p}{1-p}}\\
&\ge (1-2a_{j,x_{-j}}(u))\sqrt{\frac{1-p}{p}}\\
&\ge (1-2\varepsilon)\sqrt{\frac{1-p}{p}}.
\end{align*}
Therefore, for every $x_{-j}\in B$,
\begin{align*}
\|K_{j,x_{-j}}g\|_{\Pi_{j,x_{-j}}}^2
&\ge \int_A (K_{j,x_{-j}}g)^2\,\mathrm d\Pi_j(\cdot\mid x_{-j})\\
&\ge p(1-2\varepsilon)^2\frac{1-p}{p}
= (1-2\varepsilon)^2(1-p).
\end{align*}
Using $p\le \varepsilon$ and $\sqrt{1-\varepsilon}\ge 1-\varepsilon$ gives
\begin{align*}
\|K_{j,x_{-j}}g\|_{\Pi_{j,x_{-j}}}
\ge (1-2\varepsilon)\sqrt{1-p}
\ge (1-2\varepsilon)\sqrt{1-\varepsilon}
\ge 1-3\varepsilon
\end{align*}
for every $x_{-j}\in B$.

Now define the global test function
\begin{align*}
F(x_j,x_{-j})
:= \frac{g(x_j)\mathbf 1_B(x_{-j})}{\sqrt{\Pi_{-j}(B)}}.
\end{align*}
Then $P_{j,x_{-j}}g=0$ for every $x_{-j}\in\X_{-j}$ implies $P_jF=0$, and
\begin{align*}
\|F\|_\Pi^2
&= \frac{1}{\Pi_{-j}(B)}
\int_B
\left(
\int |g(x_j)|^2 \,\Pi_j(\mathrm d x_j\mid x_{-j})
\right)
\Pi_{-j}(\mathrm d x_{-j})\\
&= \frac{1}{\Pi_{-j}(B)}
\int_B \Pi_{-j}(\mathrm d x_{-j})
=1.
\end{align*}
Hence
\begin{align*}
\|(K_j-P_j)F\|_\Pi^2
= \|K_jF\|_\Pi^2
&= \frac{1}{\Pi_{-j}(B)}
\int_B \|K_{j,x_{-j}}g\|_{\Pi_{j,x_{-j}}}^2\,\Pi_{-j}(\mathrm dx_{-j})\\
&\ge (1-3\varepsilon)^2.
\end{align*} 
Therefore, $\|K_j-P_j\|\ge 1-3\varepsilon.$ Since $\varepsilon>0$ is arbitrary, this yields $\|K_j-P_j\|\ge 1.$

On the other hand, by Lemma~\ref{lem:Pj_Kj_properties},
\begin{align*}
\|K_j-P_j\| = \|K_j(I-P_j) \|
\le \|K_j\|\,\|I-P_j\|
\le 1.
\end{align*}
Combining the lower and upper bounds gives $\|K_j-P_j\|=1$. Since this holds for $j=1,2$, we conclude that
\begin{align*}
\lambda_0=\max_{j\in\{1,2\}}\|K_j-P_j\|=1.
\end{align*}

We now show that $P_{\mathrm{RCW}}$ has zero $L^2(\Pi)$ spectral gap, and that $K_1K_2$ and $K_2K_1$ also have zero $L^2(\Pi)$ spectral gap in this example. Define the random-scan kernel $P_{\mathrm{RCW}}=\tfrac12(K_1+K_2)$ which is $\Pi$-reversible. Then, $P_{\mathrm{RCW}}$ is not geometrically ergodic \cite[Proposition~3.1]{RobertsRosenthal1997}. Since $P_{\mathrm{RCW}}$ is reversible, this implies that $P_{\mathrm{RCW}}$ has zero $L^2(\Pi)$ spectral gap and therefore $\|P_{\mathrm{RCW}}-\Pi\|=1$.

Because $P_{\mathrm{RCW}}$ is self-adjoint on $L^2(\Pi)$, we have
\begin{align*}
\|P_{\mathrm{RCW}}^2-\Pi\|
= \|(P_{\mathrm{RCW}}-\Pi)^2\|
= \|P_{\mathrm{RCW}}-\Pi\|^2
= 1.
\end{align*}
A direct expansion gives
\begin{align*}
P_{\mathrm{RCW}}^2 = \frac14\left(K_1^2 + K_1K_2 + K_2K_1 + K_2^2\right).
\end{align*}
Using triangle inequality, we obtain
\begin{align*}
1 
= \|P_{\mathrm{RCW}}^2-\Pi\|
\le \frac14\left(\|K_1^2-\Pi\| + \|K_1K_2-\Pi\| + \|K_2K_1-\Pi\| + \|K_2^2-\Pi\|\right).
\end{align*}
Each of the four terms is at most $1$ since for any Markov operator $K$ with invariant $\Pi$ one has $\|K-\Pi\|\le 1$. Hence the average of four numbers in $[0,1]$ is at least $1$ only if each equals $1$. In particular,
\begin{align*}
\|K_1K_2-\Pi\|=\|K_2K_1-\Pi\|=1.
\end{align*}
Therefore the deterministic-scan compositions $K_1K_2$ and $K_2K_1$ have zero $L^2(\Pi)$ spectral gap as well.
\end{proof}

\section{Proofs in Section 3 and 4} \label{Appendix_B}

Recall that $\mathcal{M}_j=\operatorname{ran}(P_j)$ and $\mathcal{N}_j=\operatorname{ker}(P_j)$. Also recall that $K_{\mathcal{N}_j}:\mathcal{N}_j\to\mathcal{N}_j$ is given by $K_{\mathcal{N}_j}f_0=(I-P_j)K_jf_0$, where $ f_0\in\mathcal{N}_j$. We begin by decomposing the operator $K_j$ with respect to the orthogonal decomposition $L^2(\Pi)=\mathcal{M}_j\oplus\mathcal{N}_j$.
\begin{lemma}\label{lem:eigenvalue-decomposition}
If $\mathrm{Spec}(T)$ denotes the spectrum of the operator $T$, then
\begin{align*}
\mathrm{Spec}(K_j)=\{1\}\cup \mathrm{Spec}(K_{\mathcal{N}_j}).
\end{align*}
\end{lemma}

\begin{proof}[Proof of Lemma~\ref{lem:eigenvalue-decomposition}]
    For any $f_1 \in \mathcal{M}_j$, $P_jf_1 = f_1$. Since $K_j P_j = P_j$ by Lemma~\ref{lem:Pj_Kj_properties}, it follows that $K_j f_1 = K_j P_j f_1 = P_j f_1 = f_1$, which implies $1 \in \mathrm{Spec}(K_j)$. We characterize the spectrum of $K_j$ via the invertibility of $\lambda I-K_j$.
    
    First, assume $\lambda \neq 1$ and $\lambda \notin \mathrm{Spec}(K_{\mathcal{N}_j})$.
    Then, for any $y \in L^2(\Pi)$ one can uniquely decompose $y$ as $y = y_M + y_N$ where $y_M = P_j y \in \mathcal{M}_j$ and $y_N = (I-P_j)y \in \mathcal{N}_j$.
    We seek a solution $x = x_M + x_N$ where $x_M \in \mathcal{M}_j$ and $x_N \in \mathcal{N}_j$ to the equation $(\lambda I - K_j)x = y$.
    Since $P_jx_M = x_M$ and $(I-P_j)x_N = x_N$, projecting this equation onto $\mathcal{M}_j$ and $\mathcal{N}_j$ yields a decoupled system:
    \begin{align*}
    P_j(\lambda I - K_j)(x_M + x_N) &= (\lambda - 1)x_M = y_M, \\
    (I-P_j)(\lambda I - K_j)(x_M + x_N) 
    & = (\lambda I - K_j)(I-P_j) x_M + \lambda (I-P_j) x_N - (I-P_j) K_j x_N \\
    & = \lambda x_N - K_{\mathcal{N}_j} x_N \\
    &= (\lambda I_{\mathcal{N}_j} - K_{\mathcal{N}_j})x_N = y_N.
    \end{align*}
    Since $\lambda \neq 1$ we can uniquely solve $x_M = (\lambda - 1)^{-1}y_M$.
    Since $\lambda \notin \mathrm{Spec}(K_{\mathcal{N}_j})$ the operator $\lambda I_{\mathcal{N}_j} - K_{\mathcal{N}_j}$ is invertible on $\mathcal{N}_j$, so we uniquely solve $x_N = (\lambda I_{\mathcal{N}_j} - K_{\mathcal{N}_j})^{-1}y_N$.
    Thus $\lambda I-K_j$ is bijective, and hence $\lambda\notin \mathrm{Spec}(K_j)$ whenever $\lambda\neq 1$ and $\lambda\notin \mathrm{Spec}(K_{\mathcal{N}_j})$.
    
    Second, assume $\lambda \notin \mathrm{Spec}(K_j)$, so that the operator $\lambda I - K_j$ is invertible.
    If $\lambda = 1$ then for any nonzero $f_1 \in \mathcal{M}_j$ we have $(\lambda I - K_j)f_1 = (I - K_j)f_1 = 0$, which contradicts the injectivity of $\lambda I - K_j$.
    Hence $\lambda \neq 1$.
    
    Next, for any $f_0 \in \mathcal{N}_j$ we have
    $(\lambda I - K_j)f_0 = (\lambda I_{\mathcal{N}_j} - K_{\mathcal{N}_j})f_0$ by Lemma~\ref{lem:basic_Kj_2}.
    Since $\lambda I - K_j$ is injective on $L^2(\Pi)$, its restriction $\lambda I_{\mathcal{N}_j} - K_{\mathcal{N}_j}$ is injective on $\mathcal{N}_j$.
    Moreover, for any $y_N \in \mathcal{N}_j$, since $\lambda I - K_j$ is surjective there exists $x \in L^2(\Pi)$ such that $(\lambda I - K_j)x = y_N$.
    Projecting onto $\mathcal{M}_j$ gives $(\lambda - 1)P_j x = P_j y_N = 0$, and since $\lambda \neq 1$ this implies $P_j x = 0$, that is, $x \in \mathcal{N}_j$.
    Thus $(\lambda I_{\mathcal{N}_j} - K_{\mathcal{N}_j})x = y_N$, which shows that $\lambda I_{\mathcal{N}_j} - K_{\mathcal{N}_j}$ is surjective on $\mathcal{N}_j$.
    Therefore $\lambda \notin \mathrm{Spec}(K_{\mathcal{N}_j})$.
    Hence $\lambda \notin \mathrm{Spec}(K_j)$ implies $\lambda \neq 1$ and $\lambda \notin \mathrm{Spec}(K_{\mathcal{N}_j})$.
    
    Combining both directions, we conclude that $\mathrm{Spec}(K_j) = \{1\} \cup \mathrm{Spec}(K_{\mathcal{N}_j})$.
\end{proof}

\begin{remark}\label{rem:KjPj_to_KNj}
Suppose $\|K_j - P_j\|\le \lambda_j < 1$ for some $\lambda_j$. Then, for any $f_0 \in \mathcal{N}_j$,
    \begin{align*}
        \|K_{\mathcal{N}_j} f_0\| = \|(I - P_j) K_j f_0\| = \|(K_j - P_j) f_0\| \le \lambda_j \|f_0\|,
    \end{align*}
and since $K_{\mathcal{N}_j}$ is also self-adjoint, this implies $\mathrm{Spec}(K_{\mathcal{N}_j}) \subseteq [-\lambda_j,\lambda_j]$. In particular, all eigenvalues of $K_j$ different from $1$ have modulus at most $\lambda_j$, so the contribution of the non-projection part $K_{\mathcal{N}_j}$ is uniformly controlled by the constant $\lambda_j$.
\end{remark}

Thus, conditional on $x_{-j}$, the component-wise update $K_j$ is a strict contraction on $\mathcal{N}_j$. The next lemma quantifies this by bounding $\|(P_j-K_j)f\|_\Pi$ in terms of the one-step change $\|(I-K_j)f\|_\Pi$.

\begin{lemma}\label{lem:upperbound_nonprojection}
If $\|K_j-P_j\|\le \lambda_j<1$, then for any $f \in \PiL$,
\begin{align*}
\|(P_j-K_j)f\|_\Pi\;\le\;\frac{\lambda_j}{1-\lambda_j}\,\|(I-K_j)f\|_\Pi.
\end{align*}
\end{lemma}

\begin{proof}[Proof of Lemma~\ref{lem:upperbound_nonprojection}]
    Let $f \in \PiL$ and set $f_0 = (I-P_j)f \in \mathcal{N}_j$.
   By Lemma~\ref{lem:basic_Kj_1} and Lemma~\ref{lem:basic_Kj_2} we have
    \begin{align*}
        (P_j - K_j)f &= -K_j f_0 =  -K_{\mathcal{N}_j} f_0, \\
        (I-K_j)f &= (I_{\mathcal{N}_j} - K_{\mathcal{N}_j})f_0.
    \end{align*}

    By Remark~\ref{rem:KjPj_to_KNj} the assumption $\|P_j - K_j\| \le \lambda_j$ implies $\|K_{\mathcal{N}_j}\| \le \lambda_j$. Then, $\|(I_{\mathcal{N}_j} - K_{\mathcal{N}_j})^{-1}\| \le (1-\lambda_j)^{-1}$.
    Using these bounds together with the identities above, we obtain
    \begin{align*}
        \|(P_j-K_j)f\|_\Pi
        &= \|K_{\mathcal{N}_j} f_0\|_\Pi \\
        &= \|K_{\mathcal{N}_j}(I_{\mathcal{N}_j}-K_{\mathcal{N}_j})^{-1}(I_{\mathcal{N}_j} - K_{\mathcal{N}_j})f_0\|_\Pi \\
        &\le \|K_{\mathcal{N}_j}\| \, \|(I_{\mathcal{N}_j}-K_{\mathcal{N}_j})^{-1}\| \, \|(I-K_j)f\|_\Pi \\
        &\le \frac{\lambda_j}{1-\lambda_j} \, \|(I-K_j)f\|_\Pi,
    \end{align*} 
    where the third inequality follows from the definition of the operator norm.
\end{proof}

Next lemma quantifies the contribution of each step of the deterministic-scan component-wise update.

\begin{lemma}\label{lem:spectral_ineq}
Assume $\|K_j-P_j\|\le \lambda_j<1$. Then for every $f \in L^2(\Pi)$,
    \begin{align*}
    \|(I-K_j)f\|_\Pi^2 + \frac{1+\lambda_j}{1-\lambda_j} \|K_j f\|_\Pi^2
    \le \frac{1+\lambda_j}{1-\lambda_j} \|f\|_\Pi^2.
    \end{align*}
\end{lemma}

\begin{proof}[Proof of Lemma~\ref{lem:spectral_ineq}]
By Remark~\ref{rem:KjPj_to_KNj} and Lemma~\ref{lem:eigenvalue-decomposition}, $\|K_j - P_j \|\le \lambda_j$ implies $\mathrm{Spec}(K_j) \subset [-\lambda_j,1]$. Let $a_j > 0$ and define the quadratic polynomial
\begin{align*}
    g_{a_j}(t) := (1-t)^2 + a_j t^2 - a_j.
\end{align*}
We have $g''_{a_j}(t) = 2(1+a_j) > 0$ for all $t$, so $g_{a_j}$ is convex.
We choose $a_j$ so that $g_{a_j}(1) = 0$ and $g_{a_j}(-\lambda_j) = 0$.
The first condition holds for any $a_j$ since $g_{a_j}(1)=(1-1)^2 + a_j - a_j = 0$.
The second condition requires
\begin{align*}
    (1+\lambda_j)^2 + a_j \lambda_j^2 - a_j = 0
\end{align*}
which is equivalent to
\begin{align*}
    a_j = \frac{(1+\lambda_j)^2}{1-\lambda_j^2}
= \frac{1+\lambda_j}{1-\lambda_j}.
\end{align*}
With this choice, the convexity of $g_{a_j}$ and the inclusion
$\mathrm{Spec}(K_j) \subset [-\lambda_j,1]$ imply
\begin{align*}
    g_{a_j}(t) \le 0 \qquad \text{for all } t \in \mathrm{Spec}(K_j).
\end{align*}
Since $K_j$ is a self-adjoint operator, 
\begin{align*}
    g_{a_j}(K_j) = (I-K_j)^2 + a_j K_j^2 - a_j I
\end{align*}
is a self-adjoint operator whose spectrum is the image $g_{a_j}(\mathrm{Spec}(K_j))$.
Since $g_{a_j}(t) \le 0$ on $\mathrm{Spec}(K_j)$, the operator $g_{a_j}(K_j)$ is negative semi-definite.
Hence for any $f \in L^2(\Pi)$,
\begin{align*}
    \langle g_{a_j}(K_j) f, f \rangle_\Pi \le 0.
\end{align*}
Expanding this inner product with $a_j = (1+\lambda_j) / (1-\lambda_j)$ gives
\begin{align*}
    \|(I-K_j)f\|_\Pi^2 + \frac{1+\lambda_j}{1-\lambda_j} \|K_j f\|_\Pi^2 - \frac{1+\lambda_j}{1-\lambda_j}  \|f\|_\Pi^2 \le 0.
\end{align*}
which completes the proof.
\end{proof}

We next prove Lemma~\ref{lem:residual_upperbound} and Lemma~\ref{lem:residual_lowerbound}.
\begin{proof}[Proof of Lemma~\ref{lem:residual_upperbound}]
    Decompose the residual $K_j f - f$ using a telescoping sum:
    \begin{align*}
        K_j f - f 
        & = K_j(u_0(f) - u_{j-1} (f)) + (u_j (f) - u_0(f)) \\
        & = K_j \left(\sum_{m=1}^{j-1} (u_{m-1} (f) - u_m (f))\right) + \sum_{m=1}^j (u_m (f) - u_{m-1} (f)).
    \end{align*}
    Using the triangle inequality and $\|K_j\| \le 1$, 
    \begin{align*}
        \|K_j f - f\|_\Pi 
        & \leq \| K_j \| \sum_{m=1}^{j-1} \| u_{m-1} (f) - u_m (f) \|_\Pi + \sum_{m=1}^j \| u_m (f) - u_{m-1} (f) \|_\Pi \\
        & \leq \sum_{m=1}^{j-1} \| u_{m-1} (f) - u_m (f) \|_\Pi + \sum_{m=1}^j \| u_m (f) - u_{m-1} (f) \|_\Pi \\
        & \leq 2 \sum_{m=1}^j \| u_m (f) - u_{m-1} (f) \|_\Pi.
    \end{align*}
   Taking squares on both sides and applying the Cauchy--Schwarz inequality yields
    \begin{align*}
        \|K_j f - f\|_\Pi^2 
        & \leq 4 \left(\sum_{m=1}^j 1 \cdot \| u_m (f) - u_{m-1} (f) \|_\Pi \right)^2 \\
        & \leq 4 j \sum_{m=1}^j \| u_m (f) - u_{m-1} (f) \|_\Pi^2.
    \end{align*}
    Then, by Lemma~\ref{lem:spectral_ineq} with $u_{m-1}(f)$ and $\lambda_0 = \max_j \lambda_j$, we have
    \begin{align*}
        \|K_j f - f\|_\Pi^2 
        & \leq 4 j \sum_{m=1}^j \frac{1+\lambda_m}{1-\lambda_m} \left( \|u_{m-1}(f)\|_\Pi^2 - \| u_{m}(f)\|_\Pi^2 \right) \\
        & \leq 4 j \frac{1+\lambda_0}{1-\lambda_0} \sum_{m=1}^j \left( \|u_{m-1}(f)\|_\Pi^2 - \| u_{m}(f)\|_\Pi^2 \right) \\
        & = 4 j \frac{1+\lambda_0}{1-\lambda_0} \left(\|u_0(f)\|_\Pi^2 - \| u_{j}(f)\|_\Pi^2\right).
    \end{align*}
    Substituting $u_0(f) = f - \Pi f$ completes the proof.
\end{proof}

\begin{proof}[Proof of Lemma~\ref{lem:residual_lowerbound}]
     For any $f \in \PiL$,  set $f_0 = (I-P_j) f \in \mathcal{N}_j$. Recall that $\mathcal{N}_j = \operatorname{ker}(P_j)$ and $K_{\mathcal{N}_j} f_0 = (I-P_j) K_j f_0$. By Lemma~\ref{lem:basic_Kj_2}, $(K_j - I) f = (K_{\mathcal{N}_j} - I_{\mathcal{N}_j}) f_0\in \mathcal{N}_j$. Since $P_j f$ is orthogonal to the $\mathcal{N}_j$,
    \begin{align*}
        \langle f, (K_j - I) f \rangle_\Pi 
        &= \langle P_j f + f_0, (K_{\mathcal{N}_j} - I_{\mathcal{N}_j}) f_0 \rangle_\Pi \\
        &= \langle f_0, (K_{\mathcal{N}_j} - I_{\mathcal{N}_j})f_0 \rangle_\Pi.
    \end{align*}
    By Remark~\ref{rem:KjPj_to_KNj}, $\|P_j - K_j \| \le \lambda_j < 1$ implies $\mathrm{Spec}(K_{\mathcal{N}_j}) \subseteq [-\lambda_j, \lambda_j]$. Consider the inequality
    \begin{align*}
        (1-t)^2 \ge (1-\lambda_j)(1-t),
    \end{align*}
     which holds for all $t \in [-\lambda_j, \lambda_j]$ since $1-t \ge 1-\lambda_j > 0$. Since $K_{\mathcal{N}_j}$ is a self-adjoint operator, this scalar inequality extends to the operator inequality on $\mathcal{N}_j$
     \begin{align*}
        \langle (I_{\mathcal{N}_j} - K_{\mathcal{N}_j})^2 f_0, f_0 \rangle_\Pi \ge (1-\lambda_j) \langle (I_{\mathcal{N}_j} - K_{\mathcal{N}_j}) f_0, f_0 \rangle_\Pi.         
     \end{align*}
    Substituting the norm and inner product expressions derived above, we obtain
    \begin{align*}
        \|K_j f - f\|_\Pi^2 
        &= \|(K_{\mathcal{N}_j} - I_{\mathcal{N}_j}) f_0 \|_\Pi^2 \\
        &= \langle (I_{\mathcal{N}_j} - K_{\mathcal{N}_j})^2 f_0, f_0 \rangle_\Pi \\
        &\ge (1-\lambda_j) \langle (I_{\mathcal{N}_j} - K_{\mathcal{N}_j}) f_0, f_0 \rangle_\Pi \\
        &= -(1-\lambda_j) \langle f_0 , (K_{\mathcal{N}_j} - I_{\mathcal{N}_j}) f_0 \rangle_\Pi \\
        &= -(1-\lambda_j) \langle f, K_j f - f \rangle_\Pi.
    \end{align*}
\end{proof}

\begin{proof}[Proof of Lemma~\ref{lem:RCW_to_RSG}]
    Let $f \in L^2(\Pi)$ be such that $\|f - \Pi f\|_\Pi = 1$. Then
    \begin{equation} \label{eq:qjw-1}
    \begin{aligned}
        \langle f, (P_{\mathrm{RCW}} - \Pi) f \rangle_\Pi &= d^{-1} \sum_{j=1}^d \langle f, (K_j - \Pi) f \rangle_\Pi \\
        &= d^{-1} \sum_{j=1}^d \langle f, (K_j - P_j) f \rangle_\Pi + d^{-1} \sum_{j=1}^d \langle f, (P_j - \Pi) f \rangle_\Pi.
    \end{aligned}
    \end{equation}
    By Lemma~\ref{lem:Pj_Kj_properties}, $(K_j - P_j)f = (K_j - P_j) (I-P_j)f$. Also, $P_j (K_j - P_j) = P_j - P_j = 0$ since $P_j$ is an orthogonal projection and $P_j K_j = P_j$. Then, $|\langle f, (K_j - P_j) f \rangle_\Pi| = |\langle (I-P_j) f, (K_j - P_j) (I - P_j) f \rangle_\Pi|$ and it follows that
    \begin{equation} \label{ine:qjw-2}
    \begin{aligned}
        |\langle f, (K_j - P_j) f \rangle_\Pi| &= |\langle (I-P_j) f, (K_j - P_j) (I - P_j) f \rangle_\Pi| \\
        & \leq \| K_j - P_j \| \|(I-P_j) f\|_\Pi^2 \\
        &\leq \lambda_0 \|(I-P_j) f\|_\Pi^2 \\
        &= \lambda_0 \langle f, (I-P_j) f \rangle_\Pi \\
        &= \lambda_0 \langle f, (I-\Pi) f \rangle_\Pi - \lambda_0 \langle f, (P_j - \Pi) f \rangle_\Pi \\
        &= \lambda_0 - \lambda_0 \langle f, (P_j - \Pi) f \rangle_\Pi
    \end{aligned}
    \end{equation}
    Combining \eqref{eq:qjw-1} and \eqref{ine:qjw-2} shows that
    \begin{equation} \label{ine:qjw-3}
    \begin{aligned}
        & \langle f, (P_{\mathrm{RCW}} - \Pi) f \rangle_\Pi \leq \lambda_0 + (1- \lambda_0) \langle f, (P_{\mathrm{RSG}} - \Pi) f \rangle_\Pi, \\
        & \langle f, (P_{\mathrm{RCW}} - \Pi) f \rangle_\Pi \geq -\lambda_0 + (1 + \lambda_0) \langle f, (P_{\mathrm{RSG}} - \Pi) f \rangle_\Pi,
    \end{aligned}
    \end{equation}
    where $\langle f, (P_{\mathrm{RSG}} - \Pi) f \rangle_\Pi = d^{-1} \sum_{i=1}^d \langle f, (P_i - \Pi) f \rangle_{\Pi}$.
    Note that $P_j$ is an orthogonal projection, $(I-\Pi)^2 = I-\Pi$ is self-adjoint and $P_j \Pi = \Pi P_j = \Pi$, so
    \begin{align*}
        \langle f, (P_{\mathrm{RSG}} - \Pi) f \rangle_\Pi = d^{-1} \sum_{j=1}^d \langle f - \Pi f, P_j(f - \Pi f) \rangle_\Pi = d^{-1} \sum_{j=1}^d \|P_j (f - \Pi f) \|_\Pi^2 \geq 0. 
    \end{align*}
    Take the supremum in each inequality with respect to $f$ in \eqref{ine:qjw-3}. We get
    \begin{equation} \label{ine:qjw-sup}
    \begin{aligned}
        & \sup_{f: \, \|f - \Pi f\|_\Pi = 1} \langle f, (P_{\mathrm{RCW}} - \Pi) f \rangle_\Pi \leq \lambda_0 + (1 - \lambda_0) \|P_{\mathrm{RSG}} - \Pi\|, \\
        & \sup_{f: \, \|f - \Pi f\|_\Pi = 1} \langle f, (P_{\mathrm{RCW}} - \Pi) f \rangle_\Pi \geq - \lambda_0 + (1 + \lambda_0) \|P_{\mathrm{RSG}} - \Pi\|.
    \end{aligned}
    \end{equation}
    Take the infimum in each inequality with respect to $f$ in \eqref{ine:qjw-3}. We obtain
    \begin{equation} \label{ine:qjw-inf}
        -\lambda_0 \leq \inf_{f: \, \|f - \Pi f\|_\Pi = 1} \langle f, (P_{\mathrm{RCW}} - \Pi) f \rangle_\Pi \leq \lambda_0.
    \end{equation}
    Since $P_{\mathrm{RCW}} - \Pi$ is self-adjoint, its norm is $\sup_{f: \,\|f - \Pi f\|_\Pi = 1} |\langle f, (P_{\mathrm{RCW}} - \Pi) f \rangle_\Pi|$.
    Thus, by \eqref{ine:qjw-sup} and \eqref{ine:qjw-inf},
    \begin{align*}
        -\lambda_0 + (1+\lambda_0) \|P_{\mathrm{RSG}} - \Pi\| \leq \|P_{\mathrm{RCW}} - \Pi\| \leq \lambda_0 + (1-\lambda_0) \|P_{\mathrm{RSG}} - \Pi\|.
    \end{align*}
    The desired result then follows.
\end{proof}

\begin{proof}[Proof of Lemma~\ref{lem:DCW_larger_product}]
    Define indices $(k_j^\sigma)_{j=1}^d$ by
    \begin{align*}
        k_1^\sigma & := \min \{k \in [L] \mid \theta(k) = \sigma(1) \}, \\
        k_j^\sigma & := \min \{k \in [L] \mid k > k_{j-1}, \theta(k) = \sigma(j) \}.
    \end{align*}
    Since $(\theta(L), \dots, \theta(1))$ contains $(\sigma(d), \dots, \sigma(1))$ as a subsequence, $k_d^\sigma \le L$. Fix $f \in L^2(\Pi)$ with $\|f - \Pi f\|_\Pi=1$. Define the trajectory sequence by
    \begin{align*}
        D(0) := f - \Pi f, \qquad D(\ell) := K_{\theta(\ell)} \dots K_{\theta(1)} f - \Pi f, \quad \ell = 1,\dots,L.
    \end{align*}
    Note that all $D(\ell)$ are functions of $f$. For any $\ell \in \{0,\dots,L-1\}$, use Lemma~\ref{lem:spectral_ineq} with $j = \theta(\ell + 1)$, $f = D(\ell)$. Then,
    \begin{align*}
        \|D(\ell+1) - D(\ell)\|_\Pi^2 
        \leq \frac{1 + \lambda_{\theta(\ell+1)}}{1 - \lambda_{\theta(\ell+1)}} 
        \left(\|D(\ell)\|_\Pi^2 -  \|D(\ell+1)\|_\Pi^2\right).
    \end{align*}
    
    We establish two bounds for $\|D(L)\|_\Pi$. First, using the telescoping sum and the fact that norms of the sequence $D(\ell)$ are non-increasing by $\| K_{\ell}\| \le 1$, for any subset $S \subset [L]$,
    \begin{align}
        \|D(L)\|_\Pi^2 
        &= \|f - \Pi f\|_\Pi^2 + \sum_{\ell = 1}^L \left(\|D(\ell)\|_\Pi^2 - \|D(\ell-1)\|_\Pi^2\right) \notag \\
        &\leq 1 + \sum_{\ell \in S} \left(\|D(\ell)\|_\Pi^2 - \|D(\ell-1)\|_\Pi^2\right). \label{eq:thm42eq1}
    \end{align}
    Second, decompose $D(L)$ as
    \begin{align*}
        D(L) 
        &= K_{\theta(L)} \dots K_{\theta(k_d + 1)} K_{\theta(k_d)} \, D(k_d-1) \\
        &= K_{\theta(L)} \dots K_{\theta(k_d + 1)} K_{\sigma(d)} \, D(k_d-1) \\        
        & = K_{\theta(L)} \dots K_{\theta(k_d + 1)}  \\
        & \hspace{2em} \{(K_{\sigma(d)} \dots K_{\sigma(1)} - \Pi)f + K_{\sigma(d)} (D(k_d-1) - (K_{\sigma(d-1)} \dots K_{\sigma(1)} - \Pi )f ) \},
    \end{align*}
    since $\Pi$ is invariant distribution for all $K_{\theta(\ell)}$. Using the assumption $\|K_{\sigma(d)} \dots K_{\sigma(1)} - \Pi\| \le 1-\eta_{\mathrm{DCW}}$ and $\|K_{\theta(\ell)}\| \le 1$ for all $\ell$, the triangle inequality yields
    \begin{equation}\label{eq:thm42eq2}
    \begin{aligned}
    & \|D(L)\|_\Pi\\
    &\le \|K_{\theta(L)} \dots K_{\theta(k_d + 1)}\|\\
    & \hspace{2em} \left\{\|(K_{\sigma(d)} \dots K_{\sigma(1)} - \Pi)f\|_\Pi
    + \|K_{\sigma(d)}(D(k_d-1) - (K_{\sigma(d-1)} \dots K_{\sigma(1)} - \Pi)f)\|_\Pi\right\} \\
    &\le \|(K_{\sigma(d)} \dots K_{\sigma(1)} - \Pi)f\|_\Pi
    + \|K_{\sigma(d)}\|\,\|D(k_d-1) - (K_{\sigma(d-1)} \dots K_{\sigma(1)} - \Pi)f\|_\Pi \\
    &\le 1 - \eta_{\mathrm{DCW}}
    + \|D(k_d-1) - (K_{\sigma(d-1)} \dots K_{\sigma(1)} - \Pi) f\|_\Pi .
    \end{aligned}
    \end{equation}

    To control the residual term, define $R_j := D(k_{j+1}-1) - (K_{\sigma(j)} \dots K_{\sigma(1)} - \Pi) f$ for $j=1,\dots,d-1$. Note that the second term in \eqref{eq:thm42eq2} corresponds to $R_{d-1}$. We bound $R_j$ recursively. Observing that $D(k_{j+1}-1) = D(k_j) + \sum_{\ell = k_j}^{k_{j+1}-2} (D(\ell+1) - D(\ell))$,
    \begin{align}\label{inq:Rj}
        \|R_j\|_\Pi
        &\leq \|D(k_j) - (K_{\sigma(j)} \dots K_{\sigma(1)} - \Pi) f\|_\Pi + \sum_{\ell = k_j}^{k_{j+1}-2} \|D(\ell+1) - D(\ell)\|_\Pi.
    \end{align}
    Since $\theta(k_j)=\sigma(j)$, $D(k_j) = K_{\sigma(j)} D(k_j-1)$. Thus,
    \begin{align*}
       \|D(k_j) - (K_{\sigma(j)} \dots K_{\sigma(1)} - \Pi) f\|_\Pi
        &= \|K_{\sigma(j)}(D(k_j-1) - (K_{\sigma(j-1)} \dots K_{\sigma(1)} - \Pi) f)\|_\Pi\\
        &\le \|R_{j-1}\|_\Pi,
    \end{align*}
    by $\|K_{\sigma(j)}\| \le 1$. Substituting this back to~\eqref{inq:Rj} and applying the Cauchy--Schwarz inequality,
    \begin{align*}
        \|R_j\|_\Pi
        &\le \|R_{j-1}\|_\Pi + \sqrt{k_{j+1} - k_j - 1}\, \sqrt{\sum_{\ell = k_j}^{k_{j+1}-2} \|D(\ell+1) - D(\ell)\|_\Pi^2}.
    \end{align*}
    For each $\ell = k_j, k_j + 1, \dots, k_{j+1}-2$, use Lemma~\ref{lem:spectral_ineq} with $j = \theta(\ell)$, $f = D(\ell)$. Also, for any $\ell$ it holds that ${(1+\lambda_{\theta(\ell)})}/{(1-\lambda_{\theta(\ell)})} \le {(1+\lambda_0)}/{(1-\lambda_0)}$. Thus,
    \begin{align*}
        \|R_j\|_\Pi \le \|R_{j-1}\|_\Pi + \sqrt{\frac{(k_{j+1} - k_j - 1)(1+\lambda_0)}{1-\lambda_0}} \sqrt{\sum_{\ell = k_j}^{k_{j+1}-2} \left(\|D(\ell)\|_\Pi^2 - \|D(\ell+1)\|_\Pi^2\right)}.
    \end{align*}
    Set $\Xi_\ell := \|D(\ell)\|_\Pi^2 - \|D(\ell+1)\|_\Pi^2$. Iterating this bound from $j=1$ to $d-1$ with $R_0 = 0$, and applying Cauchy--Schwarz again to the sum over $j$ we get
    \begin{align*}
        \|R_{d-1}\|_\Pi
        &\le \sqrt{\frac{1+\lambda_0}{1-\lambda_0}} \sum_{j=1}^{d-1} \sqrt{k_{j+1} - k_j - 1} \sqrt{\sum_{\ell = k_j}^{k_{j+1}-2} \Xi_\ell} \\
        &\le \sqrt{\frac{1+\lambda_0}{1-\lambda_0}} \sqrt{\sum_{j=1}^{d-1} (k_{j+1} - k_j - 1)} \sqrt{\sum_{j=1}^{d-1} \sum_{\ell = k_j}^{k_{j+1}-2} \Xi_\ell},
    \end{align*}
    Let $S = \bigcup_{j=1}^{d-1} \{k_j,\dots,k_{j+1}-2\}$. Since $\sum_{j=1}^{d-1} (k_{j+1} - k_j - 1) \le L-d$, substituting this into \eqref{eq:thm42eq2} gives
    \begin{align}\label{eq:thm42fin}
        \|D(L)\|_\Pi \leq 1 - \eta_{\mathrm{DCW}} + \sqrt{\frac{(L-d)(1+\lambda_0)}{1-\lambda_0}} \sqrt{\sum_{\ell \in S} \Xi_\ell}.
    \end{align}

    Let $\varkappa := \sum_{\ell \in S} \Xi_\ell$. Combining \eqref{eq:thm42eq1} and \eqref{eq:thm42fin}, and using $\|D(L)\|_\Pi^2 \le \|D(L)\|_\Pi$ by contraction property, we have
    \begin{align*}
        \|D(L)\|_\Pi^2 \leq \min\left\{ 1 - \varkappa, \; 1 - \eta_{\mathrm{DCW}} + \mathcal{A} \sqrt{\varkappa} \right\}, \, \text{where } \mathcal{A} := \sqrt{\frac{(L-d)(1+\lambda_0)}{1-\lambda_0}}.
    \end{align*}
    Define $\varkappa_\star$ by $1 - \varkappa_\star = 1 - \eta_{\mathrm{DCW}} + \mathcal{A} \sqrt{\varkappa_\star}$. Equivalently,
    \begin{align*}
        \sqrt{\varkappa_\star}
        = \frac{\sqrt{\mathcal{A}^2+4\eta_{\mathrm{DCW}}}-\mathcal{A}}{2}
        = \frac{2\eta_{\mathrm{DCW}}}{\sqrt{\mathcal{A}^2+4\eta_{\mathrm{DCW}}}+\mathcal{A}}.    
    \end{align*}
    Since $1 - \varkappa$ is decreasing in $\varkappa$ and $1 - \eta_{\mathrm{DCW}} + \mathcal{A}  \sqrt{\varkappa}$ is increasing in $\varkappa$, we have
    \begin{align*}
        \min\left\{ 1 - \varkappa, 1 - \eta_{\mathrm{DCW}} + \mathcal{A}\sqrt{\varkappa} \right\} \le 1-\varkappa_\star.
    \end{align*}
    Hence $\|D(L)\|_\Pi^2 \le 1-\varkappa_\star$. Since $\eta_{\mathrm{DCW}}\le 1$,
    \begin{align*}
        \sqrt{\mathcal{A}^2+4\eta_{\mathrm{DCW}}}+\mathcal{A} \le \sqrt{\mathcal{A}^2 + 4} + \mathcal{A} \le 2(\mathcal{A}+1).
    \end{align*}
    Therefore, $\varkappa_\star \ge {\eta_{\mathrm{DCW}}^2}/{(\mathcal{A}+1)^2}$. Moreover,
    \begin{align*}
    (\mathcal{A}+1)^2 \le 2(\mathcal{A}^2+1) = 2\left\{ \frac{(L-d)(1+\lambda_0)}{1-\lambda_0} + 1 \right\} \le 2(L-d+1)\frac{1+\lambda_0}{1-\lambda_0}.
    \end{align*}
    It follows that 
    \begin{align*}
        \|D(L)\|_\Pi^2 \le 1 - \frac{\eta_{\mathrm{DCW}}^2}{(\mathcal{A}+1) ^2} 
        \le 1 - \frac{1-\lambda_0}{2(L-d+1)(1+\lambda_0)}\eta_{\mathrm{DCW}}^2.
    \end{align*}

    Since our $f$ is arbitrary with $\|f - \Pi f\|_\Pi = 1$, we get an upper bound of the square of the operator norm
    \begin{align*}
    \|K_{\theta(L)} \dots K_{\theta(1)} - \Pi\|^2 \le 1 - \frac{1-\lambda_0}{2(L-d+1)(1+\lambda_0)} \eta_{\mathrm{DCW}}^2.
    \end{align*}
    Finally, using $\sqrt{1-x} \le 1-x/2$ for $x \in [0,1]$,
    \begin{align*}
    \|K_{\theta(L)} \dots K_{\theta(1)} - \Pi\| \le 1 - \frac{1-\lambda_0}{4(L-d+1)(1+\lambda_0)} \eta_{\mathrm{DCW}}^2.
    \end{align*}
\end{proof}

\begin{proof}[Proof of Theorem~\ref{thm:clt_1}]
    Since both $(P_{\mathrm{RCW}})^d$ and $P_{\mathrm{DCW}}$ have a positive spectral gap, it suffices to prove the central limit theorem for a kernel $K$ that admits a spectral gap.

    Suppose $\|K-\Pi\|\le 1-\epsilon$ for some $\epsilon>0$. Since $\Pi$ is an invariant stationary distribution of $K$, we have $K^n-\Pi=(K-\Pi)^n$, and hence
    \begin{align*}
        \|K^n-\Pi\| \le \|K-\Pi\|^n \le (1-\epsilon)^n.
    \end{align*}
    Therefore, for any $m > 0$, $f\in L^2(\Pi)$ with $\Pi f = 0$,
    \begin{align*}
        \|K^m f\|_{\Pi} = \|(K^m-\Pi)f\|_{\Pi}
        \le \|K^m-\Pi\|\,\|f\|_\Pi \le (1-\epsilon)^m\|f\|_\Pi,
    \end{align*}
    and thus $\sum_{m=0}^\infty \|K^m f\|_{\Pi}<\infty$. Then, it holds that
    \begin{align*}
        n^{-1/2}\sum_{m=0}^{n-1} f(X_m) \xRightarrow{\,\mathbb P_\Pi\,} \mathcal N(0,\sigma_\Pi^2(f)),
    \end{align*}
    where
    \begin{align*}
        \sigma_\Pi^2(f) = \Pi(f^2) + 2\sum_{m=1}^\infty \langle f, K^m f\rangle_\Pi
    \end{align*}
    \cite[Theorem~21.2.6]{DoucMoulinesPriouret2018}. Applying this with $K=P_{\mathrm{DCW}}$ and $K=(P_{\mathrm{RCW}})^d$ yields the claim. Note that if we have an additional global block-wise contraction condition, Theorems~\ref{thm:RCW_to_DCW} and~\ref{thm:DCW_to_RCW} imply that $P_{\mathrm{DCW}}$ has a positive $L^2(\Pi)$ spectral gap if and only if $P_{\mathrm{RCW}}$ does, equivalently if and only if $(P_{\mathrm{RCW}})^d$ does.
\end{proof}

\section{Direct proof from random-scan Gibbs chain to deterministic-scan component-wise chain} \label{Appendix_C}

\begin{theorem}\label{thm:RSG_to_DCW}
Assume $\|K_j-P_j\|\le \lambda_0<1$ for all $j\in[d]$. Suppose $\|P_{\mathrm{RSG}}-\Pi\|\le 1-\eta_{\mathrm{RSG}}$ for some $\eta_{\mathrm{RSG}}\in(0,1]$. Then,
\begin{align*}
    \|P_{\mathrm{DCW}} - \Pi \| & \le 1 - \frac{1-\lambda_0}{2(1+\lambda_0)} \, \left(\underset{j \in [d]}{\min} \frac{(1-\lambda_j)^2}{j} \right) ( 1 - \sqrt{1-\eta_{\mathrm{RSG}}})^2.
\end{align*}
\end{theorem}

The goal of this section is to prove Theorem~\ref{thm:RSG_to_DCW}, which shows directly that a positive spectral gap for the random-scan Gibbs sampler implies a positive spectral gap for the deterministic-scan component-wise chain. The argument in this section follows a structure similar to that of previous results \cite[Lemma 3.4]{ChlebickaLatuszynskiMiasojedow2025} \cite[Theorem 4.1]{BGM2011}. Our approach connects the spectral theory of random-scan Gibbs samplers with the analysis of deterministic component-wise updates. We use standard Hilbert space geometry for orthogonal projections, in particular the method of alternating projections. We first introduce the geometric quantities that measure angles between subspaces, which will be used in the subsequent perturbation argument.

Recall that our block-wise Gibbs kernel $P_j$ is a projection onto the subspace of $x_{-j}$-measurable functions, and that we defined the subspaces $\mathcal{M}_j = \operatorname{ran}(P_j)$. We define the generalized Friedrichs angle $c(\mathcal{M}_1,\dots,\mathcal{M}_d)$ and the auxiliary quantity $l(\mathcal{M}_1,\dots,\mathcal{M}_d)$ by
\begin{align*}
    c(\mathcal{M}_1,\dots,\mathcal{M}_d)
    &=
    \sup\left\{
    \frac{\sum_{i\neq j} \langle f_j,f_i\rangle}{(d-1)\sum_{i=1}^d \langle f_i,f_i\rangle}
    :\ f_i\in \mathcal{M}_i,\ \Pi f_i = 0,\ \sum_{i=1}^d \Pi f_i^2 > 0
    \right\}, \\
    l(\mathcal{M}_1,\dots,\mathcal{M}_d)
    &:= \inf_{f \in L^2_0(\Pi)\,:\,\operatorname{dist}(f,\mathcal{M})=1}
    \max_{1 \le i \le d} \operatorname{dist}(f,\mathcal{M}_i),
\end{align*}
where $\mathcal{M} := \bigcap_{i=1}^d \mathcal{M}_i$, and the distances are given by $\operatorname{dist}(f, \mathcal{M}_{j}) = \|f - P_{j}f \|$ and $\operatorname{dist}(f,\mathcal{M}) = \|f - \Pi f\|_\Pi$. The generalized Friedrichs angle $c(\mathcal{M}_1,\dots,\mathcal{M}_d)$ measures the worst-case average correlation between vectors from the different subspaces, where values close to $0$ correspond to nearly orthogonal subspaces and values close to $1$ indicate strong alignment. The quantity $l(\mathcal{M}_1,\dots,\mathcal{M}_d)$ quantifies the orthogonality of the subspaces. Intuitively, a larger value of $l(\mathcal{M}_1,\dots,\mathcal{M}_d)$ indicates that the subspaces are farther from being parallel, which facilitates faster convergence. The following result links these geometric quantities to the spectral gap of the random-scan Gibbs sampler. Lemma~\ref{lem:Badea-Grivaux-Muller} links the $L^2(\Pi)$ contraction of $P_{\mathrm{RSG}}$ to the lower bound on $l(\mathcal{M}_1,\dots,\mathcal{M}_d)$, which quantifies the common intersection of the subspaces.

\begin{lemma}[Propositions 3.7 and Proof of Proposition 3.9~\cite{BGM2011}]\label{lem:Badea-Grivaux-Muller}
    Let $d \ge 2$. Then,
    \begin{align*}
        \|P_{\mathrm{RSG}} - \Pi \| = \frac{d-1}{d} \left[c(\mathcal{M}_1, \dots, \mathcal{M}_d) + \frac{1}{d-1} \right].
    \end{align*}
    Further, set $l := l (\mathcal{M}_1, \dots, \mathcal{M}_d)$. Then, it holds that
    \begin{align*}
        \frac{d (1-l)^2-1}{d-1}\le c(\mathcal{M}_1, \dots, \mathcal{M}_d) \le \frac{d (1-(l^2 / 2d)^2)-1}{d-1}.
    \end{align*}
\end{lemma}

We are ready to prove Theorem~\ref{thm:RSG_to_DCW}.

\begin{proof}[Proof of Theorem~\ref{thm:RSG_to_DCW}]
By definition, $\max_{j \in [d]} \operatorname{dist}(f,\mathcal{M}_j) \ge l \,\|f - \Pi f\|_\Pi$ for all $f \in L^2(\Pi)$. As each $P_j$ is a projection from Lemma~\ref{lem:Pj_Kj_properties}, $(f - P_j f)$ is orthogonal to $P_j(f - \Pi f - u_{j-1} (f))$ and
\begin{align}
    \|f - \Pi f - P_{j}u_{j-1}(f) \|_\Pi^2 & = \|f - P_{j}f + P_{j}f - \Pi f -  P_{j}u_{j-1}(f) \|_\Pi^2 \notag \\
    & = \|f - P_{j}f \|_\Pi^2 + \| P_{j}f - \Pi f - P_{j} u_{j-1}(f) \|_\Pi^2\notag\\
    & \geq \|f - P_jf \|_\Pi^2 \label{thm21eq1}.
\end{align}
Also, the triangle inequality and Lemma~\ref{lem:upperbound_nonprojection} imply that
\begin{align}
     \|f - \Pi f- K_{j} u_{j-1}(f)\|_\Pi & \geq \|f - \Pi f- P_{j} u_{j-1}(f)\|_\Pi - \|K_{j} u_{j-1}(f) - P_{j} u_{j-1}(f) \|_\Pi \notag \\
     & \geq \|f - \Pi f - P_{j} u_{j-1}(f)\|_\Pi - \frac{\lambda_j}{1-\lambda_j} \| u_{j-1}(f) - K_{j}u_{j-1}(f)\|_\Pi \label{thm21eq2}.
\end{align}
Then, it holds that
 \begin{align*}
 \operatorname{dist}(f, \mathcal M_{j})^2 & =  \|f - P_{j}f \|_\Pi^2\\
    & \leq \|f - \Pi f - P_{j}u_{j-1}(f) \|_\Pi^2\\
    & \leq \left(\|f - \Pi f - K_{j} u_{j-1}(f)\|_\Pi +\frac{\lambda_j}{1-\lambda_j} \| u_{j-1}(f) - K_{j}u_{j-1}(f)\|_\Pi  \right)^2 \\
    & = \left(\|f - \Pi f - u_j(f)\|_\Pi + \frac{\lambda_j}{1-\lambda_j} \| u_{j-1}(f) - u_j (f)\|_\Pi \right)^2,
\end{align*}
where the first inequality follows from Equation~\eqref{thm21eq1} and second inequality follows from Equation~\eqref{thm21eq2}. By the triangle inequality, 
\begin{align*}
    \|f -\Pi f - u_j(f)\|_\Pi & \leq \|u_0(f) - u_1(f)\|_\Pi + \cdots + \| u_{j-1}(f) - u_{j}(f)\|_\Pi.
\end{align*}
Now,
\begin{align*}            
    & \operatorname{dist}(f, \mathcal  M_{j})^2 \\
    & \leq\left( \|u_0(f) - u_1(f)\|_\Pi + \cdots  + \| u_{j-2}(f) - u_{j-1}(f)\|_\Pi + (1+\frac{\lambda_j}{1-\lambda_j}) \| u_{j-1}(f) - u_{j}(f)\|_\Pi \right)^2 \\
    & \leq \left(\sum_{i=1}^j \frac{1+\lambda_i}{1-\lambda_i}\right) \Bigl\{ \frac{1-\lambda_1}{1+\lambda_1}\|u_0(f) - u_1(f)\|_\Pi^2 + \cdots  +  \frac{1-\lambda_{j-1}}{1+\lambda_{j-1}} \| u_{j-2}(f) - u_{j-1}(f)\|_\Pi^2 \\
    & \hspace{8em} + \frac{1}{(1-\lambda_j)(1+\lambda_j)} \| u_{j-1}(f) - u_{j}(f)\|_\Pi^2 \Bigr\}\\
    & \le \left(\sum_{i=1}^j \frac{1+\lambda_i}{1-\lambda_i}\right) \Bigl\{ \| u_0(f) \|_\Pi^2 - \| u_1(f) \|_\Pi^2  + \cdots + \| u_{j-2}(f) \|_\Pi^2 - \| u_{j-1}(f) \|_\Pi^2 \\
    & \hspace{8em} + \frac{1}{(1-\lambda_j)^2}\left(\| u_{j-1}(f) \|_\Pi^2 - \| u_{j}(f) \|_\Pi^2 \right)\Bigr\} \\
    &= \left(\sum_{i=1}^j \frac{1+\lambda_i}{1-\lambda_i}\right) \left\{ \|u_0(f) \|_\Pi^2 - \| u_j(f)\|_\Pi^2 + \left(\frac{1}{(1-\lambda_j)^2} - 1 \right) \left(\| u_{j-1}(f) \|_\Pi^2 - \| u_{j}(f) \|_\Pi^2 \right) \right\},
 \end{align*}    
where the second inequality follows from the Cauchy–Schwarz inequality, and the third inequality follows from Lemma~\ref{lem:spectral_ineq}. As $u_i(f) = K_i u_{i-1}(f)$ and $\|K_i\| \le 1$ for all $i\in [d]$, $\|f - \Pi f \|_\Pi = \| u_0(f)\|_\Pi \ge \| u_1(f)\|_\Pi  \ge \dots \ge \| u_{j}(f)\|_\Pi \ge \dots \ge \|u_d(f) \|_\Pi = \| P_{\mathrm{DCW}}f - \Pi f\|_\Pi$. Therefore,
\begin{align*}            
     \operatorname{dist}(f, \mathcal M_{j})^2
    & \le \left(\sum_{i=1}^j \frac{1+\lambda_i}{1-\lambda_i}\right) \Biggl\{ \|f - \Pi f \|_\Pi^2 - \| P_{\mathrm{DCW}}f - \Pi f\|_\Pi^2 \, + \\
    & \hspace{8em} \left(\frac{1}{(1-\lambda_j)^2} - 1 \right) \left(\| f - \Pi f \|_\Pi^2 - \| P_{\mathrm{DCW}}f - \Pi f\|_\Pi^2 \right) \Biggr\}\\
    & = \left(\sum_{i=1}^j \frac{1+\lambda_i}{1-\lambda_i}\right) \frac{1}{(1-\lambda_j)^2} \left\{\|f - \Pi f \|_\Pi^2 - \| P_{\mathrm{DCW}}f - \Pi f\|_\Pi^2 \right\}
\end{align*}

As $\lambda_i \le \lambda_0 < 1$ for any $i \in [d]$, it follows that
\begin{align*}
    \operatorname{dist}(f, \mathcal{M}_{j})^2 &  \le  \frac{j(1+\lambda_0)}{(1-\lambda_0)(1-\lambda_j)^2} \left\{\|f - \Pi f \|_\Pi^2 - \| P_{\mathrm{DCW}}f - \Pi f\|_\Pi^2 \right\}.
\end{align*}
Then,
\begin{align*}
    l^2 \|f - \Pi f\|_\Pi^2 & \leq \underset{j \in [d]}{\max} \;\operatorname{dist}(f, \mathcal{M}_j)^2 \\
    & \leq \left(\underset{j \in [d]}{\max} \frac{j}{(1-\lambda_j)^2} \right) \frac{(1+\lambda_0)}{(1-\lambda_0)}\left\{\|f - \Pi f \|_\Pi^2 - \| P_{\mathrm{DCW}}f - \Pi f\|_\Pi^2 \right\}
\end{align*}
and the rearrangement yields
\begin{align*}
    \|P_{\mathrm{DCW}}f - \Pi f \|_\Pi^2 
    & \leq \left\{1 - \frac{1-\lambda_0}{1+\lambda_0} \, \left(\underset{j \in [d]}{\min} \frac{(1-\lambda_j)^2}{j} \right) l^2 \right\} \|f - \Pi f \|_\Pi^2 \\
    & \leq \left\{1 - \frac{1-\lambda_0}{1+\lambda_0} \, \left(\underset{j \in [d]}{\min} \frac{(1-\lambda_j)^2}{j} \right) l^2 \right\}\|f\|_\Pi^2.
\end{align*}
As $\sqrt{1-x} \le 1-0.5x$ for all $x \in (0, 1)$,
\begin{align}\label{thm31eq1}   
    \|P_{\mathrm{DCW}} - \Pi \| & \le \sqrt{1 - \frac{1-\lambda_0}{1+\lambda_0} \, \left(\underset{j \in [d]}{\min} \frac{(1-\lambda_j)^2}{j} \right) l^2}\notag \\
    & \le 1 - \frac{1-\lambda_0}{2(1+\lambda_0)} \, \left(\underset{j \in [d]}{\min} \frac{(1-\lambda_j)^2}{j} \right) l^2.
\end{align}
Finally, using Lemma~\ref{lem:Badea-Grivaux-Muller} with $\| P_{\mathrm{RSG}} - \Pi\| \le 1 - \eta_{\mathrm{RSG}}$, we get $1 - c(\mathcal{M}_1, \dots, \mathcal{M}_d) \ge \{d/(d-1)\} \,\eta_{\mathrm{RSG}} $ and
\begin{align*}
    & (1-l)^2  \le \frac{1}{d}\left\{(d-1) c(\mathcal{M}_1, \dots, \mathcal{M}_d) + 1\right\} \\
    \Leftrightarrow \; &  l \;  \ge 1 - \sqrt{\frac{1}{d}\left\{(d-1) c(\mathcal{M}_1, \dots, \mathcal{M}_d) + 1\right\}} \\
    \Leftrightarrow \; & l^2  \; \ge  ( 1 - \sqrt{1-\eta_{\mathrm{RSG}}})^2
\end{align*}
We can bound \eqref{thm31eq1} in terms of $\eta_{RSG}$:
\begin{align*}
    \|P_{\mathrm{DCW}} - \Pi \| & \le 1 - \frac{1-\lambda_0}{2(1+\lambda_0)} \, \left(\underset{j \in [d]}{\min} \frac{(1-\lambda_j)^2}{j} \right) ( 1 - \sqrt{1-\eta_{\mathrm{RSG}}})^2.
\end{align*}
\end{proof}
\begin{remark}
In Theorem~\ref{thm:RSG_to_DCW}, we have
\begin{align*}
\min_{j \in [d]} \frac{(1-\lambda_j)^2}{j}
\ge \frac{(1-\lambda_0)^2}{d}.
\end{align*}
Therefore, Theorem~\ref{thm:RSG_to_DCW} yields the bound
\begin{align*}
\|P_{\mathrm{DCW}} - \Pi \|
\le 1 - \frac{(1-\lambda_0)^3}{2d(1+\lambda_0)}
\left(1-\sqrt{1-\eta_{\mathrm{RSG}}}\right)^2.
\end{align*}
This bound has a similar form to the bound in Corollary~\ref{cor:RSG_to_DCW_2}. In some parameter regimes, however, it can be sharper than the bound in Corollary~\ref{cor:RSG_to_DCW_2}, for example when $d=2$ and $\eta_{\mathrm{RSG}}$ is close to $1/d$. Moreover, if the values $\lambda_j$ vary substantially across $j$ so that $\min_{j \in [d]} (1-\lambda_j)^2/j$ is much larger than $(1-\lambda_0)^2/d$, then the bound in Theorem~\ref{thm:RSG_to_DCW} can also be sharper.
\end{remark}

\section{Proofs for Section~\ref{sec:MALA example}}\label{Appendix_D}

Let $\Gamma$ be a probability measure on $(\mathbb R^N,\mathcal B(\mathbb R^N))$ with density $\gamma$ with respect to Lebesgue measure. Let $S(x,\mathrm{d}x')$ be a Metropolis--Hastings kernel targeting $\Gamma$ with proposal density $q(x,x')$ and acceptance probability
\begin{align*}
\alpha(x,x') = \min\left\{1,\frac{\gamma(x')\,q(x',x)}{\gamma(x)\,q(x,x')}\right\}.
\end{align*}
Let $T:\mathbb R^N\to\mathbb R^N$ be a bijective and differentiable transformation with differentiable inverse $T^{-1}$. Let $\tilde \Gamma=\Gamma\circ T^{-1}$ be the push-forward of $\Gamma$. By the change-of-variables formula, its density satisfies
\begin{align*}
\tilde \gamma(y) = \gamma(T^{-1}(y))\,\left|\det J_T(T^{-1}(y))\right|^{-1},
\end{align*}
where $J_T(x)$ is the Jacobian matrix of $T$ at $x$. Consider the transformed proposal density
\begin{align*}
\tilde q(y,y') = q(T^{-1}(y),T^{-1}(y'))\,\left|\det J_T(T^{-1}(y'))\right|^{-1}.
\end{align*}
Let $\tilde S$ be the Metropolis--Hastings kernel targeting $\tilde \Gamma$ with proposal density $\tilde q$ and acceptance probability
\begin{align*}
\tilde \alpha(y,y') = \min\left\{ 1, \frac{\tilde \gamma(y')\,\tilde q(y',y)}{\tilde \gamma(y)\,\tilde q(y,y')} \right\}.
\end{align*}
Write $\Gamma$ also as an expectation operator on $L^2(\Gamma)$, that is,
\begin{align*}
(\Gamma f)(x)=\int_{\mathbb R^N} f(u)\,\Gamma(\mathrm{d}u),
\end{align*}
and similarly write $\tilde \Gamma$ for the expectation operator on $L^2(\tilde \Gamma)$. Then the following change-of-variables lemma holds. Related variable-transformation arguments appear in previous work \cite[Appendix A]{johnson2012variable}.

\begin{lemma}\label{lem:MH_change_of_variables}
Let $U:L^2(\Gamma)\to L^2(\tilde \Gamma)$ defined by $Uf=f\circ T^{-1}$ for $f \in L^2(\Gamma)$. Then, $\tilde S = USU^{-1}$ and
$\|S-\Gamma\|_{L^2(\Gamma)} = \|\tilde S- \tilde \Gamma \|_{L^2(\tilde \Gamma)}$.
\end{lemma}

\begin{proof}[Proof of Lemma~\ref{lem:MH_change_of_variables}]
Since $T$ is bijective, $U$ has inverse $U^{-1}g=g\circ T$ for $g\in L^2(\tilde \Gamma)$. Moreover,
\begin{align*}
\|Uf\|_{L^2(\tilde \Gamma)}^2
=\int_{\mathbb R^N} |f(T^{-1}(y))|^2\,\tilde \Gamma(\mathrm{d}y)
=\int_{\mathbb R^N} |f(x)|^2\,\Gamma(\mathrm{d}x)
=\|f\|_{L^2(\Gamma)}^2,
\end{align*}
so $U$ is unitary. 
Fix $g\in L^2(\tilde \Gamma)$. Then $\Gamma U^{-1}g$ is the constant function
\begin{align*}
(\Gamma U^{-1}g)(x)=\int_{\mathbb R^N} g(T(u))\,\Gamma(\mathrm{d}u)
=\int_{\mathbb R^N} g(v)\,\tilde \Gamma(\mathrm{d}v),
\end{align*}
so applying $U$ for $\Gamma U^{-1}g$ also gives the same constant function in $L^2(\tilde \Gamma)$. Hence $\tilde \Gamma =U \Gamma  U^{-1}$.

It remains to show that $USU^{-1}=\tilde S$. Fix $g\in L^2(\tilde \Gamma)$ and $y\in\mathbb R^N$. Since $(USU^{-1}g)(y)=(SU^{-1}g)(T^{-1}(y))$, we have
\begin{align*}
(USU^{-1}g)(y)
&=\int_{\mathbb R^N} g(T(x'))\,q(T^{-1}(y),x')\,\alpha(T^{-1}(y),x')\,\mathrm{d}x' \\
&\quad+\left(1-\int_{\mathbb R^N} q(T^{-1}(y),x')\,\alpha(T^{-1}(y),x')\,\mathrm{d}x'\right)\,g(y).
\end{align*}
Applying the change of variables $y'=T(x')$ gives
\begin{align*}
(USU^{-1}g)(y)
&=\int_{\mathbb R^N} g(y')\,q(T^{-1}(y),T^{-1}(y'))\,\alpha(T^{-1}(y),T^{-1}(y'))\,\left|J_T(T^{-1}(y'))\right|^{-1}\,\mathrm{d}y' \\
&\hspace{-2em}+\left(1-\int_{\mathbb R^N} q(T^{-1}(y),T^{-1}(y'))\,\alpha(T^{-1}(y),T^{-1}(y'))\,\left|J_T(T^{-1}(y'))\right|^{-1}\,\mathrm{d}y'\right)\,g(y).
\end{align*}
Therefore $USU^{-1}$ is a Markov operator with proposal density
\begin{align*}
\tilde q(y,y')=q(T^{-1}(y),T^{-1}(y'))\,\left|\det J_T(T^{-1}(y'))\right|^{-1}
\end{align*}
and acceptance probability $\alpha(T^{-1}(y),T^{-1}(y'))$.

We now identify this acceptance probability with $\tilde \alpha(y,y')$. By definition of the push-forward density, $\tilde \gamma(y)=\gamma(T^{-1}(y))\,\left|J_T(T^{-1}(y))\right|^{-1}.$
Using the definitions of $\tilde \gamma$ and $\tilde q$, we obtain
\begin{align*}
\frac{\tilde \gamma(y')\,\tilde q(y',y)}{\tilde \gamma(y)\,\tilde q(y,y')}
=\frac{\gamma(T^{-1}(y'))\,q(T^{-1}(y'),T^{-1}(y))}{\gamma(T^{-1}(y))\,q(T^{-1}(y),T^{-1}(y'))}.
\end{align*}
Hence
\begin{align*}
\alpha(T^{-1}(y),T^{-1}(y')) \hspace{-0.1em}
&=  \min\left\{1,\frac{\gamma(T^{-1}(y'))\,q(T^{-1}(y'),T^{-1}(y))}{\gamma(T^{-1}(y))\,q(T^{-1}(y),T^{-1}(y'))}\right\}\\
&= \min\left\{1,\frac{\tilde \gamma(y')\,\tilde q(y',y)}{\tilde \gamma(y)\,\tilde q(y,y')}\right\}
 = \tilde \alpha(y,y').
\end{align*}
Therefore $USU^{-1}$ coincides with the Metropolis--Hastings operator $\tilde S$. Since $U$ is unitary, it follows that
\begin{align*}
\|S-\Gamma\|=\|U(S-\Gamma)U^{-1}\|=\|USU^{-1}-U\Gamma U^{-1}\|=\|\tilde S-\tilde \Gamma\|.
\end{align*}
\end{proof}
Now consider the deterministic-scan component-wise MALA chain targeting the Gaussian distribution introduced in Section~\ref{sec:gaussian_block_MALA}. To bound $\|K_j-P_j\|$, we use Lemma~\ref{lem:MH_change_of_variables} to transform the conditional MALA kernel $K_{j,x_{-j}}$.

\begin{lemma}\label{lem:norm_eq}
Let $\gamma_j$ be the standard Gaussian measure on $\mathbb{R}^{N_j}$ and let $\Pi_{\gamma_j}:L^2(\gamma_j)\to L^2(\gamma_j)$ denote the expectation operator on $\gamma_j$, that is, $(\Pi_{\gamma_j}f)(x)=\int f(u)\,\gamma_j(\mathrm{d}u)$. Define $\Delta_j=h_jQ_{j,j}$ and let $M_{\Delta_j}$ be the Metropolis--Hastings kernel targeting $\gamma_j$ with proposal $Y\mid X \sim \mathcal N\left((I_{N_j}-\Delta_j)X,2\Delta_j\right)$ and acceptance probability given by the Metropolis--Hastings ratio. Then, for $\Pi_{-j}$-almost every $x_{-j}$,
\begin{align*}
\|K_{j,x_{-j}} - P_{j,x_{-j}}\|_{j,x_{-j}} = \|M_{\Delta_j} - \Pi_{\gamma_j}\|,
\end{align*}
where the right-hand side is the $L^2(\gamma_j)$ operator norm. 
Consequently, $\|K_j-P_j\| \le \|M_{\Delta_j} - \Pi_{\gamma_j}\|$ by Lemma~\ref{lem:block_contraction_equiv}.
\end{lemma}

\begin{proof}
Fix $x_{-j}$. Define the affine bijection $T_{j,x_{-j}}:\mathbb R^{N_j}\to\mathbb R^{N_j}$ by
\begin{align*}
T_{j,x_{-j}}(v)=Q_{j,j}^{1/2}(v-m_j),
\qquad
T_{j,x_{-j}}^{-1}(u)=m_j+Q_{j,j}^{-1/2}u,
\end{align*}
where $v, u\in\mathbb R^{N_j}$. Then $\gamma_j=\Pi_j(\cdot\mid x_{-j})\circ T_{j,x_{-j}}^{-1}$. Moreover, since $\Pi=\mathcal N(\mu,Q^{-1})$, we have $\nabla_j \log \pi(x) = -Q_{j,j}(x_j-m_j)$ for all $x$. Hence, under the local MALA kernel $K_{j,x_{-j}}$, if the current state is $x$, then the proposed state $y_j$ is distributed as
\begin{align*}
y_j\mid x \sim \mathcal N\left(x_j-h_jQ_{j,j}(x_j-m_j),\,2h_j I_{N_j}\right).
\end{align*}
Let $X=T_{j,x_{-j}}(x_j)$ and $Y=T_{j,x_{-j}}(y_j)$. Since $T_{j,x_{-j}}$ is affine, $Y\mid X$ follows a Gaussian distribution. Using the proposal for $y_j\mid x$,
\begin{align*}
\mathbb E[Y\mid X]
&= Q_{j,j}^{1/2}\left(\mathbb E[y_j\mid x]-m_j\right)
= (I_{N_j}-\Delta_j)X,\\
\operatorname{Var}(Y\mid X)
&= Q_{j,j}^{1/2}(2h_j I_{N_j})Q_{j,j}^{1/2}
= 2\Delta_j,
\end{align*}
so $Y\mid X \sim \mathcal N\left((I_{N_j}-\Delta_j)X,\,2\Delta_j\right)$. 
Now we apply Lemma~\ref{lem:MH_change_of_variables} with $S=K_{j,x_{-j}}$ and $T=T_{j,x_{-j}}$, and let $U:L^2\left(\Pi_j(\cdot\mid x_{-j})\right)\to L^2(\gamma_j)$ be the associated unitary operator. By Lemma~\ref{lem:MH_change_of_variables}, $USU^{-1}$ is a Metropolis--Hastings kernel on $L^2(\gamma_j)$ targeting $\gamma_j$, whose proposal is the push-forward of the proposal of $K_{j,x_{-j}}$ under $T_{j,x_{-j}}$. Since we have shown that this push-forward proposal satisfies $Y\mid X \sim \mathcal N\left((I_{N_j}-\Delta_j)X,\,2\Delta_j\right)$, it follows by the defining characterization of $M_{\Delta_j}$ that $USU^{-1}=M_{\Delta_j}$. Therefore, again by Lemma~\ref{lem:MH_change_of_variables},
\begin{align*}
\|K_{j,x_{-j}}-P_{j,x_{-j}}\|_{j,x_{-j}}=\|M_{\Delta_j}-\Pi_{\gamma_j}\|.
\end{align*}
This holds for almost every $x_{-j}$. Thus, by Lemma~\ref{lem:block_contraction_equiv}, $\|K_j-P_j\| \le \|M_{\Delta_j} - \Pi_{\gamma_j}\|$.
\end{proof}

Note that $M_{\Delta_j}$ is also a MALA kernel with step size $\Delta_j$. We simplify the formula for the acceptance probability $\tilde \alpha(X, Y)$ of $M_{\Delta_j}$ as follows.
\begin{align*}
\log\frac{\gamma_j(Y)}{\gamma_j(X)} & =-\frac12\left(Y^\top Y-X^\top X\right),\\
\log\frac{\tilde r_{\Delta_j}(Y,X)}{\tilde r_{\Delta_j}(X,Y)}
&=-\frac14\Bigl[(X-(I_{N_j}-\Delta_j)Y)^\top\Delta_j^{-1}(X-(I_{N_j}-\Delta_j)Y)\\
& \hspace{4em} -(Y-(I_{N_j}-\Delta_j)X)^\top\Delta_j^{-1}(Y-(I_{N_j}-\Delta_j)X)\Bigr]\\
& =-\frac14\left[X^\top(2I_{N_j}-\Delta_j)X-Y^\top(2I_{N_j}-\Delta_j)Y\right]\\
&=\frac12\left(Y^\top Y-X^\top X\right)+\frac14\left(X^\top \Delta_j X-Y^\top \Delta_j Y\right),
\end{align*}
where the second equality follows by expanding the quadratic forms and simplifying. Therefore,
\begin{align*}
\tilde\alpha(X,Y)
&=\min\left\{1,\exp\left(\frac14\left(X^\top\Delta_j X-Y^\top\Delta_j Y\right)\right)\right\}\\
&=\min\left\{1,\exp\left(\frac{h_j}{4}\left(X^\top Q_{j,j}X-Y^\top Q_{j,j}Y\right)\right)\right\}.
\end{align*}
Hence, to determine the block-wise contraction rate $\lambda_j$ for each $j$, it suffices to lower bound $1-\|M_{\Delta_j}-\Pi_{\gamma_j}\|$. To lower bound the spectral gap of $M_{\Delta_j}$, we use the notion of conductance. Define the conductance of Markov kernel $K$ by
\begin{align*}
\Phi_K
:= \inf_{S \in \mathcal B,\ \Pi(S)\in(0,1)}
\frac{\int_S K(x,S^c)\,\Pi(\mathrm{d}x)}{\Pi(S)\Pi(S^c)}
\end{align*}
\cite{helmberg2008introduction}. 
Further, define the right spectral gap by
\begin{align*}
G(K)
:= 1 - \sup\left\{ 
\frac{ \langle f,(K - \Pi)f\rangle_\Pi}{\|f\|_\Pi^2}
:\ f\in L^2(\Pi),\ f\neq 0
\right\}.
\end{align*}
If $K - \Pi$ is self-adjoint on $L^2(\Pi)$, then
\begin{align*}
    \|K- \Pi \| = \max \left\{ \sup_{f \in L^2(\Pi), f \neq 0} \frac{\langle f, (K- \Pi) f \rangle_\Pi}{\|f\|_\Pi^2} , \, - \inf_{f \in L^2(\Pi), f \neq 0} \frac{\langle f, (K- \Pi) f \rangle_\Pi}{\|f\|_\Pi^2} \right\}
\end{align*}
\cite[Corollary 5.1]{helmberg2008introduction}. Moreover, if $K-\Pi$ is positive semidefinite on $L^2(\Pi)$, meaning that $\langle f,(K-\Pi)f\rangle_\Pi\ge 0$ for all $f\in L^2(\Pi)$, then the infimum term is nonnegative. Consequently,
\begin{align*}
G(K) = 1-\sup_{f \in L^2(\Pi), f \neq 0} \frac{\langle f, (K- \Pi) f \rangle_\Pi}{\|f\|_\Pi^2} = 1 - \|K - \Pi\|.
\end{align*}
The following Cheeger's inequality relates conductance to the right spectral gap.
\begin{lemma}{(\cite[Theorem 2.1]{lawler1988bounds})} \label{lem:Cheeger_ineq}
Assume that the Markov kernel $K$ is reversible. Then,
\begin{align*}
\frac{\Phi_K^2}{8}\le G(K)\le \Phi_K.
\end{align*}
\end{lemma}

In our MALA setting, we do not assume that $M_{\Delta_j}$ is non-negative. Therefore, we work with the two-step kernel $M_{\Delta_j}^2$. Since $M_{\Delta_j}$ is reversible with respect to $\gamma_j$, the operator $M_{\Delta_j}-\Pi_{\gamma_j}$ is self-adjoint. Moreover, by the fact that $M_{\Delta_j}\Pi_{\gamma_j}=\Pi_{\gamma_j}M_{\Delta_j}=\Pi_{\gamma_j}$, we have $\|M_{\Delta_j}-\Pi_{\gamma_j}\|^2 = \|M_{\Delta_j}^2-\Pi_{\gamma_j}\|$. Then $\|M_{\Delta_j}^2-\Pi_{\gamma_j}\|\in[0,1]$ and it follows that
\begin{align}\label{eq:M_Deltaj_spectral_lowerbound}
1-\|M_{\Delta_j}-\Pi_{\gamma_j}\|
 = 1-\sqrt{\|M_{\Delta_j}^2-\Pi_{\gamma_j}\|} = \frac{1-\|M_{\Delta_j}^2-\Pi_{\gamma_j}\|}{1+\sqrt{\|M_{\Delta_j}^2-\Pi_{\gamma_j}\|}}
\ge \frac{1-\|M_{\Delta_j}^2-\Pi_{\gamma_j}\|}{2}.
\end{align}
Further, the kernel $M_{\Delta_j}^2 - \Pi_{\gamma_j}$ is reversible with respect to $\gamma_j$ and is non-negative because, for any $f\in L^2(\gamma_j)$,
\begin{align*}
\langle f, (M_{\Delta_j}^2 - \Pi_{\gamma_j}) f\rangle_{\gamma_j}
= \langle f, (M_{\Delta_j} - \Pi_{\gamma_j})^2 f\rangle_{\gamma_j} 
=  \langle  (M_{\Delta_j} - \Pi_{\gamma_j})f,  (M_{\Delta_j} - \Pi_{\gamma_j})f\rangle_{\gamma_j}
\ge 0.
\end{align*}
Hence $G(M_{\Delta_j}^2)=1-\|M_{\Delta_j}^2-\Pi_{\gamma_j}\|$, and Lemma~\ref{lem:Cheeger_ineq} applied to $M_{\Delta_j}^2$ yields
\begin{align*}
1-\|M_{\Delta_j}^2-\Pi_{\gamma_j}\|
= G(M_{\Delta_j}^2)
\ge \frac{\Phi(M_{\Delta_j}^2)^2}{8}.
\end{align*}
Combining with~\eqref{eq:M_Deltaj_spectral_lowerbound} this yields
\begin{align*}
1-\|M_{\Delta_j}-\Pi_{\gamma_j}\|
\ge \frac{\Phi(M_{\Delta_j}^2)^2}{16}.
\end{align*}
Then using $\|K_j-P_j\|=\|M_{\Delta_j}-\Pi_{\gamma_j}\|$, we obtain an $L^2$ contraction bound for the block-wise MALA update.

To lower bound the conductance $\Phi(M_{\Delta_j}^2)$, we use a three-set isoperimetric inequality. Related Cheeger-type inequalities of this form can be found in existing literature \cite{lovasz1999hit, belloni2009computational, Dwivedi2019}.

\begin{lemma}{(\cite[Theorem 1.4.6]{qin2024convergence})}\label{lem:typeofisoperimetric}
    Suppose there exist constants $t>0$, $\epsilon\in(0,1)$, and $\kappa > 0$ satisfying the following two conditions with respect to some distance $\mathrm{d}$:
    \begin{enumerate}
        \item \textbf{Close coupling:} For any $x, y \in \mathbb R^{N_j}$, if $\mathrm{d}(x,y) < t$, then $\|M_{\Delta_j}^2(x,\cdot)-M_{\Delta_j}^2 (y,\cdot)\|_{\operatorname{TV}}\le 1-\epsilon$.
        \item \textbf{Three-set isoperimetric inequality:} For any measurable partition $\{S_1, S_2, S_3\}$ of $\mathbb R^{N_j}$ such that $\mathrm{d}(S_1, S_2) \ge t$,
        \begin{align*}
            \gamma_{N_j}(S_3) \ge \kappa t \gamma_{N_j}(S_1) \gamma_{N_j}(S_2).
        \end{align*}
    \end{enumerate}
    Then, for every $a \in (0,1)$,
    \begin{align*}
        \Phi (M_{\Delta_j}^2) \geq \epsilon \min \left\{\frac{1-a}{2}, \frac{a^2 \kappa t}{4} \right\}.
    \end{align*}
\end{lemma}

\noindent
Since MALA consists of a Gaussian proposal followed by a Metropolis step, we first bound the total variation distance between the proposal distributions at two different points $u,v\in\mathbb R^{N_j}$, given by $r_{\Delta_j}(u,\cdot)=\mathcal N\big((I-\Delta_j)u,\,2\Delta_j\big)$ and $r_{\Delta_j}(v,\cdot)=\mathcal N\big((I-\Delta_j)v,\,2\Delta_j\big)$, respectively. For any vector $v\in\mathbb R^k$ and matrix $A$, write $\|v\|_2$ and $\|A\|_2$ for the Euclidean and spectral norms. Recall that $\|\cdot\|$ is the operator norm on $L^2(\Pi)$. 

\begin{lemma}\label{lem:tv-prop-block}
    For all $u,v\in\mathbb R^{N_j}$,
    \begin{align*}
        \left\| r_{\Delta_j}(u,\cdot)- r_{\Delta_j}(v,\cdot)\right\|_{\operatorname{TV}}
        = 2\,\Phi_{\mathcal{N}}\!\Big(\frac{\big\|\Delta_j^{-1/2}(I-\Delta_j)(u-v)\big\|_2}{2\sqrt{2}}\Big)-1,
    \end{align*}
    where $\Phi_{\mathcal{N}}$ is the standard normal cdf. Define $\mathrm{d}_{\Delta_j}(u,v) := \|\Delta_j^{-1/2}(u-v)\|_2$. If $\mathrm{d}_{\Delta_j}(u,v) \le \eta\sqrt{2}$, then
    \begin{align*}
        1 - \left\| r_{\Delta_j}(u,\cdot)- r_{\Delta_j}(v,\cdot)\right\|_{\operatorname{TV}}
        \ge 2\Big\{1-\Phi_{\mathcal{N}}\!\Big(\frac{\|I-\Delta_j\|_2\,\eta}{2}\Big)\Big\}.
    \end{align*}
\end{lemma}

$\mathrm{d}_{\Delta_j}$ continues to serve as the distance metric $d$ in Lemma~\ref{lem:typeofisoperimetric}.

\begin{proof}[Proof of Lemma~\ref{lem:tv-prop-block}]
For $u,v\in\mathbb R^{N_j}$, define the means $m_u := (I-\Delta_j)u,$ and $m_v := (I-\Delta_j)v.$ Then $r_{\Delta_j}(u,\cdot)=\mathcal N(m_u,2\Delta_j)$ and $r_{\Delta_j}(v,\cdot)=\mathcal N(m_v,2\Delta_j)$. For two Gaussian measures with a common covariance matrix $2\Delta_j$, one has
\begin{align}\label{eq:tv-prop-block-1}
\left\|\mathcal N(m_u,\Sigma)-\mathcal N(m_v,\Sigma)\right\|_{\operatorname{TV}}
& = 2\,\Phi_{\mathcal N}\!\Big(\frac{\|(2\Delta_j)^{-1/2}(m_u-m_v)\|_2}{2}\Big)-1.\notag \\
& =
2\,\Phi_{\mathcal{N}}\!\Big(\frac{\big\|\Delta_j^{-1/2}(I-\Delta_j)(u-v)\big\|_2}{2\sqrt{2}}\Big)-1, 
\end{align}
which proves the first equation in the lemma.

For the second claim, note that
\begin{align*}
\left\|\Delta_j^{-1/2}(I-\Delta_j)(u-v)\right\|_2
=
\left\|(I-\Delta_j)\Delta_j^{-1/2}(u-v)\right\|_2
\le
\|I-\Delta_j\|_2\,\|\Delta_j^{-1/2}(u-v)\|_2.
\end{align*}
So $\|\Delta_j^{-1/2}(u-v)\|_2\le \eta\sqrt{2}$ implies
\begin{align*}
\frac{\left\|\Delta_j^{-1/2}(I-\Delta_j)(u-v)\right\|_2}{2\sqrt{2}}
\le
\frac{\|I-\Delta_j\|_2}{2}\eta.
\end{align*}
Since $\Phi_{\mathcal N}$ is increasing, Equation~\eqref{eq:tv-prop-block-1} yields
\begin{align*}
1-\left\|r_{\Delta_j}(u,\cdot)-r_{\Delta_j}(v,\cdot)\right\|_{\operatorname{TV}}
& = 2\Big\{1-\Phi_{\mathcal N}\!\Big(\frac{\big\|\Delta_j^{-1/2}(I-\Delta_j)(u-v)\big\|_2}{2\sqrt{2}}\Big)\Big\}\\
& \ge 2\Big\{1-\Phi_{\mathcal N}\!\Big(\frac{\|I-\Delta_j\|_2\,\eta}{2}\Big)\Big\}.
\end{align*}
\end{proof}

Intuitively, when $\mathrm d_{\Delta_j}(u,v)$ is small, the one-step MALA transitions from $u$ and $v$ still admit a close coupling with a uniform success probability, which also yields a bound for two-step MALA transition. To control the acceptance probability uniformly, we need a tail bound for a shifted Gaussian quadratic form that arises after diagonalizing $\Delta_j$. The next lemma records a Chernoff bound for $\mathcal{T}=\sum_{i=1}^n a_i(Z_i+m_i)^2$ with $Z_i\sim \mathcal N(0,1)$.

\begin{lemma}{\cite[Theorem 3]{gallagher2019improved}}\label{lem:qf_chernoff}
Let $Z_1,\ldots,Z_n$ be iid $\mathcal N(0,1)$. Fix weights $a_i \ge 0$ and shifts $m_i\in\mathbb R$, and define
\begin{align*}
\mathcal{T}=\sum_{i=1}^n a_i(Z_i+m_i)^2,
\qquad
a_{\max}=\max_{1\le i\le n} a_i.
\end{align*}
Then, for any $\lambda\in(0,1/(2a_{\max}))$,
\begin{align*}
\mathbb E[\exp(\lambda \mathcal{T})]
=
\prod_{i=1}^n (1-2\lambda a_i)^{-1/2}
\exp\Big(\frac{\lambda a_i m_i^2}{1-2\lambda a_i}\Big).
\end{align*}
Consequently, for any $q\in\mathbb R$ and $\lambda\in(0,1/(2a_{\max}))$,
\begin{align*}
\mathbb P(\mathcal{T}>q)
\le
\exp\Big(
-\lambda q
-\frac12\sum_{i=1}^n\log(1-2\lambda a_i)
+\sum_{i=1}^n \frac{\lambda a_i m_i^2}{1-2\lambda a_i}
\Big).
\end{align*}
\end{lemma}

\begin{proof}
Fix $\lambda\in(0,1/(2a_{\max}))$. For each $i$, if $a_i=0$ then
$\mathbb E[\exp(\lambda a_i(Z_i+m_i)^2)]=1$. Assume $a_i>0$ and let $Z\sim\mathcal N(0,1)$. Since $\lambda<1/(2a_i)$, we have
$1-2\lambda a_i>0$ and
\begin{align*}
\mathbb E\big[\exp\big(\lambda a_i(Z+m_i)^2\big)\big]
&=
\frac{1}{\sqrt{2\pi}}
\int_{\mathbb R}
\exp\left(\lambda a_i(z+m_i)^2-\frac{z^2}{2}\right)\,{\rm d}z \\
&=
\frac{\exp(\lambda a_i m_i^2)}{\sqrt{2\pi}}
\int_{\mathbb R}
\exp\left(-\frac{1-2\lambda a_i}{2}z^2+2\lambda a_i m_i z\right)\,{\rm d}z.
\end{align*}
Completing the square gives
\begin{align*}
-\frac{1-2\lambda a_i}{2}z^2+2\lambda a_i m_i z
=
-\frac{1-2\lambda a_i}{2}\left(z-\frac{2\lambda a_i m_i}{1-2\lambda a_i}\right)^2
+\frac{2\lambda^2 a_i^2 m_i^2}{1-2\lambda a_i}.
\end{align*}
Therefore,
\begin{align*}
\mathbb E\big[\exp\big(\lambda a_i(Z+m_i)^2\big)\big]
&=
\exp(\lambda a_i m_i^2)\,
\exp\left(\frac{2\lambda^2 a_i^2 m_i^2}{1-2\lambda a_i}\right)\,
\frac{1}{\sqrt{2\pi}}
\int_{\mathbb R}
\exp\left(-\frac{1-2\lambda a_i}{2}w^2\right)\,{\rm d}w \\
&=
(1-2\lambda a_i)^{-1/2}
\exp\left(\frac{\lambda a_i m_i^2}{1-2\lambda a_i}\right).
\end{align*}
By the independence of $(Z_i)_{i=1}^n$,
\begin{align*}
\mathbb E[\exp(\lambda \mathcal{T})]
=
\prod_{i=1}^n (1-2\lambda a_i)^{-1/2}
\exp\left(\frac{\lambda a_i m_i^2}{1-2\lambda a_i}\right).
\end{align*}
For the tail bound, Markov's inequality yields, for any $q\in\mathbb R$,
\begin{align*}
\mathbb P(\mathcal{T}>q)
=
\mathbb P\big(\exp(\lambda \mathcal{T})>\exp(\lambda q)\big)
\le
\exp(-\lambda q)\,\mathbb E[\exp(\lambda \mathcal{T})],
\end{align*}
which gives the stated Chernoff bound after substituting the moment generating function formula above.
\end{proof}
The next lemma bounds the two-step MALA distance under a mild spectral condition on $\Delta_j$. In particular, it verifies the close-coupling condition in Lemma~\ref{lem:typeofisoperimetric}.
\begin{lemma}\label{lem:close_coupling}
Let $(\delta_1,\ldots,\delta_{N_j})$ be the eigenvalues of $\Delta_j$, and let $\delta_{\max}=\max_{i\in[N_j]}\delta_i$.
For $\eta>0$, define
\begin{align*}
\epsilon(\eta,\Delta_j) := \frac{2}{1+\delta_{\max}} \exp\left\{-\frac{\operatorname{tr}(\Delta_j^2)}{2(1-\delta_{\max})}
\right\} 
- 2\Phi_{\mathcal N}\left( \frac{\|I-\Delta_j\|_{\operatorname{op}}}{2}\eta \right).
\end{align*}
Suppose that $\delta_{\max}<1$. If $\mathrm d_{\Delta_j}(u,v)<\eta\sqrt{2}$, then
\begin{align*}
\|M_{\Delta_j}^2(u,\cdot)-M_{\Delta_j}^2(v,\cdot)\|_{\operatorname{TV}}
\le \|M_{\Delta_j}(u,\cdot)-M_{\Delta_j}(v,\cdot)\|_{\operatorname{TV}}
\le 1-\epsilon(\eta,\Delta_j).
\end{align*}
Moreover, if
\begin{align*}
\operatorname{tr}(\Delta_j^2)
< 2(1-\delta_{\max}) \log\left( \frac{2}{1+\delta_{\max}} \right),
\end{align*}
then there exists $\eta>0$ such that $\epsilon(\eta,\Delta_j)>0$.
\end{lemma}

\begin{proof}[Proof of Lemma~\ref{lem:close_coupling}]
The first inequality is the standard contraction of total variation under a Markov kernel. For any probability measures $\mu,\nu$ on $\mathbb R^{N_j}$,
\begin{align*}
\|\mu M_{\Delta_j}-\nu M_{\Delta_j}\|_{\operatorname{TV}} \le \|\mu-\nu\|_{\operatorname{TV}}.
\end{align*}
Applying this with $\mu=M_{\Delta_j}(u,\cdot)$ and $\nu=M_{\Delta_j}(v,\cdot)$ yields
\begin{align*}
\|M_{\Delta_j}^2(u,\cdot)-M_{\Delta_j}^2(v,\cdot)\|_{\operatorname{TV}}
\le
\|M_{\Delta_j}(u,\cdot)-M_{\Delta_j}(v,\cdot)\|_{\operatorname{TV}}.
\end{align*}
We now bound $\|M_{\Delta_j}(u,\cdot)-M_{\Delta_j}(v,\cdot)\|_{\operatorname{TV}}$ via coupling.
For any coupling $(U,V)$ with marginals $U\sim M_{\Delta_j}(u,\cdot)$ and $V\sim M_{\Delta_j}(v,\cdot)$, the coupling inequality gives
\begin{align*}
\big\|M_{\Delta_j}(u,\cdot)-M_{\Delta_j}(v,\cdot)\big\|_{\operatorname{TV}}\le \mathbb P(U\neq V).
\end{align*}
We now construct a coupling $(U, V)$ as follows. Let $s(y)=\min\{r_{\Delta_j}(u,y),r_{\Delta_j}(v,y)\}$ and $p=\int s(y)\,\mathrm{d}y$.
Construct an overlap coupling $(Y_u,Y_v)$ of $r_{\Delta_j}(u,\cdot)$ and $r_{\Delta_j}(v,\cdot)$ so that $\mathbb P(Y_u=Y_v)=p$. Then $p=1-\big\|r_{\Delta_j}(u,\cdot)-r_{\Delta_j}(v,\cdot)\big\|_{\operatorname{TV}}.$
Moreover, by Lemma~\ref{lem:tv-prop-block} and the assumption that $\mathrm d_{\Delta_j}(u,v)< \eta\sqrt{2}$,
\begin{align*}
p\ge 2\Big\{1-\Phi_{\mathcal N}\!\Big(\frac{\|I-\Delta_j\|_2\,\eta}{2}\Big)\Big\}.
\end{align*}
Note that $p\to 1$ as $\eta\to 0$. Now draw $U_0\sim{\rm Unif}(0,1)$ independently of $(Y_u,Y_v)$ and set
\begin{align*}
U=\begin{cases}
Y_u,& U_0\le \tilde \alpha(u,Y_u),\\
u,& U_0> \tilde \alpha(u,Y_u),
\end{cases}
\qquad
V=\begin{cases}
Y_v,& U_0\le \tilde \alpha(v,Y_v),\\
v,& U_0> \tilde \alpha(v,Y_v).
\end{cases}
\end{align*}
where $\tilde \alpha(\cdot, \cdot)$ is the Metropolis--Hastings acceptance probability of $M_{\Delta_j}$. Then $U\sim M_{\Delta_j}(u,\cdot)$ and $V\sim M_{\Delta_j}(v,\cdot)$ marginally. Furthermore,
\begin{align*}
\mathbb P(U\neq V)
&\le \mathbb P(Y_u\neq Y_v) + \mathbb P\big(U_0>\tilde \alpha(u,Y_u)\big) + \mathbb P\big(U_0>\tilde \alpha(v,Y_v)\big)\\
&= (1-p)+\big(1-\mathbb E[\tilde \alpha(u,Y_u)]\big)+\big(1-\mathbb E[\tilde \alpha(v,Y_v)]\big).
\end{align*}
It remains to lower bound the average acceptance $g(x)=\mathbb E[\tilde \alpha(x,Y)]$, where $Y \sim r_{\Delta_j}(x,\cdot)$ for $x \in \mathbb R^{N_j}$. Let $\Delta_j=O^{\top}\operatorname{diag}(\delta_1,\ldots,\delta_{N_j})O$ with $O$ orthogonal, and write $x'=Ox$ and $Y'=OY$. Then $Y_i'\mid x' \sim \mathcal N((1-\delta_i)x_i',\,2\delta_i)$ independently for $i\in[N_j]$. Moreover,
\begin{align*}
\tilde \alpha(x,Y)=\min\left\{1,\exp\left(\frac14\sum_{i=1}^{N_j}\delta_i\left((x_i')^2-(Y_i')^2\right)\right)\right\}.
\end{align*}
Define $\mathcal{T}=\sum_{i=1}^{N_j}\delta_i(Y_i')^2$. Then $g(x)=\mathbb E[\min\{1,\exp\{-({\mathcal{T} - \sum_{i=1}^{N_j}\delta_i(x_i')^2})/{4}\}\}]$. Using $\min\{1,e^{-t}\}=\int_0^\infty e^{-s}\mathbf 1\{t\le s\}\,{\rm d}s$,
\begin{align*}
g(x)
=
\int_0^\infty e^{-s}\,\mathbb P\left(\mathcal{T}\le \sum_{i=1}^{N_j}\delta_i(x_i')^2+4s\right)\,{\rm d}s.
\end{align*}
Let $Z_1,\ldots,Z_{N_j}$ be iid $\mathcal N(0,1)$. Then each $Y_i'$ can be viewed as $Y_i'=(1-\delta_i)x_i'+\sqrt{2\delta_i}\,Z_i$, and we can write
\begin{align*}
\mathcal{T}
&=
\sum_{i=1}^{N_j}\delta_i\left((1-\delta_i)x_i'+\sqrt{2\delta_i}\,Z_i\right)^2
=
\sum_{i=1}^{N_j}2\delta_i^2\left(Z_i+m_i\right)^2,
\qquad
m_i=\frac{(1-\delta_i)x_i'}{\sqrt{2\delta_i}}.
\end{align*}
Set $\lambda = 1 / (4\delta_{\max})$ where $\delta_{\max}=\max_{i\in[N_j]}\delta_i$. Then Lemma~\ref{lem:qf_chernoff} with $a_i=2\delta_i^2$ and $q=\sum_{i=1}^{N_j}\delta_i(x_i')^2+4s$ gives $2 a_{\max} \lambda = \delta_{\max} < 1$ and 
\begin{align*}
& \mathbb P\left(\mathcal{T}>\sum_{i=1}^{N_j}\delta_i(x_i')^2+4s\right)
 \\
\le &
\exp\left\{
-\lambda\left(\sum_{i=1}^{N_j}\delta_i(x_i')^2+4s\right)
-\frac12\sum_{i=1}^{N_j}\log(1-4\lambda\delta_i^2)
+\sum_{i=1}^{N_j}\frac{2\lambda\delta_i^2 m_i^2}{1-4\lambda\delta_i^2}
\right\}.
\end{align*}
Since $2\delta_i^2 m_i^2=\delta_i(1-\delta_i)^2(x_i')^2$, the exponent becomes
\begin{align*}
-4\lambda s
+\frac12\sum_{i=1}^{N_j}\log\left(\frac{1}{1-4\lambda\delta_i^2}\right)
+\lambda\sum_{i=1}^{N_j}\delta_i(x_i')^2\left(\frac{(1-\delta_i)^2}{1-4\lambda\delta_i^2}-1\right).
\end{align*}
Further, $4\lambda \delta_i^2 \le 4\lambda (2-\delta_i) \delta_i^2 \le \{(2-\delta_i) / \delta_i\}\delta_i^2  =  (2-\delta_i) \delta_i < 1$. Therefore, it holds that
\begin{align*}
\frac{(1-\delta_i)^2}{1-4\lambda\delta_i^2}-1
= \frac{-2\delta_i+\delta_i^2(1+4\lambda)}{1-4\lambda\delta_i^2}
\le 0
\end{align*}
for each $i\in[N_j]$. Therefore,
\begin{align*}
\mathbb P\left(\mathcal{T}>\sum_{i=1}^{N_j}\delta_i(x_i')^2+4s\right)
\le
\exp\left(B(\lambda)-4\lambda s\right),
\qquad
B(\lambda)=\frac12\sum_{i=1}^{N_j}\log\left(\frac{1}{1-4\lambda\delta_i^2}\right).
\end{align*}
Since probabilities are at most $1$, $\mathbb P\left(\mathcal{T}>\sum_{i=1}^{N_j}\delta_i(x_i')^2+4s\right)\le \min\{1,\exp(B(\lambda)-4\lambda s)\}$.
Substituting into the integral representation gives
\begin{align*}
g(x)
&\ge
\int_0^\infty e^{-s}\left[1-\min\{1,\exp(B(\lambda)-4\lambda s)\}\right]\,{\rm d}s\\
&=
\int_{B(\lambda)/(4\lambda)}^\infty e^{-s}\left[1-\exp(B(\lambda)-4\lambda s)\right]\,{\rm d}s
=
\frac{4\lambda}{1+4\lambda}\exp\left(-\frac{B(\lambda)}{4\lambda}\right).
\end{align*}
With $4\lambda \delta_i^2 = \delta_i^2/\delta_{\max} < 1$,
\begin{align*}
B(\lambda)
=\frac12\sum_{i=1}^{N_j}\log\left(\frac{1}{1-\delta_i^2/\delta_{\max}}\right)
\le
\frac12\sum_{i=1}^{N_j}\frac{\delta_i^2/\delta_{\max}}{1-\delta_i^2/\delta_{\max}}
\le
\frac{{\rm tr}(\Delta_j^2)}{2\delta_{\max}(1-\delta_{\max})}.
\end{align*}
Therefore,
\begin{align*}
g(x)\ge g_\star(\Delta_j) := \frac{1}{1+\delta_{\max}}\exp\left(-\frac{{\rm tr}(\Delta_j^2)}{2(1-\delta_{\max})}\right)
\end{align*}
uniformly in $x$.

Define $\epsilon(\eta,\Delta_j)=p+2g_{\star}(\Delta_j)-2,$
which agrees with the definition in the statement of the lemma. Then
\begin{align*}
\|M_{\Delta_j}(u,\cdot)-M_{\Delta_j}(v,\cdot)\|_{\operatorname{TV}}
\le \mathbb P(U\neq V)
\le 1-\epsilon(\eta,\Delta_j).
\end{align*}
Moreover, if
\begin{align*}
{\rm tr}(\Delta_j^2)
< 2(1-\delta_{\max})
\log\left(\frac{2}{1+\delta_{\max}}\right),
\end{align*}
holds, then $g_{\star}(\Delta_j)>1/2$. Since $p\to 1$ as $\eta\to 0$, we can choose $\eta>0$ small so that $\epsilon(\eta,\Delta_j)>0$. This completes the proof.
\end{proof}

We next verify the three-set isoperimetric condition in Lemma~\ref{lem:typeofisoperimetric}. Related three-set isoperimetric inequalities appear in the literature \cite{lovasz1993random, lovasz1999hit, kannan1996sampling, cousins2014cubic}.

\begin{lemma}\label{thm:gauss-3set}
    Let $S_1,S_2,S_3$ be a partition of $\mathbb{R}^{N_j}$ satisfies $\mathrm{d}_{\Delta_j}(S_1,S_2) \ge t$. Then
    \begin{align*}
        \gamma_{N_j}(S_3) \ge \kappa\,t\,\gamma_{N_j}(S_1)\,\gamma_{N_j}(S_2),
        \qquad \kappa := \frac{4}{\sqrt{2\pi}}\sqrt{\psi_{\min} (\Delta_j)},        
    \end{align*}
    where $\psi_{\min} (\Delta_j)$ is the smallest eigenvalue of $\Delta_j$.
\end{lemma}

\begin{proof}[Proof of Lemma~\ref{thm:gauss-3set}]
Note that $\mathrm{d}_{\Delta_j}(S_1,S_2)\ge t$ implies that $\mathrm{d}(S_1, S_2) \ge \sqrt{\psi_{\min} (\Delta_j)} t$, where $\mathrm{d}(S_1, S_2)$ is the Euclidean distance between two sets $S_1$ and $S_2$. Therefore, it suffices to prove that for $S_1,S_2,S_3$ a partition of $\mathbb{R}^d$, if $\mathrm{d}(S_1, S_2) \ge t_0 := \sqrt{\psi_{\min} (\Delta_j)} \, t$, then $\gamma_{N_j}(S_3)\ \ge\ (4 / \sqrt{2\pi})\,t_0\,\gamma_{N_j}(S_1)\,\gamma_{N_j}(S_2)$. 

Fix $S_1$ with $\gamma_{N_j}(S_1) = a \in (0,1)$. Let $S_1^{t_0}$ denote the Euclidean neighborhood for $S_1$ of order ${t_0}\ge 0$. As $\mathrm{d}(S_1, S_2) \ge {t_0}$ and $\{S_1, S_2, S_3 \}$ is a partition of $\mathbb R^{N_j}$, $S_3 \supseteq S_1^{t_0} \setminus S_1$ and $\gamma_{N_j}(S_3) \ge \gamma_{N_j}(S_1^{t_0}) - \gamma_{N_j}(S_1) = \gamma_{N_j}(S_1^{t_0}) - a$.

Let $\Phi$ be the cumulative distribution function of the standard Gaussian measure on $\mathbb R$ and $\Phi^{-1}$ be its inverse function. Let $\phi$ be the density function of a standard Gaussian on $\mathbb R$. Then for every $t_0 \ge 0$, it holds that
\begin{align}\label{eq:lowerbound_S3_gaussian}
    \gamma_{N_j}(S_3) \ge \gamma_{N_j}(S_1^{t_0}) - a & \ge  \Phi(\Phi^{-1}(a) + {t_0}) - a  = \int_0^{t_0} \phi(\Phi^{-1}(a) + u) du
\end{align}
where the second inequality follows from \cite[Equation 1.6]{ledoux2006isoperimetry} in Theorem 1.3, with $\gamma_{N_j} (S_1) = a$. Then,
\begin{align*}
    \gamma_{N_j}(S_3) & \ge \int_0^{t_0} \phi(\Phi^{-1}(a) + u) du\\
    & \ge \int_0^{t_0} \frac{4}{\sqrt{2\pi}} \Phi(\Phi^{-1}(a) + u)(1-\Phi(\Phi^{-1}(a) + u)) du\\
    & \ge \int_0^{t_0} \frac{4}{\sqrt{2\pi}} \Phi(\Phi^{-1}(a))(1-\Phi(\Phi^{-1}(a) + {t_0})) du\\
    & =  \frac{4}{\sqrt{2\pi}} a (1-\Phi(\Phi^{-1}(a) + {t_0})) {t_0}, 
\end{align*}
where for the second inequality follows from the following Lemma~\ref{lem:Normal PDF to Normal CDF}. For the third inequality we used that $\Phi$ is monotonically increasing and $u\mapsto 1-\Phi(\Phi^{-1}(a)+u)$ is monotonically decreasing on $[0,t_0]$. 

Further, from $\mathrm{d}(S_1, S_2) \ge t_0$ it follows that $S_2 \subset (S_1^{t_0})^c$. Then, $\gamma_{N_j}(S_2) \le 1 - \gamma_{N_j}(S_1^{t_0}) \le 1-\Phi(\Phi^{-1}(a) + {t_0})$ by Equation~\eqref{eq:lowerbound_S3_gaussian}. In conclusion, we get
\begin{align*}
    \gamma_{N_j}(S_3) \ge \frac{4}{\sqrt{2\pi}} {t_0} a (1-\Phi(\Phi^{-1}(a) + {t_0}))  \ge \frac{4}{\sqrt{2\pi}} {t_0} \gamma_{N_j}(S_1) \gamma_{N_j}(S_2).
\end{align*}
\end{proof}

\begin{lemma}\label{lem:Normal PDF to Normal CDF}
    Let $\Phi$ be the distribution function of the standard Gaussian measure on $\mathbb R$, and $\phi$ the density function of a standard Gaussian on $\mathbb R$. For any $x \in \mathbb R$, 
    \begin{align*}
        \phi(x) \ge \frac{4}{\sqrt{2\pi}} \Phi(x) (1-\Phi(x))
    \end{align*}
\end{lemma}

\begin{proof}
    Set
    \begin{align*}
        r(x) = \phi(x) - \frac{4}{\sqrt{2\pi}} \Phi(x) + \frac{4}{\sqrt{2\pi}} (\Phi(x))^2.
    \end{align*}
    Since $\Phi(x) (1- \Phi(x)) = \Phi(-x) (1- \Phi(-x)) $ and $\phi(x) = \phi(-x)$ for $x \in \mathbb R$, $r(x)$ is an even function and it sufficies to show that $r(x) \ge 0$ for $x \ge 0$. Since $\phi'(x) = -x\phi(x)$, 
    \begin{align*}
        r'(x) & = -x \phi(x)- \frac{4}{\sqrt{2\pi}} \phi(x)+ \frac{8}{\sqrt{2\pi}} \phi(x)\Phi(x)\\
        & = \phi(x) \left\{\frac{8}{\sqrt{2\pi}} \Phi(x) - \frac{4}{\sqrt{2\pi}} -x\right\}.
    \end{align*}
    Set
    \begin{align*}
        m(x) = \frac{8}{\sqrt{2\pi}} \Phi(x) - \frac{4}{\sqrt{2\pi}} -x.
    \end{align*}
    Then $m(0) = 0$ and $\lim_{x \to \infty}m(x) = -\infty$. Moreover,
    \begin{align*}
       m'(x) = \frac{8}{\sqrt{2\pi}} \phi(x) - 1 .
    \end{align*}
     Let $s >0$ satisfy $\phi(s) = \sqrt{2\pi} / 8$. Then $m'(x) > 0$ for $x \in (0, s)$ and $m'(x) < 0$ for $(s, \infty)$. In particular, $m(s)>0$, and since $\lim_{x\to\infty} m(x)=-\infty$, there exists $s_0>s$ such that $m(x)>0$ for $x\in(0,s_0)$ and $m(x)<0$ for $x\in(s_0,\infty)$. Since $\phi(x) > 0$, $r'(x) = \phi(x) m(x)$ is also positive over $(0, s_0)$ and negative over $(s_0, \infty)$. Thus, $r(x)$ is minimized at either $x = 0$ or $x = \infty$.

    Finally, $r(0) = \lim_{x \rightarrow \infty} r(x) = 0$. So $r(x) \ge 0$ on $x \in [0, \infty)$, which implies $r(x) \ge 0$ on $x \in \mathbb R$ since $r(x)$ is an even function. We get 
    \begin{align*}
        \phi(x) \ge \frac{4}{\sqrt{2\pi}} \Phi(x) (1-\Phi(x))
    \end{align*}
\end{proof}

Finally, we combine this three-set inequality with the coupling estimates to control boundary flows through $\mathrm d_{\Delta_j}$-neighborhoods, which will be the key step toward the desired conductance lower bound. 

\begin{proof}[Proof of Proposition~\ref{prop:MALA_DCW_gap}]
Note that $\psi_{\min}(\Delta_j)=h_j\psi_{\min}(Q_{j,j})$. By Lemma~\ref{lem:norm_eq}, it suffices to show that
\begin{align*}
\|M_{\Delta_j}-\Pi_{\gamma_j}\|\le 1-c_0 \psi_{\min}(\Delta_j).
\end{align*}
Let $\kappa = 4\sqrt{\psi_{\min}(\Delta_j)} / \sqrt{2\pi}$. Take $\eta = 1/4$ and set $t=\eta\sqrt{2}$.
Set
\begin{align*}
\epsilon=\varepsilon(\eta,\Delta_j)
= \frac{2}{1+\delta_{\max}} \exp\left\{-\frac{\operatorname{tr}(\Delta_j^2)}{2(1-\delta_{\max})} \right\}
- 2 \Phi_N\left(\frac{\|I-\Delta_j\|_{\operatorname{op}}}{2} \eta\right).
\end{align*}
For any $u, v\in \mathbb R^{N_j}$, suppose $\mathrm{d}_{\Delta_j}(u, v) \le t = \eta \sqrt{2}$. Since $\delta_{\max}=\psi_{\max}(\Delta_j)\le 1/2<1$, Lemma~\ref{lem:close_coupling} yields
\begin{align*}
\|M_{\Delta_j}^2(u,\cdot)-M_{\Delta_j}^2(v,\cdot)\|_{\operatorname{TV}}
\le \|M_{\Delta_j}(u,\cdot)-M_{\Delta_j}(v,\cdot)\|_{\operatorname{TV}}
\le 1-\epsilon.
\end{align*}
By assumption,
\begin{align*}
\operatorname{tr}(\Delta_j^2)=h_j^2\operatorname{tr}(Q_{j,j}^2)\le \frac{1}{10} \le \log(10/9),
\end{align*}
and since $\delta_{\max}\le 1/2$, this implies
\begin{align*}
\operatorname{tr}(\Delta_j^2)\le 2(1-\delta_{\max})\log\left(\frac{5/3}{1+\delta_{\max}}\right).
\end{align*}
Hence the first term in the definition of $\epsilon$ is at least $6/5$. Also, since $\|I-\Delta_j\|_{\operatorname{op}}=1-\psi_{\min}(\Delta_j)\le 1$, we have
\begin{align*}
\Phi_{\mathcal{N}}\left(\frac{\|I-\Delta_j\|_{\operatorname{op}}}{2}\eta\right)
=\Phi_{\mathcal{N}}\left(\frac{1-\psi_{\min}(\Delta_j)}{8}\right)
\le \Phi_{\mathcal{N}}\left(\frac{1}{8}\right).
\end{align*}
Therefore,
\begin{align*}
\epsilon \ge \frac{6}{5} - 2 \Phi_N\left(\frac{1}{8}\right) > 0.1,
\end{align*}
which is positive. Thus, by Lemma~\ref{lem:close_coupling}, Close coupling condition in Lemma~\ref{lem:typeofisoperimetric} holds with distance $\mathrm d_{\Delta_j}$.
Combining this with Lemma~\ref{thm:gauss-3set}, which verifies the Three-set isoperimetric inequality condition, we obtain from Lemma~\ref{lem:typeofisoperimetric} that, for every $a\in(0,1)$,
\begin{align*}
\Phi(M^2_{\Delta_j})
\ge \epsilon \min \left\{\frac{1-a}{2}, \frac{a^2 \kappa t}{4} \right\}.
\end{align*}
Take $a_\star=1/\sqrt{2}$. Since 
\begin{align*}
    \psi_{\min}(\Delta_j)\le \psi_{\max}(\Delta_j) =  h_j\psi_{\max}(Q_{j,j})\le \frac12,
\end{align*}
we have $\kappa t=\sqrt{\psi_{\min}(\Delta_j)/\pi}<1$, which implies $(1-a_\star)/2\ge 1/8$ and $a_\star^2\kappa t/4=\kappa t/8\le 1/8$. Hence
\begin{align*}
\Phi(M^2_{\Delta_j}) \ge \epsilon \frac{\kappa t}{8}
\ge \frac{\sqrt{\psi_{\min}(\Delta_j)}}{80\sqrt{\pi}}.
\end{align*}
By Lemma~\ref{lem:Cheeger_ineq},
\begin{align*}
1-\|M^2_{\Delta_j}-\Pi_{\gamma_j}\|
\ge \frac{\Phi(M^2_{\Delta_j})^2}{8}
\ge \frac{\psi_{\min}(\Delta_j)}{51200 \pi}.
\end{align*}
Finally, by~\eqref{eq:M_Deltaj_spectral_lowerbound},
\begin{align*}
1-\|M_{\Delta_j}-\Pi_{\gamma_j}\|
\ge \frac{1-\|M^2_{\Delta_j}-\Pi_{\gamma_j}\|}{2}
\ge \frac{\psi_{\min}(\Delta_j)}{102400\pi},
\end{align*}
which yields $\|M_{\Delta_j}-\Pi_{\gamma_j}\|\le 1-c_0 \psi_{\min}(\Delta_j)$ with $c_0=1/(102400\pi)$. The claimed bound for $\|K_j-P_j\|$ follows from Lemma~\ref{lem:norm_eq}.
\end{proof}

\begin{proof}[Proof of Lemma~\ref{lem:MALA_compound_symmetry_precision}]
Note that precision matrix $Q$ has the form
\begin{align*}
Q=\Sigma^{-1}
=\frac{1}{1-\zeta}I_N-\frac{\zeta}{(1-\zeta)\{1+(N-1)\zeta\}}\,1_N1_N^\top.
\end{align*}
Hence, for each block $j$,
\begin{align*}
Q_{j,j}
=\frac{1}{1-\zeta}I_s-\frac{\zeta}{(1-\zeta)\{1+(N-1)\zeta\}}\,1_s1_s^\top.
\end{align*}
Let $v\in\mathbb R^s$ satisfy $1_s^\top v=0$. Then $1_s1_s^\top v=0$, so
\begin{align*}
Q_{j,j}v=\frac{1}{1-\zeta}v,
\qquad
\Delta_j v=\frac{h_j}{1-\zeta}v.
\end{align*}
Also, $1_s1_s^\top 1_s=s1_s$, hence
\begin{align*}
Q_{j,j}1_s
&=\left(\frac{1}{1-\zeta}-\frac{\zeta s}{(1-\zeta)\{1+(N-1)\zeta\}}\right)1_s
=\frac{1+(N-s-1)\zeta}{(1-\zeta)\{1+(N-1)\zeta\}}\,1_s,\\
\Delta_j 1_s
&=\frac{h_j\{1+(N-s-1)\zeta\}}{(1-\zeta)\{1+(N-1)\zeta\}}\,1_s.
\end{align*}
Let $V=\{v\in\mathbb R^s:1_s^\top v=0\}$. Then $\dim(V)=s-1$ and $\mathbb R^s=V\oplus \operatorname{span}\{1_s\}$. Also, since $\zeta > 0$, it holds that $\{1 + (N-s-1)\zeta \}/\{1+(N-1)\zeta \} \le 1$. Therefore, $\Delta_j$ has eigenvalue $\psi_{\max}(\Delta_j)=h_j/(1-\zeta)$ with multiplicity $s-1$ and
\begin{align*}
\psi_{\min}(\Delta_j)=\frac{h_j\{1+(N-s-1)\zeta\}}{(1-\zeta)\{1+(N-1)\zeta\}}
\end{align*}
with multiplicity $1$. As $h_j = (1-\zeta) / \sqrt{10s}$,
\begin{align*}
\psi_{\max}(\Delta_j)=h_j\psi_{\max}(Q_{j,j})=\frac{1}{\sqrt{10s}}\le \frac{1}{2},
\end{align*}
so $h_j \le 1/\{2\psi_{\max}(Q_{j,j})\}$. Moreover, $\operatorname{tr}(Q_{j,j}^2)\le s\,\psi_{\max}(Q_{j,j})^2=s/(1-\zeta)^2$ implies
\begin{align*}
h_j^2\,\operatorname{tr}(Q_{j,j}^2)\le \frac{(1-\zeta)^2}{10s}\cdot \frac{s}{(1-\zeta)^2}=\frac{1}{10},
\end{align*}
so $h_j^2 \le 1/\{10\,\operatorname{tr}(Q_{j,j}^2)\}$. Therefore the conditions of Proposition~\ref{prop:MALA_DCW_gap} hold for every $j$, and applying it yields the desired bound on $\lambda_0$.
\end{proof}

\begin{proof}[Proof of Proposition~\ref{prop:psiA_compound_symmetry}]
As
\begin{align*}
Q=\frac{1}{1-\zeta}I_N-\frac{\zeta}{(1-\zeta)\left(1-\zeta+\zeta N\right)}\,1_N1_N^\top,
\end{align*}
for any $j, j'\in[d]$ with $j\neq j'$,
\begin{align*}
Q_{jj}=\frac{1}{1-\zeta}I_s-\frac{\zeta}{(1-\zeta)\left(1-\zeta+\zeta N\right)}\,1_s1_s^\top, \quad Q_{jj'}=-\frac{\zeta}{(1-\zeta)\left(1-\zeta+\zeta N\right)}\,1_s1_s^\top.
\end{align*}
Since $D=\operatorname{diag}(Q_{11}^{-1},\dots,Q_{dd}^{-1})$, we  have $(DQ)_{jj}=I_s$ and $(DQ)_{jj'}=Q_{jj}^{-1}Q_{jj'}$. By the Sherman--Morrison formula,
\begin{align*}
Q_{jj}^{-1}=(1-\zeta)I_s+\frac{\zeta(1-\zeta)}{1-\zeta+\zeta(N-s)}\,1_s1_s^\top.
\end{align*}
Hence, for $j\neq j'$,
\begin{align*}
(DQ)_{jj'}& = -\frac{\zeta}{1-\zeta + \zeta N}1_s1_s^\top - \frac{s\zeta^2}{(1+(N-s-1)\zeta)(1-\zeta + \zeta N)}1_s1_s^\top\\
& = -\frac{\zeta}{1+(N-s-1)\zeta}\,1_s1_s^\top. 
\end{align*}
Set $\Upsilon= \zeta/(1+(N-s-1)\zeta) > 0$, so that $(DQ)_{jj'}=-\Upsilon1_s1_s^\top$. 

Consider vectors $x = (x^{(1)}, \dots, x^{(d)})\in\mathbb R^N$ such that, for every block $j\in[d]$, the $j$th block $x^{(j)}\in\mathbb R^s$ satisfies $1_s^\top x^{(j)}=0$. Since $(DQ)_{jj}=I_s$ and each off-diagonal block $(DQ)_{jj'}$ is $-\Upsilon 1_s1_s^\top$, we have $(DQ)_{jj'}x^{(j')}=0$ for all $j\neq j'$, hence $DQ\,x=x$. The set of such vectors has dimension $d(s-1)$, so $1$ is an eigenvalue of $DQ$ with multiplicity $d(s-1)$.

Next consider block-wise constant vectors $y=(c_1 1_s,\dots,c_d 1_s)\in\mathbb R^N$ with $\sum_{j=1}^d c_j=0$. Using $(DQ)_{jj}=I_s$ and $(DQ)_{jj'}=-\Upsilon 1_s1_s^\top$ for $j\neq j'$, the $j$th block of $DQ\,y$ equals
\begin{align*}
c_j 1_s-\Upsilon\sum_{j'\neq j}1_s1_s^\top(c_{j'}1_s)
&=c_j 1_s-\Upsilon s\sum_{j'\neq j}c_{j'}\,1_s \\
&=(1+\Upsilon s)c_j 1_s-\Upsilon s\left(\sum_{j'=1}^d c_{j'}\right)1_s.
\end{align*}
Since $\sum_{j'=1}^d c_{j'}=0$, $DQ\,y=(1+\Upsilon s)y$. As the constraint $\sum_{j=1}^d c_j=0$ defines a $(d-1)$-dimensional subspace of $\mathbb R^d$, the eigenvalue $1+\Upsilon s$ has multiplicity $d-1$.

Finally, take $1_N$. For each block $j$,
\begin{align*}
(DQ)_{j\, \cdot}1_N
=1_s-\Upsilon\sum_{j'\neq j}1_s1_s^\top 1_s
=1_s-\Upsilon s(d-1)\,1_s,
\end{align*}
so $DQ\,1_N=\left(1-(N-s)\Upsilon\right)1_N$.

The three subspaces considered above have dimensions $d(s-1)$, $d-1$, and $1$, and their dimensions sum to $sd=N$. Hence they span $\mathbb R^N$, and the eigenvalues of $DQ$ are $1$, $1+\Upsilon s$, and $1-(N-s)\Upsilon$. Therefore the eigenvalues of $I_N-DQ$ are $0$, $-\Upsilon s$, and $(N-s)\Upsilon$. Since $\Upsilon > 0$, the largest eigenvalue of $I_N-DQ$ would be
\begin{align*}
    \psi_{\max}(I_N-DQ)=(N-s)\Upsilon=\frac{(N-s)\zeta}{1+(N-s-1)\zeta}.
\end{align*}

Finally, using Corollary~\ref{cor:RSG_to_DCW_2} to lower bound the spectral gap for the deterministic-scan component-wise MALA chain $P_{\mathrm{DCW}}$ we get
\begin{align*}
\|P_{\mathrm{DCW}}-\Pi\|
&\le 1-\frac{1}{8(d+1)}\cdot \frac{1-\psi_{\max}(I_N-DQ)}{d}\cdot
\frac{ \left\{c_0\frac{1+(N-s-1)\zeta}{\sqrt{10s}\left(1+(N-1)\zeta\right)}\right\}^{3}}
{2-c_{0}\frac{1+(N-s-1)\zeta}{\sqrt{10s}\left(1+(N-1)\zeta\right)}} \\
& \le 1-\frac{c_{0}^{\,3}}{16(d+1)}\cdot \frac{1-\zeta}{d\left(1+(N-s-1)\zeta\right)}\cdot
{\left\{\frac{1+(N-s-1)\zeta}{\sqrt{10s}\left(1+(N-1)\zeta\right)}\right\}^{3}} \\
&\le 1-\frac{c_{0}^{\,3}}{2^9(d+1)}\cdot \frac{1-\zeta}{d\,s^{3/2}}\cdot
\frac{\left(1+(N-s-1)\zeta\right)^{2}}{\left(1+(N-1)\zeta\right)^{3}} \\
&= 1-\frac{c_{0}^{\,3}(1-\zeta)}{2^9\,N\sqrt{s}(d+1)}\cdot
\frac{\left(1+(N-s-1)\zeta\right)^{2}}{\left(1+(N-1)\zeta\right)^{3}}.
\end{align*}
where $2^9$ comes from $2^9 > 16 \times 10^{3/2}$.
\end{proof}

\begin{proof}[Proof of Lemma~\ref{lem:MALA_AR1_precision}]
For any matrix $\mathcal{S}$, write $\mathcal{S}^{(i,i')}$ for the entry in the $i$th row and the $i'$th column of $\mathcal{S}$. Since $Q_{j,j}\in\mathbb R^{s\times s}$ is a symmetric tridiagonal matrix, for any unit vector $x\in\mathbb R^s$ satisfying $\|x\|_2=1$,
\begin{align*}
x^\top Q_{j,j}x
&=\sum_{i=1}^s Q_{j,j}^{(i,i)} x_i^2 + 2\sum_{i=1}^{s-1} Q_{j,j}^{(i,\,i+1)}x_i x_{i+1}.
\end{align*}
Note that $|(Q_{j,j})_{i,i+1}|=|\varphi|/(1-\varphi^2)$. Using $2|uv|\le u^2+v^2$, for each $i\in\{1,\dots,s-1\}$,
\begin{align*}
  -\frac{|\varphi|}{1-\varphi^2}(x_i^2+x_{i+1}^2)\le Q_{j,j}^{(i,\,i+1)} x_i x_{i+1}
&\le \frac{|\varphi|}{1-\varphi^2}(x_i^2+x_{i+1}^2).
\end{align*}
Summing over $i=1,\dots,s-1$ yields
\begin{align*}
\sum_{i=1}^s \left(Q_{j,j}^{(i,i)}-\frac{c_i|\varphi|}{1-\varphi^2}\right)x_i^2
\le x^\top Q_{j,j}x
\le
\sum_{i=1}^s \left(Q_{j,j}^{(i,i)}+\frac{c_i|\varphi|}{1-\varphi^2}\right)x_i^2,
\end{align*}
where $c_1=c_s=1$ and $c_i=2$ for $i\in\{2,\dots,s-1\}$. Note that the diagonal entries $Q_{j,j}^{(i,i)}$ satisfy
\begin{align*}
Q_{j,j}^{(i,i)}\in\left\{\frac{1}{1-\varphi^2},\,\frac{1+\varphi^2}{1-\varphi^2}\right\},
\qquad i=1,\dots,s,
\end{align*}
where $Q_{j,j}^{(i,i)} = {1}/({1-\varphi^2})$ may hold only if $i \in \{1, s\}$.

In particular, $Q_{j,j}^{(i,i)}\le (1+\varphi^2)/(1-\varphi^2)$ for all $i$.
Therefore, as $\|x\|_2 = 1$,
\begin{align*}
x^\top Q_{j,j}x
&\le \sum_{i=1}^s \left(\frac{1+\varphi^2}{1-\varphi^2}+\frac{2|\varphi|}{1-\varphi^2}\right)x_i^2
= \frac{1+\varphi^2+2|\varphi|}{1-\varphi^2}
= \frac{1+|\varphi|}{1-|\varphi|}.
\end{align*}
For the lower bound, consider $i\in\{2,\dots,s-1\}$. Then, $Q_{j,j}^{(i,i)}=(1+\varphi^2)/(1-\varphi^2)$ and $c_i=2$. Then,
\begin{align*}
Q_{j,j}^{(i,i)}-\frac{c_i|\varphi|}{1-\varphi^2}
=\frac{1+\varphi^2-2|\varphi|}{1-\varphi^2}
=\frac{1-|\varphi|}{1+|\varphi|}.
\end{align*}
If $i\in\{1,s\}$, $c_i=1$ and $Q_{j,j}^{(i,i)}\ge 1/(1-\varphi^2)$, so
\begin{align*}
Q_{j,j}^{(i,i)}-\frac{c_i|\varphi|}{1-\varphi^2}
\ge \frac{1-|\varphi|}{1-\varphi^2}
= \frac{1}{1+|\varphi|}
\ge \frac{1-|\varphi|}{1+|\varphi|}.
\end{align*}
Combining these bounds and using $\|x\|_2=1$ gives
\begin{align*}
x^\top Q_{j,j}x
\ge \sum_{i=1}^s \frac{1-|\varphi|}{1+|\varphi|}x_i^2
= \frac{1-|\varphi|}{1+|\varphi|}.
\end{align*}
Taking the supremum and infimum over $\|x\|_2=1$ yields
\begin{align*}
\psi_{\max}(Q_{j,j}) \le \frac{1+|\varphi|}{1-|\varphi|},
\qquad
\psi_{\min}(Q_{j,j}) \ge \frac{1-|\varphi|}{1+|\varphi|}.
\end{align*}
\end{proof}

\begin{proof}[Proof of Proposition~\ref{prop:psiA_AR1}]
Let us write $D_j = Q_{j,j}^{-1} \in \mathbb R^{s\times s}$ so that $D=\operatorname{diag}(D_1, \dots, D_d)$.
Since $(DQ)_{j,j}=D_jQ_{j,j}=I_s$ for every $j\in[d]$, $(I_N-DQ)_{j,j}=0$. Since $Q$ is block tridiagonal, we also have $(I_N-DQ)_{j,j'}=0$ when $|j-j'|>1$. 
Let $e_1 = (1, 0, \dots, 0)^\top \in \mathbb R^s$ and $e_s = (0, \dots, 0, 1)^\top \in \mathbb R^s$. For $j\in\{1,\dots,d-1\}$,
\begin{align*}
(I_N-DQ)_{j,j+1} & = -D_j Q_{j,j+1} \\
& = \frac{\varphi}{1-\varphi^2}D_je_se_1^\top,
\end{align*}
which is an $s\times s$ matrix with only its first column possibly nonzero. For $j\in\{2,\dots,d\}$,
\begin{align*}
(I_N-DQ)_{j,j-1}=\frac{\varphi}{1-\varphi^2}D_je_1e_s^\top,
\end{align*}
which is an $s\times s$ matrix with only its last column possibly nonzero. Thus, for all $j \in [d]$ and $i \in [s]$, the $i$th row of the $j$th block row $(I_N-DQ)_{j, \cdot}$ has at most two nonzero entries. Recall that, for a matrix $\mathcal{S}$, we write $\mathcal{S}^{(i,i')}$  for its entry in the $i$th row and $i'$th column. Since the eigenvalues of a matrix are always bounded by its norm \cite[Theorem 13.5]{noble1977applied},
\begin{align*}
\psi_{\max}(I_N-DQ)
& \le \|I_N-DQ\|_\infty\\
& \le \max_{j\in[d]}\max_{i\in[s]} \frac{|\varphi|}{1-|\varphi|^2} \left( \left|(D_je_se_1^\top)^{(i,\cdot)}1_s\right| \mathbf  1_{\{j\le d-1\}}  + \left|(D_je_1e_s^\top)^{(i,\cdot)}1_s\right|  \mathbf  1_{\{j\ge 2\}} \right)\\
& = \max_{j\in[d]}\max_{i\in[s]}
\frac{|\varphi|}{1-|\varphi|^2}
\left( \left|D_j^{(i,s)}\right| \mathbf 1_{\{j\le d-1\}} + \left|D_j^{(i,1)}\right| \mathbf 1_{\{j\ge 2\}} \right).
\end{align*}
It remains to bound $D_j^{(i,1)}$ and $D_j^{(i,s)}$ case by case.

Let $j=1$. Then,
\begin{align*}
D_1 = Q_{1,1}^{-1}
=
\begin{pmatrix}
1-\varphi^{2s} & \varphi-\varphi^{2s-1} & \varphi^{2}-\varphi^{2s-2} & \cdots & \varphi^{s-1}-\varphi^{s+1} \\
\varphi-\varphi^{2s-1} & 1-\varphi^{2s-2} & \varphi-\varphi^{2s-3} & \cdots & \varphi^{s-2}-\varphi^{s} \\
\varphi^{2}-\varphi^{2s-2} & \varphi-\varphi^{2s-3} & 1-\varphi^{2s-4} & \cdots & \varphi^{s-3}-\varphi^{s-1} \\
\vdots & \vdots & \vdots & \ddots & \vdots \\
\varphi^{s-1}-\varphi^{s+1} & \varphi^{s-2}-\varphi^{s} & \varphi^{s-3}-\varphi^{s-1} & \cdots & 1-\varphi^{2}
\end{pmatrix}.
\end{align*}
In particular, for $i \in [s]$, $D_1^{(i,s)} = (1-\varphi^2)\varphi^{s-i}$. Therefore, for $j = 1$, 
\begin{align*}
\max_{i\in[s]}
\frac{|\varphi|}{1-|\varphi|^2}
\left( \left|D_j^{(i,s)}\right| \mathbf 1_{\{j\le d-1\}} + \left|D_j^{(i,1)}\right| \mathbf 1_{\{j\ge 2\}} \right) = |\varphi|.
\end{align*}
where the maximum is attained at $i = s$.

Let $j=d$. Then,
\begin{align*}
D_d = Q_{d,d}^{-1}
=
\begin{pmatrix}
1-\varphi^{2} & \varphi-\varphi^{3} & \varphi^{2}-\varphi^{4} & \cdots & \varphi^{s-1}-\varphi^{s+1} \\
\varphi-\varphi^{3} & 1-\varphi^{4} & \varphi-\varphi^{5} & \cdots & \varphi^{s-2}-\varphi^{s+2} \\
\varphi^{2}-\varphi^{4} & \varphi-\varphi^{5} & 1-\varphi^{6} & \cdots & \varphi^{s-3}-\varphi^{s+3} \\
\vdots & \vdots & \vdots & \ddots & \vdots \\
\varphi^{s-1}-\varphi^{s+1} & \varphi^{s-2}-\varphi^{s+2} & \varphi^{s-3}-\varphi^{s+3} & \cdots & 1-\varphi^{2s}
\end{pmatrix}.
\end{align*}
In particular, for $i \in [s]$, $D_d^{(i, 1)} = (1-\varphi^2) \varphi^{i-1}$. Therefore, for $j = d$,
\begin{align*}
\max_{i\in[s]}
\frac{|\varphi|}{1-|\varphi|^2}
\left( \left|D_j^{(i,s)}\right| \mathbf 1_{\{j\le d-1\}} + \left|D_j^{(i,1)}\right| \mathbf 1_{\{j\ge 2\}} \right) = |\varphi|.
\end{align*}
where the maximum is attained at $i = 1$.

For any interior block $j\in\{2,\dots,d-1\}$, we have
\begin{align*}
D_j = Q_{j,j}^{-1}
=
\frac{1}{1-\varphi^{2s+2}}
\begin{pmatrix}
a_{11} & a_{12} & \cdots & a_{1s} \\
a_{21} & a_{22} & \cdots & a_{2s} \\
\vdots & \vdots & \ddots & \vdots \\
a_{s1} & a_{s2} & \cdots & a_{ss}
\end{pmatrix},
\end{align*}
where $a_{lm}
= \varphi^{|l-m|}-\varphi^{l+m}-\varphi^{2s+2-l-m}+\varphi^{2s+2-|l-m|}$ for $l,m\in[s]$. In particular, for $i \in [s]$, 
\begin{align*}
    D_j^{(i,1)} & = a_{i 1} = \frac{\varphi^{i-1} - \varphi^{i+1} - \varphi^{2s + 1 - i} + \varphi^{2s + 3 - i}}{1-\varphi^{2s+2}} = \frac{(1-\varphi^2)(\varphi^{i-1} - \varphi^{2s + 1 - i})}{1-\varphi^{2s+2}}, \\
    D_j^{(i,s)} & = a_{i s} = \frac{\varphi^{s-i} - \varphi^{s+i} - \varphi^{s + 2 - i} + \varphi^{s + 2 + i}}{1-\varphi^{2s+2}} = \frac{(1-\varphi^2)(\varphi^{s-i} - \varphi^{s + i})}{1-\varphi^{2s+2}}.
\end{align*}
Therefore, for $j \in  \{2, \dots, d-1 \}$,
\begin{align*}
& \max_{i\in[s]}
\frac{|\varphi|}{1-|\varphi|^2}
\left( \left|D_j^{(i,s)}\right| \mathbf 1_{\{j\le d-1\}} + \left|D_j^{(i,1)}\right| \mathbf 1_{\{j\ge 2\}} \right) \\
= &  \max_{i\in[s]}\frac{\left|\varphi^i-\varphi^{2s+2-i}\right|+ \left|\varphi^{s-i+1} - \varphi^{s + i+1} \right| }{1-|\varphi|^{2s+2}} \\
= & \max_{i\in[s]}\frac{|\varphi|^i- |\varphi|^{2s+2-i}+ |\varphi|^{s-i+1} - |\varphi|^{s + i+1} }{1-|\varphi|^{2s+2}}.
\end{align*}
For $i\in\{1,\dots,s-1\}$, subtracting the expressions at $i$ and $i+1$ gives
\begin{align*}
&\frac{|\varphi|^i- |\varphi|^{2s+2-i}+ |\varphi|^{s-i+1} - |\varphi|^{s + i+1} }{1-|\varphi|^{2s+2}}
- \frac{|\varphi|^{i+1}- |\varphi|^{2s-i+1}+ |\varphi|^{s-i} - |\varphi|^{s +i+2} }{1-|\varphi|^{2s+2}} \\
&= \frac{(1-|\varphi|)(|\varphi|^{i} - |\varphi|^{s-i} + |\varphi|^{2s-i+1} - |\varphi|^{s +i + 1})}{1-|\varphi|^{2s+2}} \\
& =  \frac{(1-|\varphi|)\left(|\varphi|^{i}-|\varphi|^{s-i}\right)}{1+|\varphi|^{s+1}}.
\end{align*}
This quantity is negative when $2i<s$ and positive when $2i>s$. Therefore, the maximum is attained at $i=1$ and $i=s$. Evaluating at $i=1$ or $i=s$ yields
\begin{align*}
\max_{i\in[s]}
\frac{|\varphi|}{1-\varphi^2}\left(\left|D_j^{(i,1)}\right|+\left|D_j^{(i,s)}\right|\right)
= \frac{|\varphi|-|\varphi|^{2s+1}+|\varphi|^{s}-|\varphi|^{s+2}}{1-|\varphi|^{2s+2}} 
= \frac{|\varphi|+|\varphi|^{s}}{1+|\varphi|^{s+1}} .
\end{align*}
Since $|\varphi| < 1$, $(|\varphi|+|\varphi|^{s})/(1+|\varphi|^{s+1}) > |\varphi|$. Then,
\begin{align*}
\psi_{\max}(I_N-DQ)\le  \max_{j\in[d]}\max_{i\in[s]}
\frac{|\varphi|}{1-|\varphi|^2}
\left( \left|D_j^{(i,s)}\right| \mathbf 1_{\{j\le d-1\}} + \left|D_j^{(i,1)}\right| \mathbf 1_{\{j\ge 2\}} \right) = \frac{|\varphi|+|\varphi|^{s}}{1+|\varphi|^{s+1}}.
\end{align*}

By Lemma~\ref{lem:MALA_AR1_precision},
\begin{align*}
\|P_{\mathrm{DCW}}-\Pi\|
&\le 1-\frac{1}{8(d+1)}\cdot \frac{1-\psi_{\max}(I_N-DQ)}{d}\cdot \frac{(1-\lambda_0)^{3}}{1+\lambda_0} \\
&\le 1-\frac{1}{8(d+1)}\cdot \frac{(1-|\varphi|)\left(1-|\varphi|^{s}\right)}{d\left(1+|\varphi|^{s+1}\right)}\cdot
\frac{ \left\{c_0\frac{(1-|\varphi|)^2}{\sqrt{10s}\left(1+|\varphi|\right)^2}\right\}^{3}}
{2-c_{0}\frac{(1-|\varphi|)^2}{\sqrt{10s}\left(1+|\varphi|\right)^2}} \\
&\le 1-\frac{c_{0}^{\,3}}{16(d+1)}\cdot \frac{(1-|\varphi|)\left(1-|\varphi|^{s}\right)}{d\left(1+|\varphi|^{s+1}\right)}\cdot
{\left\{\frac{(1-|\varphi|)^2}{\sqrt{10s}\left(1+|\varphi|\right)^2}\right\}^{3}} \\
&\le 1-\frac{c_{0}^{\,3}}{2^9(d+1)}\cdot \frac{(1-|\varphi|)\left(1-|\varphi|^{s}\right)}{d\,s^{3/2}\left(1+|\varphi|^{s+1}\right)}\cdot
\frac{(1-|\varphi|)^{6}}{\left(1+|\varphi|\right)^{6}} \\
&= 1-\frac{c_{0}^{\,3}}{2^9\,N\sqrt{s}(d+1)}\cdot
\frac{(1-|\varphi|)^{7}\left(1-|\varphi|^{s}\right)}{\left(1+|\varphi|\right)^{6}\left(1+|\varphi|^{s+1}\right)}.
\end{align*}
Here we used $16\cdot 10^{3/2}<2^9$.
\end{proof}

\begin{remark}
When $d=2$, only the boundary cases $j=1$ and $j=d$ exist. Hence, $\|I_N-DQ\|_\infty = |\varphi|,$
which is bounded above by
${(|\varphi|+|\varphi|^s)}/({1+|\varphi|^{s+1}})$.

When $s=1$, we have
\begin{align*}
I_N-DQ
= \begin{pmatrix}
0 & \varphi & 0 & \cdots & 0 \\
\frac{\varphi}{1+\varphi^2} & 0 & \frac{\varphi}{1+\varphi^2} & \ddots & \vdots \\
0 & \frac{\varphi}{1+\varphi^2} & 0 & \ddots & 0 \\
\vdots & \ddots & \ddots & \ddots & \varphi \\
0 & \cdots & 0 & \varphi & 0
\end{pmatrix}.
\end{align*}
Therefore,
\begin{align*}
\|I_N-DQ\|_\infty
= \begin{cases}
|\varphi|, & d=2,\\
\frac{2|\varphi|}{1+\varphi^2}, & d\ge 3.
\end{cases}
\end{align*}
It follows that
\begin{align*}
\psi_{\max}(I_N-DQ)
\le \|I_N-DQ\|_\infty
\le \frac{2|\varphi|}{1+\varphi^2}
= \frac{|\varphi|+|\varphi|^s}{1+|\varphi|^{s+1}},
\end{align*}
so the same bound remains valid for $s=1$.
\end{remark}

\end{document}

%% file: spectral_gap_table.tex
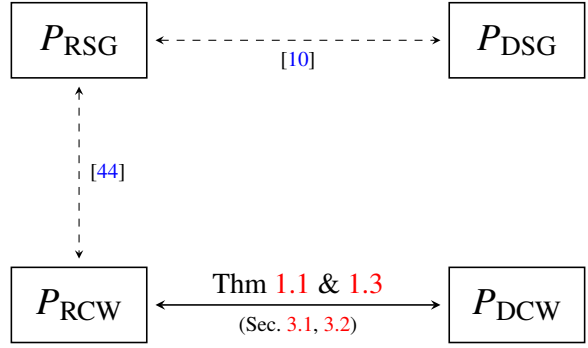
\begin{figure}[t]
\centering
\begin{tikzpicture}[
    node distance=2.5cm and 4.0cm,
    box/.style={
        rectangle, 
        draw=black, 
        semithick, 
        minimum width=1.8cm, 
        minimum height=1.0cm, 
        align=center,
        font=\large
    },
    known/.style={
        <->, 
        >=stealth, 
        shorten >=3pt, shorten <=3pt, 
        dashed, 
        thin
    },
    contribution/.style={
        <->, 
        >=stealth, 
        shorten >=3pt, shorten <=3pt, 
        solid, 
        semithick 
    },
    diag_contribution/.style={
        ->, 
        >=stealth, 
        shorten >=3pt, shorten <=3pt, 
        solid, 
        thin
    }
]

    \node[box] (PR) {$P_{\mathrm{RSG}}$};
    \node[box, right=of PR] (PD) {$P_{\mathrm{DSG}}$};
    \node[box, below=of PR] (SR) {$P_{\mathrm{RCW}}$};
    \node[box, right=of SR] (SD) {$P_{\mathrm{DCW}}$};

    
    \draw[known] (PR) -- node[above, font=\small] {} 
                         node[below, font=\scriptsize] {\citep{ChlebickaLatuszynskiMiasojedow2025}} (PD);

    \draw[known] (PR) -- node[left, align=right, font=\small] {} 
                         node[right, align=left, font=\scriptsize] {\citep{QinJuWang2025}} (SR);

    \draw[contribution] (SR) -- node[above, font=\small] {Thm~\ref{thm:RCW_to_DCW} \& \ref{thm:DCW_to_RCW}} 
                                node[below, font=\scriptsize] {(Sec.~\ref{sec:RCW_to_DCW},~\ref{sec:DCW_to_RCW})} (SD);


\end{tikzpicture}
\caption{
The web of implications for spectral gaps. Dashed arrows indicate known results: the solidarity principle for Gibbs samplers \citep{ChlebickaLatuszynskiMiasojedow2025} and the spectral gap bound for random-scan component-wise samplers \citep{QinJuWang2025}. 
Solid arrows represent our main contributions, established under the global blockwise contraction condition. Specifically, the bottom horizontal arrow shows that the random-scan and deterministic-scan component-wise samplers, $P_{\mathrm{RCW}}$ and $P_{\mathrm{DCW}}$, have positive spectral gaps simultaneously, with polynomial dependence on the dimension $d$. Appendix~C also gives a direct proof from $P_{\mathrm{RSG}}$ to $P_{\mathrm{DCW}}$.
}
\label{fig:implications}
\end{figure}